\documentclass[a4paper,11pt, reqno]{amsart}

\usepackage{amssymb}
\usepackage{hyperref}

\usepackage{float}

\hypersetup{nesting=true,debug=true,naturalnames=true}
\usepackage{graphicx,amssymb,upref}

\let\<\langle
\let\>\rangle
\newcommand{\doi}[1]{doi:\,\href{http://dx.doi.org/#1}{#1}}
\newcommand{\mailurl}[1]{\email{\href{mailto:#1}{#1}}}
\let\uml\"

\newcommand{\mrev}[1]{\href{http://www.ams.org/mathscinet-getitem?mr=#1}{MR#1}}
\newcommand{\zbl}[1]{\href{http://www.emis.de/cgi-bin/MATH-item?#1}{Zbl #1}}

\newcommand{\arx}[1]{\href{http://arXiv.org/abs/#1}{arXiv:#1}}

\usepackage{amsmath, amssymb, graphicx, multicol}
\usepackage{subfigure}
\usepackage{multicol}
\usepackage[font=small, skip=0pt]{caption}

\usepackage{graphicx,amssymb,amsfonts,epsfig,amsthm,a4,amsmath,url}
\usepackage[latin1]{inputenc}
\usepackage[english]{babel}

\usepackage{latexsym}
\usepackage{amsmath,amsthm}
\usepackage{amsfonts}
\usepackage{epsfig}
\usepackage{amssymb}
\usepackage{amscd} 
\usepackage[all]{xy}
\usepackage{stmaryrd}
\usepackage[usenames,dvipsnames]{xcolor} % for colours

\input xy
\usepackage{hyperref}

\usepackage{enumitem}

\newtheorem{Thm}{Theorem}[section]
\newtheorem{Prop}[Thm]{Proposition}
\newtheorem{Lem}[Thm]{Lemma}
\newtheorem{Cor}[Thm]{Corollary}
\newtheorem{Ex}[Thm]{Example}
\theoremstyle{definition}

\newtheorem{Rem}[Thm]{Remark}

\newcommand{\Z}{\mathbb{Z}}

\newcommand{\R}{\mathbb{R}}
\newcommand{\C}{\mathbb{C}}
\newcommand{\Q}{\mathbb{Q}}

\numberwithin{equation}{section}

\newlist{multienum}{enumerate}{1}
\setlist[multienum]{
    label=\alph*),
    before=\begin{multicols}{2},
    after=\end{multicols}
}

\newlist{multiitem}{itemize}{1}
\setlist[multiitem]{
    label=\textbullet,
    before=\begin{multicols}{2},
    after=\end{multicols}
}

\title{$K$-homology and $K$-theory of pure Braid groups}
\author{Sara Azzali}
\address{Sara Azzali\\ Universit\`a degli Studi di Bari, 
Dipartimento di Matematica,
Via E. Orabona 4, 70125 Bari, Italy}
\mailurl{sara.azzali@uniba.it}

\author{Sarah L. Browne}
\address{Sarah L. Browne\\The University of Kansas, Department of Mathematics, 1460 Jayhawk Blvd, Lawrence, KS, 66045.} 
\mailurl{slbrowne@ku.edu}

\author{Maria Paula Gomez Aparicio}
\address{Maria Paula Gomez Aparicio\\Laboratoire de math\'ematiques d'Orsay, Universit\'e Paris-Saclay - CNRS, 91405, Orsay, France}
\mailurl{maria-paula.gomez-aparicio@universite-paris-saclay.fr}

\author{Lauren C. Ruth}
\address{Lauren C. Ruth\\Mercy College, 555 Broadway, Dobbs Ferry, NY 10522, USA}
\mailurl{LRuth@mercy.edu}

\author{Hang Wang}
\address{Hang Wang\\School of Mathematical Sciences, East China Normal University, Shanghai, China 200241}
\mailurl{wanghang@math.ecnu.edu.cn}



\date{\today}

\begin{document}

\baselineskip=16pt

\maketitle

\begin{abstract}
We produce an explicit description of the $K$-theory and $K$-homology of the pure braid group on $n$ strands. We describe the  Baum--Connes correspondence  between the generators of the left- and right-hand sides for $n=4$. 
Using functoriality of the assembly map and direct computations, we recover Oyono-Oyono's result on the Baum--Connes conjecture for pure braid groups \cite{OO2001}.
We also discuss the case of the full braid group on $3$-strands.
\end{abstract}

\medskip

%\setcounter{tocdepth}{1}

%\tableofcontents

\section{Introduction}
Given a locally compact group $G$, the Baum--Connes conjecture predicts a way of computing the $K$-theory of the reduced group $C^*$-algebra of $G$ in terms of the equivariant $K$-homology of $\underline{E}G$, the classifying space for proper actions of $G$. More precisely, let $K_i^{G}{(\underline{E}G)}$ denote the $G$-equivariant $K$-homology of the space $\underline{E}G$ of order $i$ and $K_i(C_r^{\ast} (G))$ is the $K$-theory of the reduced $C^*$-algebra $C_r^{\ast}(G)$ of order $i$; the conjecture, as formulated by Baum, Connes and Higson in \cite{BCH}, states that the assembly map  
\[\mu_i \colon K_i^{G}{(\underline{E}G)} \rightarrow K_i(C_r^{\ast} (G))\]
for $i=0,1$, 
%defined  using Kasparov $KK$-theory, 
is a group isomorphism for all locally compact groups. 

The Baum--Connes conjecture has been proven for large classes of groups, including all semi-simple Lie groups and all groups satisfying Haagerup's property (\cite{Lafforgue}, \cite{Higson-Kasparov}). Many of the proofs are based on methods that use heavy machinery, such as the Dirac-dual Dirac method, introduced by Kasparov in the case of connected Lie groups and further developed by Higson and Kasparov in \cite{Higson-Kasparov} to prove the conjecture for groups having Haagerup's property. 

In the case of semi-simple Lie groups, a first proof was established by Wassermann (\cite{Wassermann}) following the work of   Penington--Plymen (\cite{MR724030}) and  Valette (\cite{MR755672, MR799592}).  This proof was based on the idea of giving a complete description of both sides of the assembly map and then proving explicitly that the correspondence was an isomorphism of groups. 
Indeed, the description of the $K$-theory of the reduced $C^*$-algebra of a semi-simple group can be made using the exhaustive work of Harish-Chandra on the classification of their tempered representations. For discrete groups, as no such classification exists, other approaches were needed and led to the development of very powerful techniques. For an account of the history of the conjecture and the recent developments, we refer to the survey \cite{GJV} and the references therein, as well as to the books \cite{valbc, misval}.

\bigskip

In this paper, we study the Baum--Connes correspondence for the pure braid group on $n$ strands. The conjecture for those groups is known to be true by the work of Oyono-Oyono  \cite{OO2001}.

Our paper fits into the context of the work of Isely \cite{Isley} followed by the works of Flores, Pooya and Valette \cite{Flores-Pooya-Valette,Pooya-Valette,Pooya19}, in which explicit computations of the Baum--Connes correspondence are given for certain discrete groups. We believe that these explicit computations contribute to a deeper understanding of the Baum--Connes correspondence. 

It is important to mention that the conjecture also holds for full braid groups by the work of Schick (\cite{Schick2007}) using permanence properties of the conjecture shown by  Chabert--Echterhoff in \cite{CE01} and the result of Oyono-Oyono for pure braid groups. 
The conjecture holds in its strong form, with coefficients, {\it i.e.} considering the action of the group on a $C^*$-algebra.
Moreover, full braid groups have property RD (see for example \cite{RD-Chatterji}).
Explicit computations for full braid groups are more difficult, though, and other methods have to be used.

Therefore, the aim of this work is to compute the $K$-theory and $K$-homology arising in the Baum--Connes assembly map explicitly for the pure braid group on $n$ strands and then to understand the correspondence of the generators under this map. The case when $n=4$ is worked out explicitly as a typical example. In this case, the classifying space $BP_4$ can be given a model of the form $S^1\times X$, where $X$ is a 2-dimensional $CW$-complex. We can then apply Lemma 4.1 from \cite{misval}, which relates the $K$-homology of $X$ to its integer singular homology, leading us to the following result:  

\begin{Thm}  For the pure braid group $P_4$ the equivariant $K$-homology of $\underline{E}P_4$ is 
$$
K_0^{P_4}(\underline{E}P_4) \simeq \mathbb{Z}^{12} \quad\text{and}\quad K_1^{P_4}(\underline{E}P_4) \simeq \mathbb{Z}^{12}.
$$
\end{Thm}

Matthey proved that the $K$-homology of a CW-complex of dimension $\leq 3$ is isomorphic to its integral homology \cite{Matthey02}; however, for higher number of strands ($n\geq 5$),  the classifying space of $P_n$ admits a model of dimension $n-1$, which is minimal because, by Arnold's result \cite{Arnold}, the classifying space has non vanishing cohomology in degree $n-1$. Hence, one cannot apply  Matthey's results to $BP_n$ when $n\geq 5$ .

In the general case, we proceed as follows. First we deduce the $K$-homology group up to torsion by means of existing results on the group homology of $P_n$. After that, we use an Atiyah--Hirzebruch spectral sequence to remove the torsion. We are then able to extend our first result to pure braid groups on $n$ strands:  

\begin{Thm}  For the pure braid group $P_n$ we have
$$
K_0^{P_n}(\underline{E}P_n) \simeq \mathbb{Z}^{\frac{n!}{2}} \quad\text{and}\quad K_1^{P_n}(\underline{E}P_n) \simeq \mathbb{Z}^{\frac{n!}{2}}.
$$
\end{Thm}

For the right-hand side of the Baum--Connes correspondence, we use the Pimsner--Voiculescu six-term exact sequence in \cite{PV80} and \cite{PV82} to show the following: 

%\begin{Thm}  For the pure braid group $P_4$ we have
%$$K_0(C^*_r (P_4) ) \simeq \mathbb{Z}^{12}, \qquad K_1(C^*_r (P_4) ) \simeq \mathbb{Z}^{12}.$$
%\end{Thm}

\begin{Thm}  For the pure braid group $P_n$ we have
$$K_0(C^*_r (P_n) ) \simeq \mathbb{Z}^{\frac{n!}{2}} \qquad\text{and}\quad K_1(C^*_r (P_n) ) \simeq \mathbb{Z}^{\frac{n!}{2}}.$$
\end{Thm}

Next, using functoriality of the Baum--Connes assembly map, together with explicit computations, we recover Oyono-Oyono's results for pure braid groups: 
\begin{Thm}
The Baum--Connes assembly map $\mu: K_i(BP_n)\rightarrow K_i(C^*_r(P_n))$ for the pure braid group $P_n$ is an isomorphism.
\end{Thm}
We explicitly describe the assembly map on each of the generators in the case of $P_4$ (see Theorem {\ref{thm1}} and Theorem \ref{thm2}). 

All our computations of $K$-theory groups can be carried out explicitly, thanks to the iterated semidirect product structure of pure braid groups:
\[
P_n=F_{n-1}\rtimes F_{n-1}\rtimes\cdots\rtimes F_1.
\]
This also indicates that the rank of the $K$-groups grows as $n$ increases.

The techniques we use for pure braid groups do not apply to full braid groups. Although there is an extension 
\[
1\rightarrow P_n\rightarrow B_n\rightarrow S_n\rightarrow 1 
\]
where we denote by $S_n$ the symmetric group over the set of $n$-elements, that implies that the braid group $B_n$ contains the pure braid group $P_n$ as a normal subgroup of finite index, the $K$-groups of $B_n$ have fewer generators than the $K$-groups for $P_n$. In fact, using an existing result on the group homology of $B_n$(see \cite{Arnold}, \cite{ACC03} and section \ref{K-homology}), one knows that, up to torsion, both the even and odd $K$-homology groups for $BB_n$ are $\Z$. Then the Baum--Connes conjecture says that, up to torsion, the $K$-theory of the reduced $C^*$-algebra of $B_n$ is $\Z$ as well. When $n=3$, $B_3$ has the special structure of a free amalgamated product, which allows us to perform a direct calculation: 
\[
K_0(C^*_r(B_3))=K_1(C^*_r(B_3))\simeq\Z.
\]
For $n=4$ the $K$-theory of $C^*_r(B_4)$  is explicitly computed in the recent paper by Li, Omland, and Spielberg (\cite{MR4218683}). 
To our knowledge, the problem of directly computing  $K$-theory for the full braid group $C^*$-algebra remains open.

The paper is organized as follows. In Section~\ref{Sec2} we recall the structure and properties of braid and pure braid groups (in the appendix we give some of the corresponding diagrams that illustrate the structure of this groups). In Section~\ref{Sec3} we describe the classifying space for $P_4$ explicitly, compute its $K$-homology and generalize to the case of $P_n.$
In Section~\ref{Sec4} we apply the Pimsner--Voiculescu six-term exact sequence to calculate the $K$-theory for the reduced group $C^*$-algebras for $P_n$ with $n=4$ as a typical example. 
In Section~\ref{Isomorphism} we describe the Baum--Connes assembly map on each generator for $P_4$ and show that the map is an isomorphism for all $n.$
In Section~\ref{Sec6} 
%[{\textcolor{red}{to be moved before - probably cannot be moved - but it is separate to $P_n$, so is okay to be at the end..}}] 
we compute the example for $B_3$ on both sides of the assembly map and show that the map is an isomorphism. 
%We end with an appendix  explicit diagrams of  generators and relations of $P_4$.

\subsection*{Acknowledgments} We thank Alain Valette for the suggestion to examine $K$-theory and $K$-homology of pure braid groups. We thank the organisers of the Women in Operator Algebras Conference that took place at BIRS where this project started. HW acknowledges the support from Shanghai Rising-Star Program 19QA1403200 and NSFC-11801178. MGA was partially supported by ANR project Singstar. 

\section{Braid and pure braid groups}
\subsection{Structure of braid and pure braid groups}
\label{Sec2}

Throughout the paper we will denote by $F_{n}(x_1\ldots x_n)$ the free group generated by $x_1,\ldots, x_n$.
Let us recall the definition and some properties of braid groups. We refer to \cite{birm}.

The Artin Braid Group on $n$ letters, denoted by $B_n$, is a finitely-generated group with generators $\sigma_1, \sigma_2,\ldots,\sigma_{n-1}$ that satisfy the following relations: 
\begin{align*}
&\sigma_j\sigma_i=\sigma_i\sigma_j \qquad \qquad &|i-j|>1,\,\quad\,i,j\in\{1,\ldots,n-1\} \\
&\sigma_i\sigma_{i+1}\sigma_i=\sigma_{i+1}\sigma_i\sigma_{i+1}\qquad  \qquad & i\in\{1,\ldots,n-2\}
\end{align*}

%\begin{figure}
%\includegraphics[width=0.7\textwidth]{braidrel.png}
%\end{figure}
It can also be described as the group of equivalence classes of all braids on $n$ strands. The generators are illustrated here for $B_4$.\\

\begin{figure}[!h]

\setlength\columnsep{0.25cm}

\begin{multicols}{3}
\centering
\includegraphics[height=0.13\textheight]{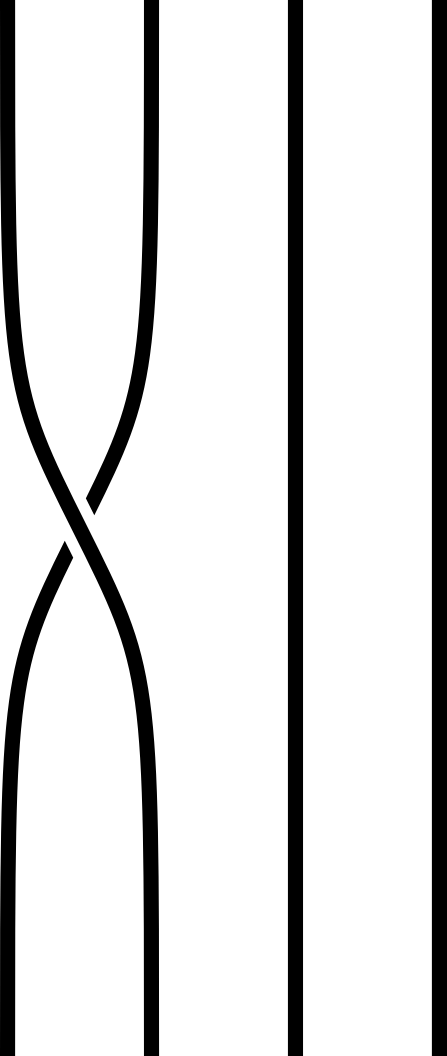}\\
\vspace{0.15cm}
\caption*{$\sigma_1$}
\includegraphics[height=0.13\textheight]{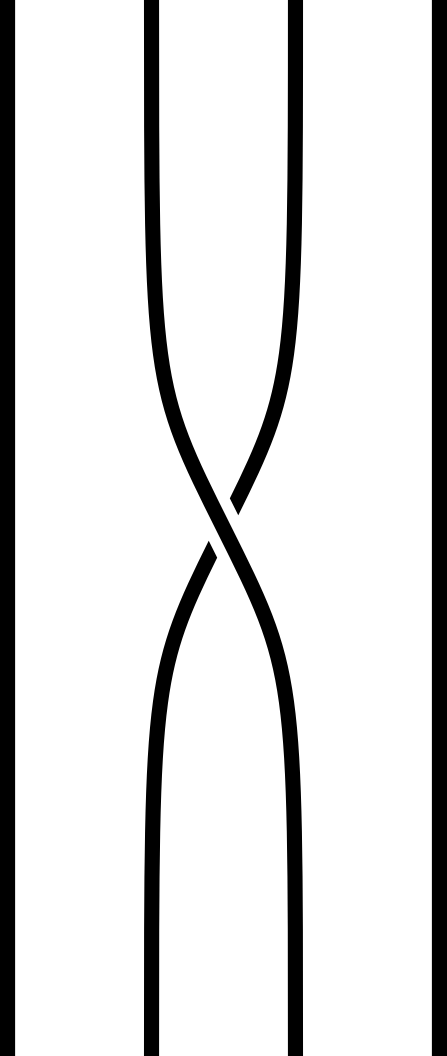}\\
\vspace{0.15cm}
\caption*{$\sigma_2$}
\includegraphics[height=0.13\textheight]{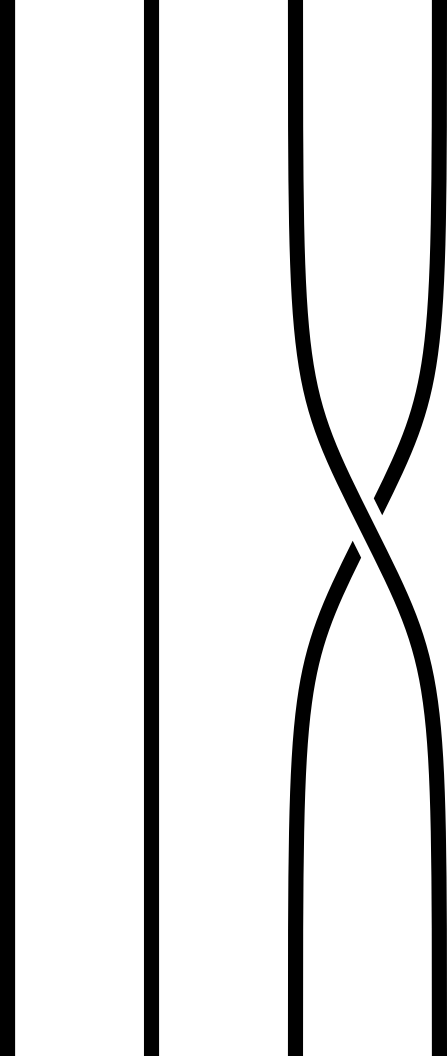}\\
\vspace{0.10cm}
\caption*{$\sigma_3$}
\end{multicols}
\end{figure}

In this framework, composition of two elements is visualized as the concatenation of the corresponding braid pictures. The identity is represented visually by four straight lines.

As every $n$-braid determines a permutation of the set of $n$ elements in an obvious way, it is easy to see that there is a surjective map from $B_n$ to $S_n$, the symmetric group consisting of all permutations of $n$ elements 
\begin{align*}
p: B_n&\to S_n.
\end{align*}
This map is compatible with the structures of the two groups so that it is a morphism of groups. Notice that the image of the element $\sigma_i$ is the permutation exchanging $i$ and $i+1$, hence $p(\sigma_i)=(i,i+1)$, a transposition. 

By definition, the \emph{pure braid group on $n$-strands} is the kernel of $p$ (hence a subgroup of $B_n$ of index $n!$). It is usually denoted by $P_n$ and it is easy to see that in the strand framework it corresponds to the elements of $B_n$ for which all strands start and end at the same point. Notice that as $(i,i+1)$ is a transposition of $S_n$, the element $\sigma^2_i$ belongs to $P_n$ for all $i\in\{1,...,n\}$. 

Starting from $P_n$ we can construct a surjective morphism $$f: P_n\to P_{n-1}$$by forgetting the $n^{th}$ strand whose kernel is known to be isomorphic to the free group on $n-1$ generators; this is easy to understand when viewing braids as configuration spaces. In that context, the kernel of $f$ corresponds to the fundamental group of the space obtained by removing $n-1$ points from the plane $\mathbb{C}$, which is isomorphic to the free group on $n-1$ generators, $F_{n-1}$. We have therefore a short exact sequence  
\[
1\rightarrow F_{n-1}\rightarrow P_n\rightarrow P_{n-1}\rightarrow 1
\]
that is split because it is always possible to add a strand to a braid in $P_{n-1}$ to obtain a braid in $P_n$. Hence $P_n$ is isomorphic to a semi-direct product $F_{n-1}\rtimes P_{n-1}$, and hence isomorphic to an iterated semi-direct product as follows : 
$$P_n\simeq F_{n-1}\rtimes P_{n-1}\simeq F_{n-1}\rtimes F_{n-2}\rtimes\cdots \rtimes F_1.$$

Throughout this paper we will use the following presentation of $P_n$ which is due to Artin (see \cite{birm} Lemma 1.8.2.). Notice that we are conjugating in the reverse order of \cite{birm}, for the sake of compatibility with the diagrams in the appendix, so our presentations appear slightly different from the presentation in \cite{birm}.

The generators of $P_n$ are given by the following formula :

$$A_{ij}=\sigma_{j-1}\sigma_{j-2} \cdots \sigma_{i+1}\sigma_i^2\sigma_{i+1}^{-1} \cdots \sigma_{j-1}^{-1}\quad\text{for} \quad 1\leq i<j\leq n$$
where the $\sigma_i$, for $i=1,\dots,n$ are the generators of $B_n$ given above; they are subject to the following relations
\begin{align*}
A_{rs}A_{ij}A_{rs}^{-1}=\begin{cases}
& A_{ij},\quad\text{if}\quad r<s<i<j\quad\text{and}\quad i<r<s<j\\
&A_{sj}^{-1}A_{ij}A_{sj}, \quad\text{if}\quad r=i\\
&(A_{rj}A_{ij})^{-1}A_{ij}(A_{rj}A_{ij}), \quad\text{if}\quad s=i\\
&(A_{rj}A_{sj})^{-1}(A_{sj}A_{rj})A_{ij}(A_{sj}A_{rj})^{-1}(A_{rj}^{-1}A_{sj}), \quad\text{if}\quad r<i<s<j.
\end{cases}
\end{align*}

Let $\alpha_i=A_{n-i,n}$ for $i=1, \ldots, n-1$. Then the subgroup of $P_n$ isomorphic to $F_{n-1}$ appearing in the decomposition $P_n=F_{n-1}\rtimes P_{n-1}$ is generated by the elements $\alpha_i$, for $i=1, \ldots, n-1$.
%$$A_{in}=\sigma_{n-1}\sigma_{n-2}\cdots\sigma_{i+1}\sigma_i^2\sigma_{i+1}^{-1}\cdots\sigma_{n-2}^{-1}\sigma_{n-1}^{-1},$$
The semidirect product decomposition can be written as 
$$
P_{n}\simeq F_{n-1}\rtimes_{\varphi} P_{n-1}
$$
where the action $\varphi$ of $P_{n-1}$ on $F_{n-1}$ is given by the map 
$$\varphi : P_{n-1} \rightarrow \text{Aut}(F_{n-1})$$
defined by
\begin{align*}
\varphi(A_{rs})(A_{in})  =\begin{cases}
&A_{in},\quad\text{if}\quad r<s<i<n,\quad\text{and}\quad i<r<s<n \\
&A_{sn}^{-1}A_{in}A_{sn}, \quad\text{if}\quad r=i\\
&(A_{rn}A_{in})^{-1}A_{in}(A_{rn}A_{in}) ,\quad\text{if}\quad s=i\\
&(A_{rn}A_{sn})^{-1}(A_{sn}A_{rn})A_{in}(A_{sn}A_{rn})^{-1}(A_{rn}^{-1}A_{sn}) ,\quad\text{if}\quad r<i<s<n.
\end{cases}
\end{align*}
Following this notation, we have that $P_n=F_{n-1}(\alpha_1, \alpha_2, \ldots \alpha_{n-1})\rtimes_{\varphi} P_{n-1}.$ 

%Therefore the pure braid group can be realized as the iterated semidirect product of free groups:
%$$P_{n}\simeq F_{n-1}\rtimes_{\varphi} P_{n-1}\simeq F_{n-1}\rtimes F_{n-2}\rtimes\cdots \rtimes F_1.$$

The center of $B_n$ is generated by the element 
$$(\sigma_1\sigma_2\cdots\sigma_{n-1})^n$$
which can be expressed in terms of elements of $P_n$ by
$$(A_{12})(A_{13}A_{23})\cdots(A_{1n}A_{2n}\cdots A_{(n-1)n}).$$
(This is illustrated in the appendix in the case $n=4$.)

For $n=3$, the generators of $P_3$ are $$
A_{13}=\sigma_2\sigma_1^2\sigma_2^{-1},\quad A_{23}=\sigma^2_2,\quad A_{12}=\sigma_1^2.
$$
Letting $\alpha_2=A_{13}$, $\alpha_1=A_{23}$ and $\sigma_1^{2}=A_{12}$, we get that $P_3$ has the following presentation 
$$P_3=\Big\langle \alpha_1,\alpha_2,\sigma_1^2\quad | \quad \sigma_1^{-2}\alpha_1\sigma_1^{2}=(\alpha_2\alpha_1)\alpha_1(\alpha_2\alpha_1)^{-1},\quad \sigma_1^{-2}\alpha_2\sigma_1^{2}=\alpha_1\alpha_2\alpha_1^{-1}\Big\rangle,$$
whence 
$$P_3\simeq F_2(\alpha_1,\alpha_2)\rtimes \langle \sigma_1^{2}\rangle,
$$
that is $P_3$ is isomorphic to the semi-direct product of the free group generated by $\alpha_1$ and $\alpha_2$ and the group generated by $\sigma_1^{2}$, where the action of $\sigma_1^{2}$ on $F(\alpha_1,\alpha_2)$ is given by conjugation.

Denoting by $c$ the element $\sigma_1^{2}\alpha_1\alpha_2$, we can check that $\alpha_1 c=c\alpha_1$ and $\alpha_2 c = c\alpha_2$ so that 
$$
P_3=F(\alpha_1,\alpha_2)\times\langle c\rangle.
$$
For $n=4$, to simplify notation in the rest of the paper, we will denote the generators of $P_4$ as follows.
%$$A_{12}=\sigma_1^2, \quad 
%A_{13}=\sigma_2\sigma_1^2\sigma_2^{-1}, \quad 
%A_{14}=\sigma_3\sigma_2\sigma_1^2\sigma_2^{-1}\sigma_3^{-1}, \quad 
%A_{23}=\sigma_2^2, A_{24}=\sigma_3\sigma_2^2\sigma_3^{-1}, \quad A_{34}=\sigma_3^2.
%$$
%In this paper we will denote the generators of $P_4$ by 
\begin{align*}
\begin{split}
&\sigma_1^{2}=A_{12} \quad \\
&\alpha_1=A_{23}=\sigma_2^2\quad \\ &\alpha_2=A_{13}=\sigma_2\sigma_1^2\sigma_2^{-1} \quad  \\
\end{split}
\begin{split}
&\beta_1=A_{34}=\sigma_3^2\quad \\
&\beta_2=A_{24}=\sigma_3\sigma_2^2\sigma_3^{-1}\quad \\
&\beta_3=A_{14}=\sigma_3\sigma_2\sigma_1^2\sigma_2^{-1}\sigma_3^{-1}
\end{split}
\end{align*}

\begin{figure}[H] \label{gens}
\setlength\columnsep{-1cm}

\begin{multicols}{3}
\centering
\includegraphics[height=0.13\textheight]{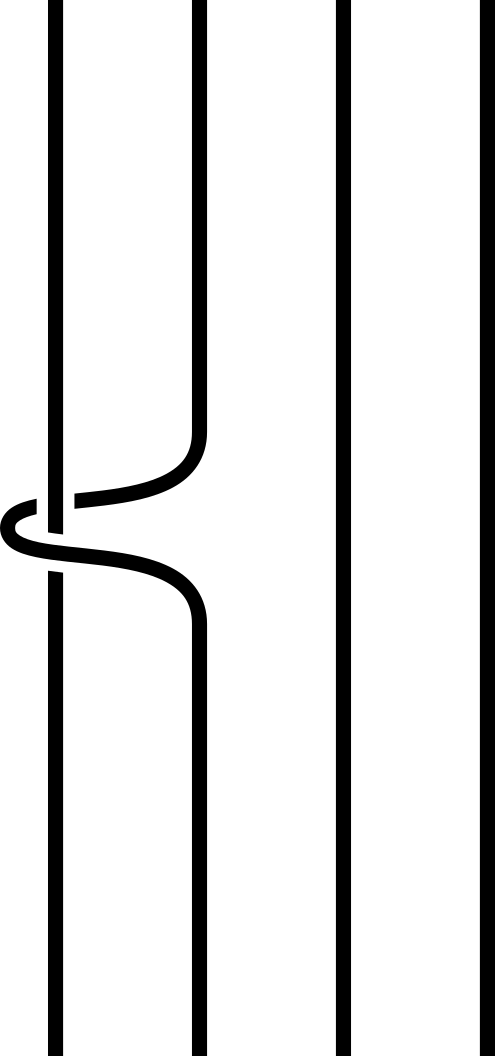}\\
\vspace{0.15cm}
\caption*{$\sigma_1^2$}
\includegraphics[height=0.13\textheight]{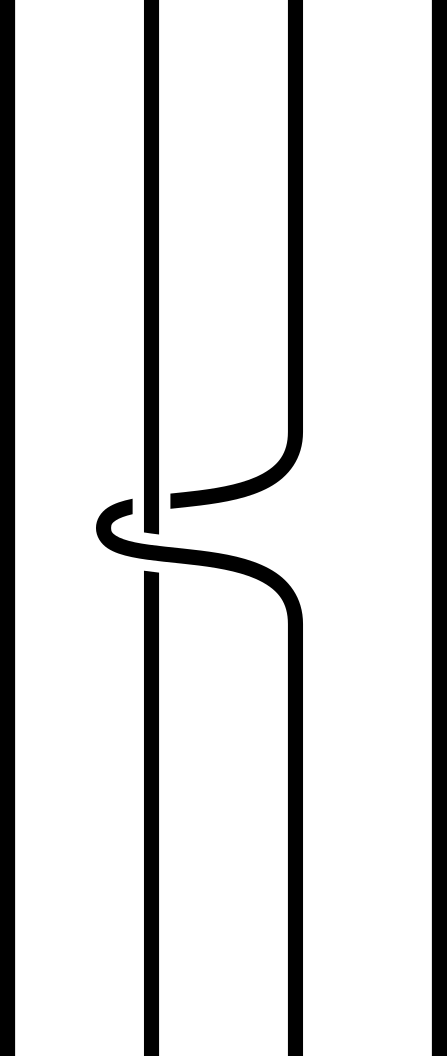}\\
\vspace{0.15cm}
\caption*{$\alpha_1=\sigma_2^2$}
\includegraphics[height=0.13\textheight]{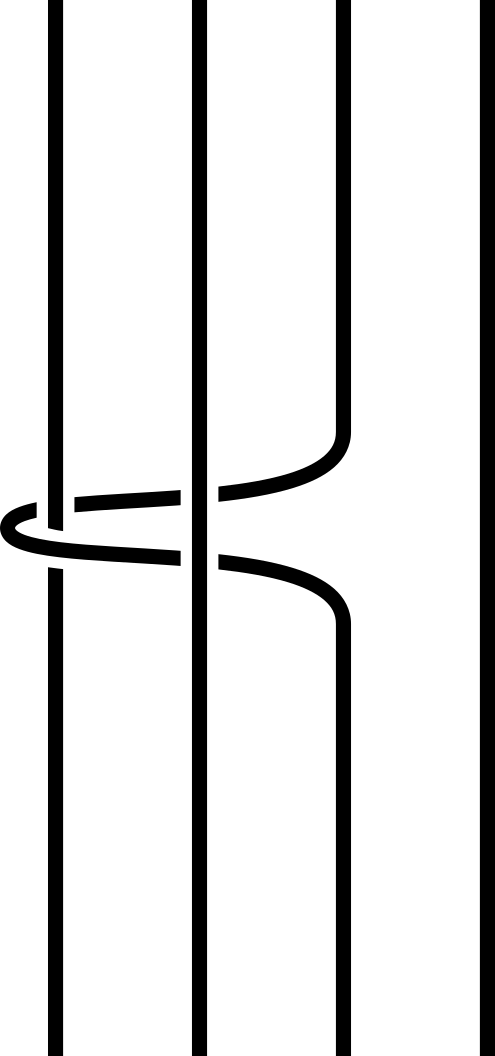}\\
\vspace{0.15cm}
\caption*{$\alpha_2=\sigma_2 \sigma_1^2 \sigma_2^{-1}$}
\end{multicols}

%\end{figure}
%very annoying how latex thinks these .pngs have different widths.  should edit the .png metadata, but for now, this "height" instead of "width" specification works OK

%\begin{figure}[!h]
\setlength\columnsep{-1cm}
\begin{multicols}{3}
\centering
\includegraphics[height=0.13\textheight]{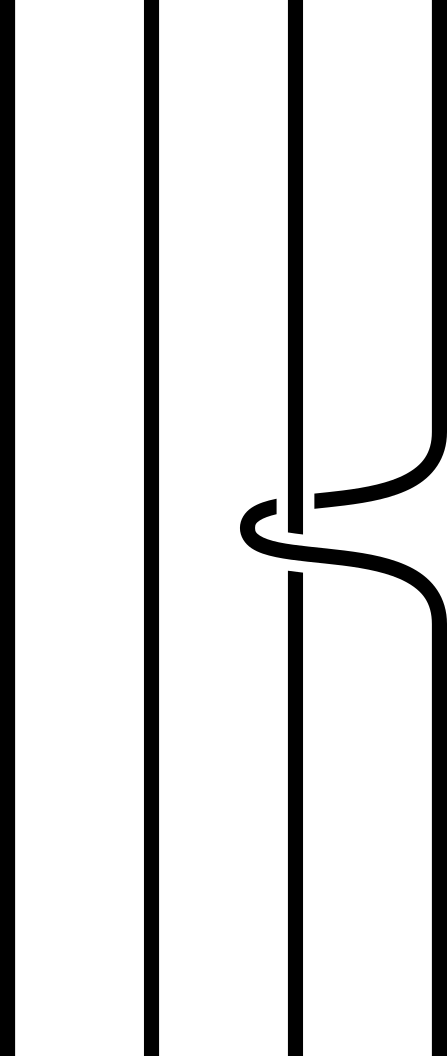}\\
\vspace{0.15cm}
\caption*{$\beta_1=\sigma_3^2$}
\includegraphics[height=0.13\textheight]{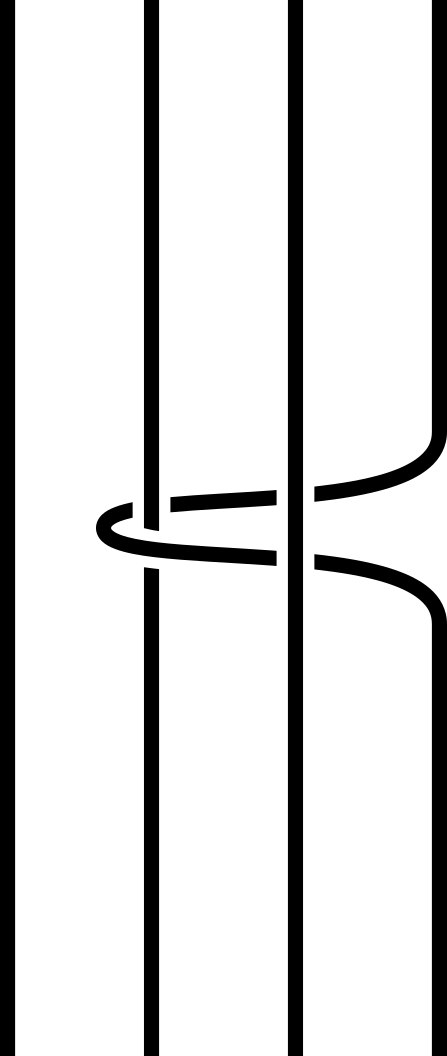}\\
\vspace{0.13cm}
\caption*{$\beta_2=\sigma_3 \sigma_2^2 \sigma_3^{-1}$}
\includegraphics[height=0.13\textheight]{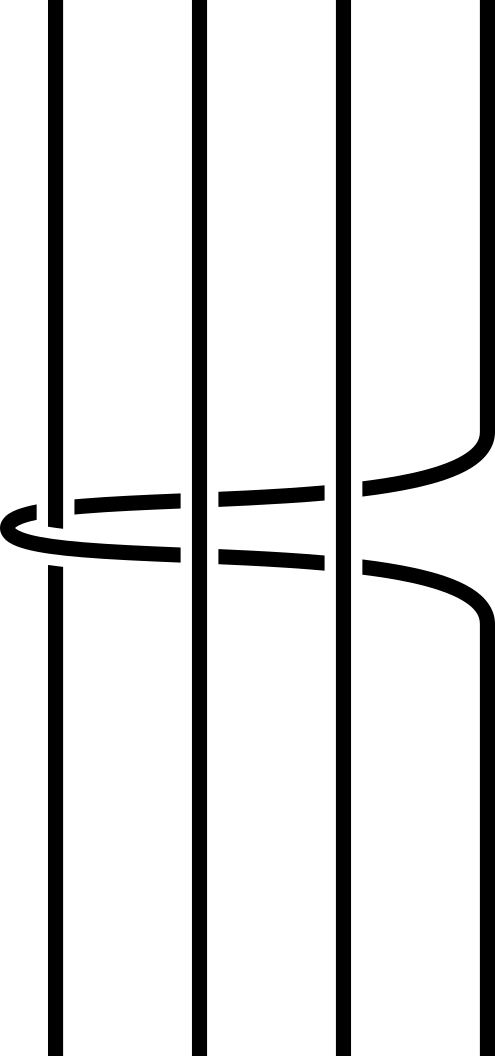}\\
\vspace{0.13cm}
\caption*{$\beta_3=\sigma_3 \sigma_2 \sigma_1^2 \sigma_2^{-1} \sigma_3^{-1}$}
\end{multicols}

\end{figure}

%\begin{figure}
%\includegraphics[width=1\textwidth]{sigma12.png}
%\end{figure}

%\begin{figure}
%\includegraphics[width=1\textwidth]{2.png}
%\end{figure}

%\begin{figure}
%\includegraphics[width=1\textwidth]{3.png}
%\end{figure}

From the relations, and from the diagrams in the appendix, we have 
$$P_4\simeq F_3(\beta_1, \beta_2, \beta_3)\rtimes \big(F_2(\alpha_1, \alpha_2)\rtimes F_1(\sigma_1^2)\big).$$
where the actions are given by the following relations :

%\begin{multicols}{2}
   \begin{enumerate}
\item $\alpha_1\beta_1\alpha_1^{-1}=(\beta_2\beta_1)^{-1}\beta_1(\beta_2\beta_1)$
\item $\alpha_1\beta_2\alpha_1^{-1}=\beta_1^{-1}\beta_2\beta_1$
\item $\alpha_1\beta_3\alpha_1^{-1}=\beta_3$
\item $\alpha_2\beta_1\alpha_2^{-1}=(\beta_3\beta_1)^{-1}\beta_1(\beta_3\beta_1)$
\item $\alpha_2\beta_2\alpha_2^{-1}=(\beta_3\beta_1)^{-1}\beta_1\beta_3\beta_2(\beta_1\beta_3)^{-1}(\beta_3\beta_1)$
\item $\alpha_2\beta_3\alpha_2^{-1}=\beta_1^{-1}\beta_3\beta_1$
\item $\sigma_1^2\beta_1\sigma_1^{-2}=\beta_1$
\item $\sigma_1^{2}\beta_2\sigma_1^{-2}=(\beta_3\beta_2)^{-1}\beta_2(\beta_3\beta_2)$
\item $\sigma_1^2\beta_3\sigma_1^{-2}=\beta_2^{-1}\beta_3\beta_2$
\item $\sigma_1^2\alpha_1\sigma_2^{-2}=(\alpha_2\alpha_1)^{-1}\alpha_1(\alpha_2\alpha_1)$
\item $\sigma_1^2\alpha_2\sigma_1^{-2}=\alpha_1^{-1}\alpha_2\alpha_1$
\label{bdr}
\end{enumerate}
%    \end{multicols}

\begin{Rem}
\label{obsv:P4} In this paper, we will use a splitting off the center of $P_4$ in order to realize $P_4$ as the direct product of its center and a semidirect product of free groups.  
The center of $P_4$ is generated by
$c=(\sigma_1\sigma_2\sigma_3)^4=\sigma_1^{2}\alpha_1\alpha_2\beta_1\beta_2\beta_3,$ as illustrated in the appendix (see \ref{CenterP4}),
and we have  \[
P_4\simeq (F(\beta_1, \beta_2, \beta_3)\rtimes F(\alpha_1, \alpha_2))\times \langle c\rangle.
\]
\end{Rem}

\subsection{The Baum--Connes conjecture for $P_n$ and $K$-amenability}

% A property of $P_n$ that we will use in order to give explicit computations of its $K$-theory groups is its $K$-amenability. This property implies that for every $C^*$-algebra $A$ endowed with an action of $P_n$, the $K$-theory of the maximal crossed product $A\rtimes P_n$ is isomorphic to the $K$-theory of the reduced crossed product $A\rtimes_r P_n$ (see \cite{Cuntz83} for the definition of $K$-amenability). 
A property of $P_n$ that we will use in order to give explicit computations of its $K$-theory groups is its $K$-amenability. This property was introduced by Cuntz (see \cite{Cuntz83} for the definition) and implies that for every $C^*$-algebra $A$ endowed with an action of $P_n$, the $K$-theory of the maximal crossed product $A\rtimes P_n$ is isomorphic to the $K$-theory of the reduced crossed product $A\rtimes_r P_n$.
In particular, 
$$K_*(C^*(P_n))\simeq K_*(C^*_r(P_n)).$$
The $K$-amenability of $P_n$ can be proven using the following result of Pimsner combined with the following proposition that is an adaptation of a result appearing in the proof of the Baum--Connes conjecture for $P_n$ given by Oyono-Oyono (see Proposition 7.3 in \cite{OO2001}) :
\begin{Thm}[\cite{Pimsner86}]
A locally compact group acting on an oriented tree such that the stabilizer group of any vertices is $K$-amenable is $K$-amenable. 
\end{Thm}

\begin{Prop}\label{K-amenability}
Let $D_0,\dots, D_n$ be a finite sequence of groups such that $D_0=\{e\}$ and for $1\leq k\leq n$ there exists $n_k$ in $\mathbb{N}$ such that $D_k=F_{n_k}\rtimes D_{k-1}$. Then $D_n$ is $K$-amenable. 
\end{Prop}

\begin{proof}
The proof is the same as the proof of Proposition 7.3 in \cite{OO2001} and it is held by induction on $n$ : if $0\leq k\leq n-1$, let $D'_k$ be the kernel of the morphism mapping $D_{k+1}$ to $D_1=F_{n_1}$. Then, $D'_0=\{e\}$ and, if $1\leq k\leq n-1$ then $D'_k=F_{n_{k-1}}\rtimes D'_{k-1}$. Hence, by induction, $D'_{n-1}$ is $K$-amenable. But the group $D_n$ acts on the Cayley graph of $F_{n_1}$ (which is a tree) through the morphism mapping $D_n$ to $D_1=F_{n_1}$ and the stabilizer group of the vertex corresponding to the neutral element of $F_{n_1}$ is exactly  $D'_{n-1}$; as the action is transitive, the stabilizer group of all vertices is $K$-amenable and hence, by Pimsner's theorem, $D_n$ is $K$-amenable.  
\end{proof}
\begin{Cor}
The pure braid group $P_n$ is $K$-amenable. 
\end{Cor}

In \cite{OO2001}, Oyono-Oyono proved that a countable discrete group acting on an oriented tree satisfies the Baum--Connes conjecture with coefficients\footnote{The Baum-Connes conjecture with coefficients is a stronger version of the Baum--Connes that considers actions of the group on a $C^*$-algebra.} if and only if the groups stabilizing the vertices of the tree satisfy Baum--Connes with coefficients. This result allows him to prove the stability of the conjecture under free and amalgamated products and HNN extensions. It also allows him to prove the analogous of Proposition \ref{K-amenability} in the context of Baum--Connes and hence to give his first proof of the Baum--Connes conjecture for $P_n$. He then proved in \cite{OO2001-2} the Baum--Connes conjecture for groups which are extensions of a group satisfying the Haagerup property by a group satisfying Baum--Connes, which allowed him to give a second proof of Baum-Connes for $P_n$ (as the free group is known to have the Haagerup property).

\section{Classifying space and \texorpdfstring{$K$-homology}{K-homology} for \texorpdfstring{$P_n$}{Pn}}
\label{Sec3}

In this section, we deal with the compactly supported $\Gamma$-equivariant $K$-homology of $\underline{E}\Gamma$, the space classifying  $\Gamma$-proper actions  for $\Gamma=P_n$. 
%, the pure braid group on $4$-strands. 

As $P_n$ is a discrete torsion-free group, $\underline{E}P_n$ coincides with $E P_n$, the universal cover of the classifying space $BP_n$, and the $P_n$-equivariant $K$-homology of $E P_n$ is the $K$-homology of the space $BP_n$, that is  $K^{P_n}_*(E P_n)\simeq K_*(B P_n)$.

We start with the case $n=4$.  We will give a model for $BP_4$ and compute its $K$-homology explicitly.

\subsection{A model for \texorpdfstring{$BP_4$}{}}\label{6relations}
Let us give a model for $BP_4$. 
Recall that 
\[
P_4\simeq (F(\beta_1, \beta_2, \beta_3)\rtimes F(\alpha_1, \alpha_2))\times \langle c\rangle,
\] 
the center of $P_4$ is generated by
$c=(\sigma_1\sigma_2\sigma_3)^4,$
and the generators $\beta_1, \beta_2, \beta_3, \alpha_1, \alpha_2$ are subject to 6 relations:
\begin{align*}
  \begin{split}
R1.\ &\alpha_1\beta_1\alpha_1^{-1}=(\beta_2\beta_1)^{-1}\beta_1(\beta_2\beta_1)\\
R2.\ &\alpha_1\beta_2\alpha_1^{-1}=\beta_1^{-1}\beta_2\beta_1\\
R3.\ &\alpha_1\beta_3\alpha_1^{-1}=\beta_3
\end{split}\hspace{.9cm}
  \begin{split}
R4.\ &\alpha_2\beta_1\alpha_2^{-1}=(\beta_3\beta_1)^{-1}\beta_1(\beta_3\beta_1)\\
R5.\ &\alpha_2\beta_2\alpha_2^{-1}=(\beta_3\beta_1)^{-1}(\beta_1\beta_3)\beta_2(\beta_1\beta_3)^{-1}(\beta_3\beta_1)\\
R6.\ &\alpha_2\beta_3\alpha_2^{-1}=\beta_1^{-1}\beta_3\beta_1
\end{split}
\end{align*}
Conjugating the relation R5 by the last  relation R6 we obtain:
\[R5'.\ \alpha_2(\beta_3\beta_2\beta_3^{-1})\alpha_2^{-1}=\beta_3\beta_2\beta_3^{-1}.
\]
We may replace relation R5 by R5' without changing the presentation of the group.
Notice that the pairs of relations R1 \& R4, R2 \& R6, R3 \& R5' are of the same type.

Let $X$ be the $2$-CW complex associated to the group $F_3\rtimes F_2$. That is, $X$ consists of $1$ $0$-cell $p$, $5$ $1$-cells attached as loops on $p$, and $6$ $2$-cells whose boundaries are given by the relations above.
Then $\pi_1(X)=F_3\rtimes F_2$. Denote by $\widetilde X$ the universal cover of $X$.
We want to show
\begin{Prop}
\label{prop:BP4}
The $2$-CW complex $X$ is a model for $B(F_3\rtimes F_2)$.
%The $2$-CW complex $X$ is a model for $B(F_3\rtimes F_2)$, and the universal cover $\widetilde X$
%is a model for $E(F_3\rtimes F_2)$.
\end{Prop}

\begin{proof}

We first construct $\widetilde X$ and then show it is contractible (Lemma~\ref{Lem:contr}).
\smallskip 

{\bf Step 1:}
We start from the $1$-skeleton of $X$, denoted by  $X^{(1)}$. This is a bouquet of $5$ circles, and its universal cover $\widetilde{X^{(1)}}$ is a tree, the Cayley graph of  the free product $F_2(\alpha_1, \alpha_2)*F_3(\beta_1, \beta_2, \beta_3)$. 
 The group $F_3*F_2$ acts freely on $\widetilde{X^{(1)}}$.

As the group $F_3\rtimes F_2$ is the quotient of $F_2*F_3$ by the six relations stated above, 
\[
F_3\rtimes F_2\simeq (F_3*F_2)/\langle R1, R2, R3, R4, R5', R6\rangle,
\]
we shall modify $\widetilde{X^{(1)}}$ so that the relations act trivially. This is done in Step 2 by  gluing $2$-cells and by identification of branches.

%Consider $E(F_2 * F_3)$, which is the Cayley graph of the free product of $F_2=\langle\alpha_1, \alpha_2\rangle$ and  $F_3=\langle \beta_1, \beta_2, \beta_3\rangle$.
%This is the universal cover of $X^{(1)}$, the $1$-skeleton of $X$, denoted $\widetilde{X^{(1)}}$.
%In fact, one can check that $F_3 * F_2$ acts freely on $\widetilde{X^{(1)}}$ with quotient $X^{(1)}.$
%The group $F_3\rtimes F_2$ is the quotient of $F_2*F_3$ by the six relations stated above, 
%\[
%F_3\rtimes F_2=F_3*F_2/\langle R1, R2, R3, R4, R5', R6\rangle.
%\]
%The action of $F_3 * F_2$ on $\widetilde{X^{(1)}}$ is free. 
%We shall modify $\widetilde{X^{(1)}}$ so that relations act trivially. This leads to the gluing of $2$-cells and identification of branches in Step 2.

\smallskip 

{\bf Step 2:}
\smallskip
Note that every point in the Cayley graph $\widetilde{X^{(1)}}$ is generic. Choose an arbitrary vertex $P\in \widetilde{X^{(1)}}$ and define $Q:=\alpha_1(P)$ and $R:=\alpha_2(P)$ in $\widetilde{X^{(1)}}$.
The six relations R1, R2, R3, R4, R5', R6 require that the pair of points on both sides of the following equations have to be identified:  
\begin{multienum}
\item $\alpha_1(\beta_2 P)=\beta_1^{-1}\beta_2\beta_1 Q$
\item $\alpha_1(\beta_1 Q)=\beta_1^{-1}\beta_2^{-1}\beta_1\beta_2\beta_1 Q$
\item $\alpha_1(\beta_3 P)=\beta_3 Q$  
\item $\alpha_2(\beta_1 P)=\beta_1^{-1}\beta_3^{-1}\beta_1\beta_3\beta_1 R$ 
\item $\alpha_2(\beta_3\beta_2\beta_3^{-1}P)=\beta_3\beta_2\beta_3^{-1}R$ 
\item $\alpha_2(\beta_3 P)=\beta_1^{-1}\beta_3\beta_1 R$
\end{multienum}
Attach $2$-cells given by the six relations and identify the branches at the vertices being glued in the 2-cells. Performing this process for each vertex in $\widetilde{X^{(1)}}$, we then obtain a $2$-dimensional CW-complex, we denoted by $\widetilde X_0$. 

\begin{Lem}
\label{lem:uni}
The space $\widetilde X_0$ constructed in Step 1 and Step 2 is the universal cover $\widetilde X$ of $X$.
\end{Lem}

\begin{proof}
Let us go through the process for R1, as an example. 
Starting from $P$ and following the expression $\alpha_1(\beta_2 P)$, we obtain the vertices $P, \, \beta_2 P,$ and $\alpha_1(\beta_2 P)$; and starting from $Q=\alpha_1(P)$ and following the expression $\beta_1^{-1}\beta_2\beta_1 Q$, we obtain vertices $Q, \, \beta_1 Q, \, \beta_2\beta_1 Q$ and $\beta_1^{-1}\beta_2\beta_1 Q$. 
So, by tracing out each letter in R1, one obtains a hexagon with $6$ vertices. Fill in the interior of the hexagon with a 2-cell. 
See Figure~\ref{hexagon}.
\begin{figure}
\center{\includegraphics[width=.28\textwidth]{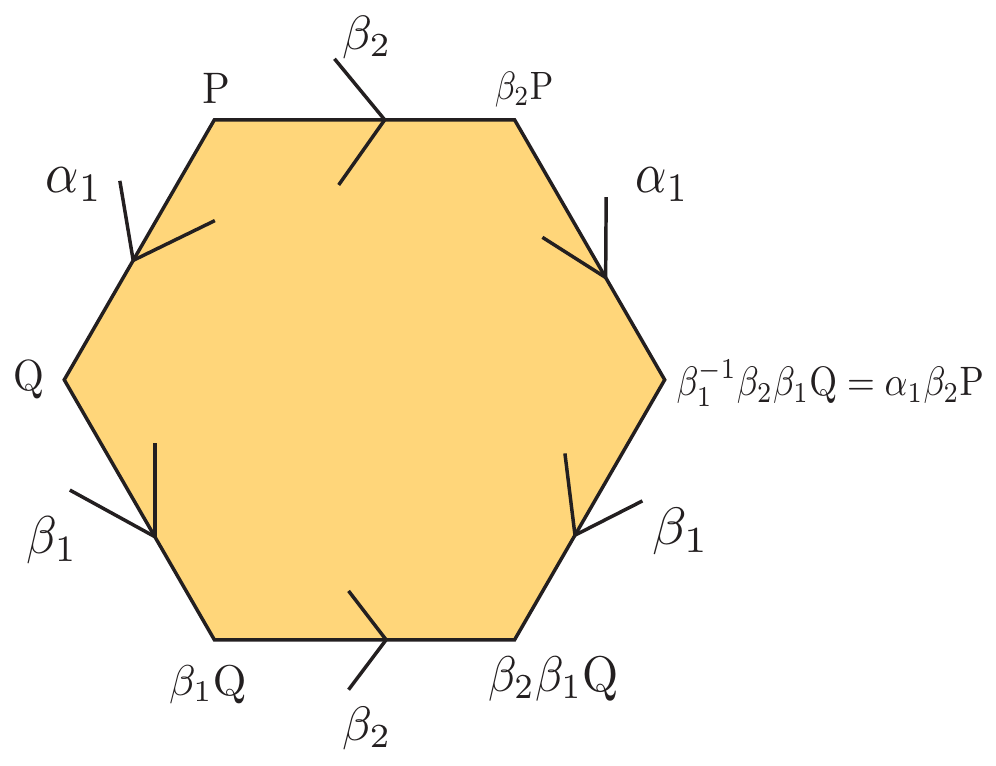}}
 \caption{\label{hexagon} 2-cell associated to R1}
\end{figure}

Because the two points $\alpha_1(\beta_2 P)$ and $\beta_1^{-1}\beta_2\beta_1 Q$ are identified in the hexagon, the branches rooted over the two points will be identified under the group action. See Figure~\ref{id} for an illustration. 

\begin{figure}
\center{\includegraphics[width=.7\textwidth]{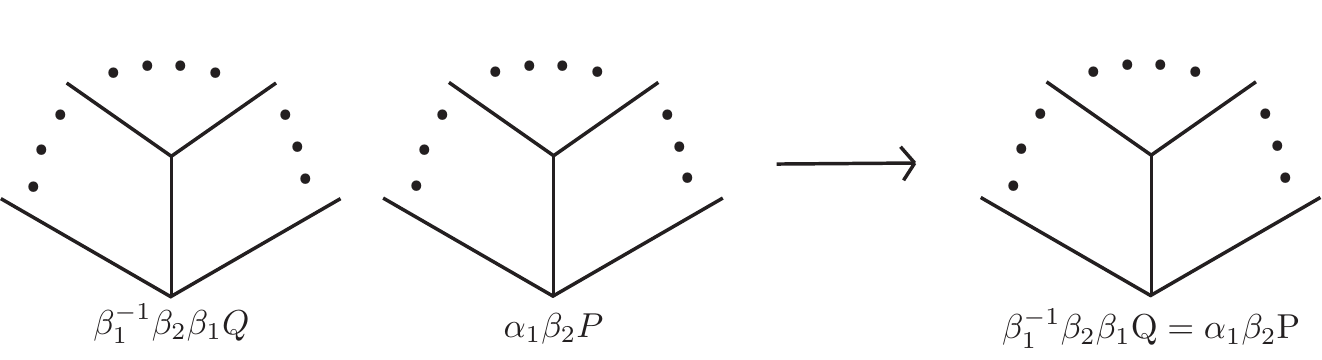}}
 \caption{\label{id} Identify branches over two identified points}
\end{figure}
%or 
%\begin{figure}
%\center{\includegraphics[width=.8\textwidth]{id}}
% \caption{\label{id} Identify branches over two %identified points}
%\end{figure}

In the left hand side of the Figure~\ref{2cellsforRelations1and2}, the hexagon associated to R1 is the green hexagon relative to other relevant vertices in the space. See Figures~\ref{2cellsforRelations1and2}, \ref{2cellsforRelations3and4} and \ref{2cellsforRelations5and6} for the typical 2-cell associated to each of the relations.

\begin{figure}
\center{\includegraphics[width=.6\textwidth]{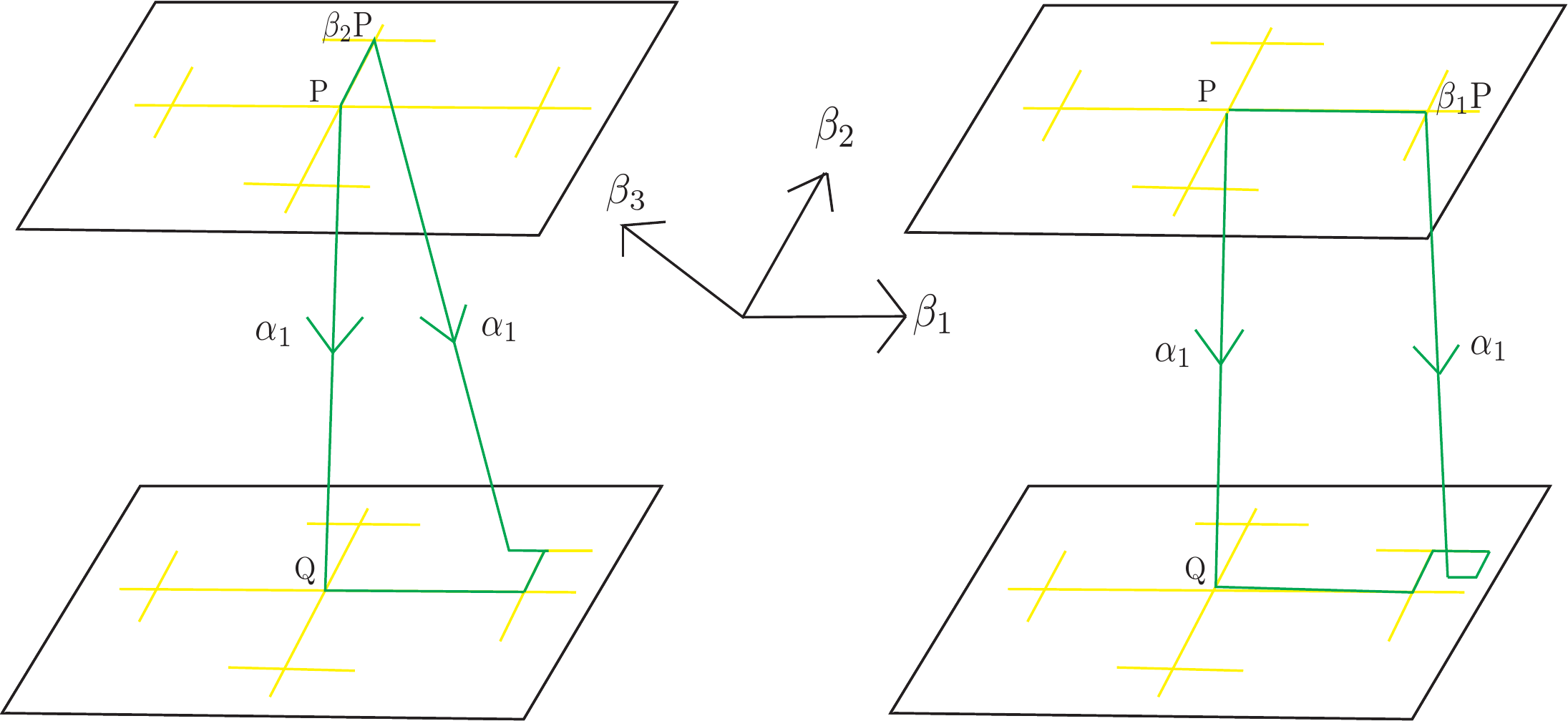}}
 \caption{\label{2cellsforRelations1and2} 2-cells for Relations 1 and 2}
\end{figure}

\begin{figure}
\center{\includegraphics[width=.6\textwidth]{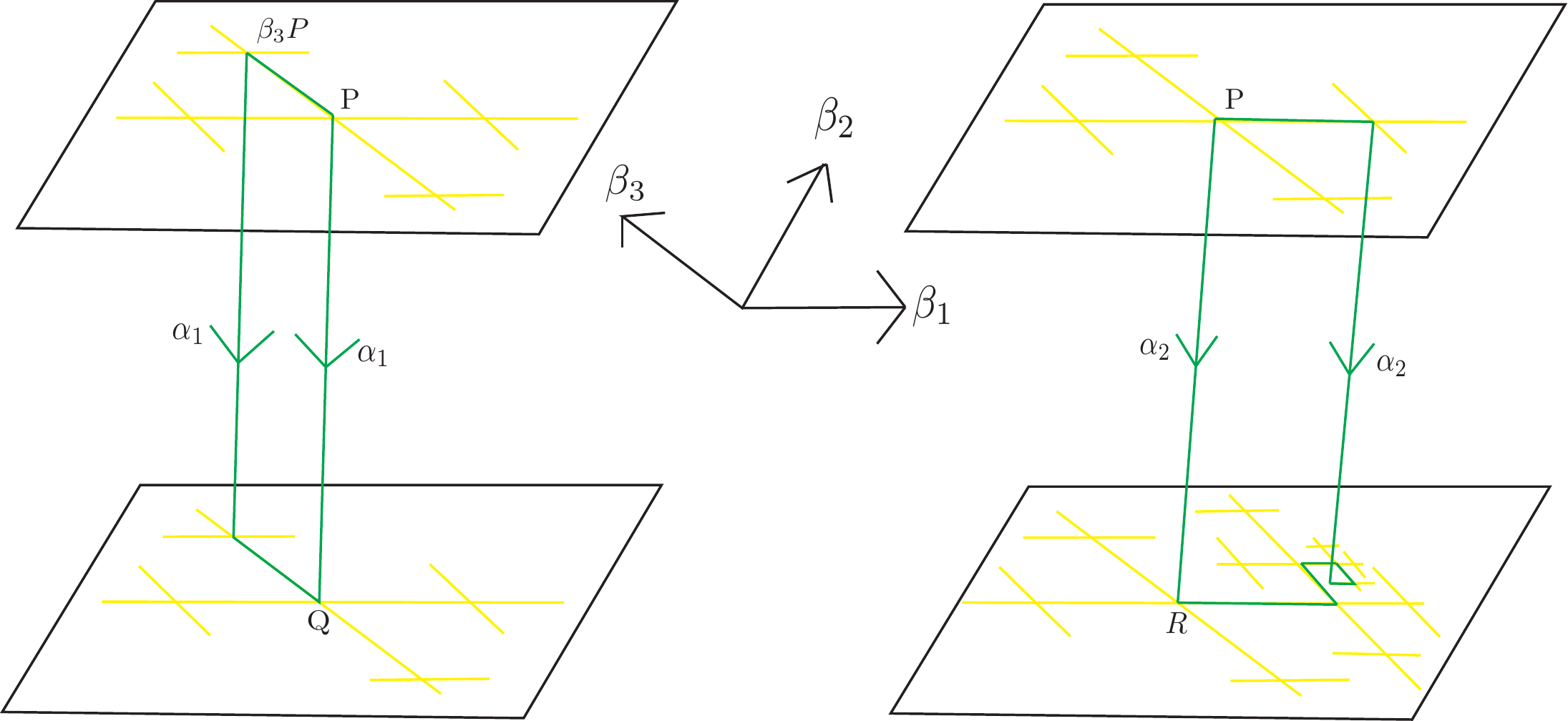}}
 \caption{\label{2cellsforRelations3and4} 2-cells for Relations 3 and 4}
\end{figure}

\begin{figure}
\center{\includegraphics[width=.6\textwidth]{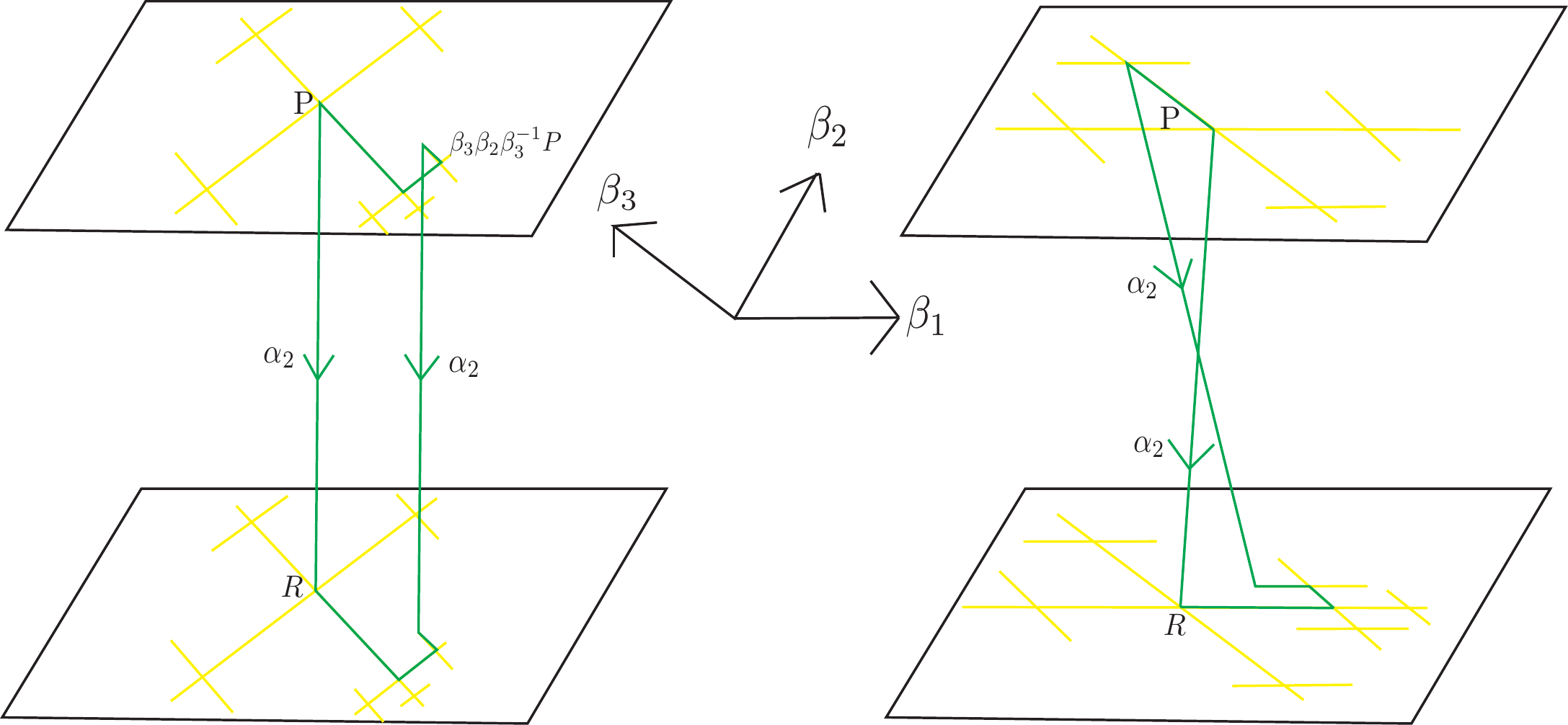}}
 \caption{\label{2cellsforRelations5and6} 2-cells for Relations 5 and 6}
\end{figure}

After the gluing of $2$-cells on $\widetilde{X^{(1)}}$ and identifications of branches,  one can check that $F_3\rtimes F_2$ acts freely on $\widetilde X_0$ with quotient $X$. 
In fact, let $g\in F_3\rtimes F_2$ and suppose $gx=x$ for some vertex $x$ in $\widetilde X_0$. 
Assume $g\neq e$. Then there exists $\tilde g\in F_3*F_2$ such that $\pi(\tilde g)=g,$ where $\pi$ is the morphism $\pi: F_3* F_2\rightarrow F_3\rtimes F_2$. Let $\tilde x\in\widetilde{X^{(1)}}$ such that $\pi(\tilde x)=x.$ 
Because $\tilde g$ is not the identity in $F_3* F_2$ and $F_3*F_2$ acts freely on $\widetilde{X^{(1)}}$, we have $\tilde g\tilde x\neq\tilde x.$
By definition, $\tilde g \tilde x$ and $\tilde x$ are identified with the same point $x$ in $\tilde X_0$. So $\tilde g \tilde x$ and $\tilde x$ can be connected by relations R1, R2, R3, R4, R5', R6. So $\tilde g$ can be expressed as $R_{i_1}\cdots R_{i_k}$, and then $g=e$ in $F_3\rtimes F_2$, which is a contradiction. This shows that $F_3\rtimes F_2$ acts freely on $\widetilde X_0$.

By construction, the quotient of $\widetilde X_0$ by $F_3\rtimes F_2$ is $X$.  Therefore the lemma is proved.
\end{proof}

%\begin{Rem}
%\label{Rem:torus}
%The construction in the proof of Lemma~\ref{lem:uni} works in general for finite $2$-CW complex. We may take an easier example $X=$ two torus, i.e., $X=B\Z\times B\Z$. As a CW complex, $X$ consists of one $0$-cell, 2 $1$-cells and one $2$-cell. The fundamental group $\Gamma$ is $\Z^2=F_2(a,b)/\{R=aba^{-1}b^{-1}\}.$ Then Step 1 is to construct the universal cover of $X^{(1)}$, which is the Cayley graph of $F_2.$ Step 2 is to attach $2$-cells and identify the branches over the identified points for each vertex $X^{(1)}$. See Figure~\ref{T2}.
%\begin{figure}
%\center{\includegraphics[width=.8\textwidth]{Step2.pdf}}
% \caption{\label{T2} Step 2 for $X=T^2$}
%\end{figure}
%Perform this process for each vertex, we will eventually obtain $\R^2$ the universal cover of $X$. 
%\end{Rem}

Proposition~\ref{prop:BP4} then follows from the lemma below.

\begin{Lem}
\label{Lem:contr}
The universal cover $\widetilde X$ is contractible.
\end{Lem}

\begin{proof}
We shall define a uniform deformation of every $2$-cell in $\widetilde X$ so that the resulting deformed $2$-complex $\widetilde X'$ is contractible.

Let $P_0=P$ be a $0$-cell in $\widetilde X$, and set $P_i=\beta_1^i(P)$ and $Q_i=\alpha_1(P_i).$
Here $\beta_1^i(P)$ means applying $\beta_1$ to $P$ $i$ times.
Let $c_i$ be the $2$-cell containing $P_i$ and $Q_i$ associated to Relation $1$.
Define a homotopy of $c_i$ by moving $Q_i$ continuously to $\beta_2\beta_1(Q_i)$ through the path $Q_i\to\beta_1(Q_i)\to\beta_2\beta_1(Q_i).$
See Figure~\ref{htp1}.
\begin{figure}
\center{\includegraphics[width=.35\textwidth]{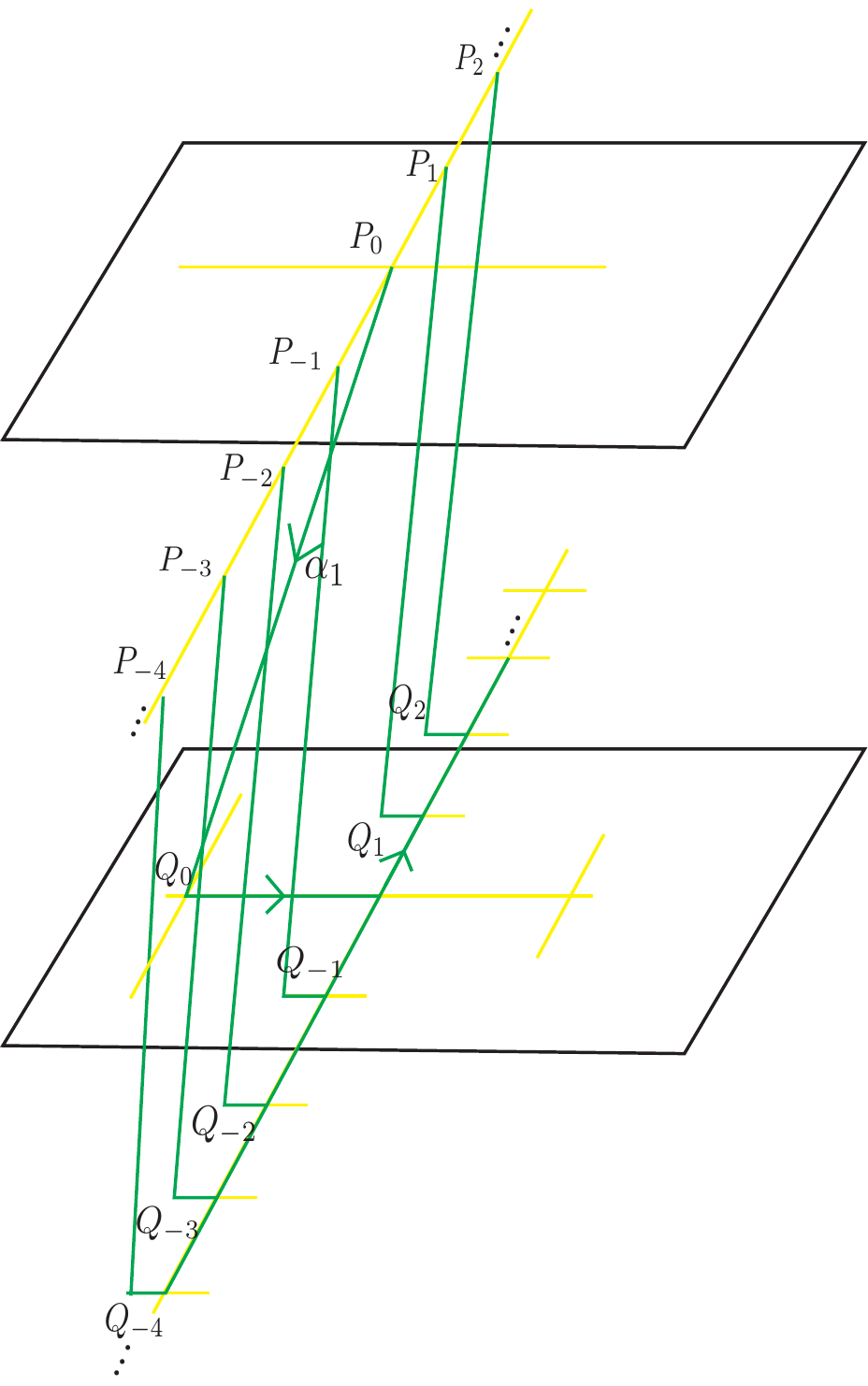}}
 \caption{\label{htp1} Homotopy for cells $c_i$}
\end{figure}
Then $c_i$ is continuously deformed to a rectangle $c_i'$, where $\cup_{i} c_i'\simeq \R\times[0,1].$ This homotopy is uniform with respect to $i$.
Similarly, one can define a homotopy for the $2$-cells associated to Relation $6.$ (In this case, $\alpha_2(P_i)$ should be moved to $\beta_3\beta_1(\alpha_2(P_i))$).

Let $P_0'=P,$ $P_i'=\beta_2^i(P),$ and $Q_i'=\alpha_1(P_i)$.
Let $d_i$ be the $2$-cell containing $P_i'$ and $Q_i'$ associated to Relation $2$. Define a homotopy of $d_i$ by moving each $Q_i'$ continuously to $\beta_2\beta_1Q_i'$ though the path $Q_i'\to\beta_1Q_i'\to\beta_2\beta_1 Q_i'$; see Figure~\ref{htp2}.
\begin{figure}
\center{\includegraphics[width=.39\textwidth]{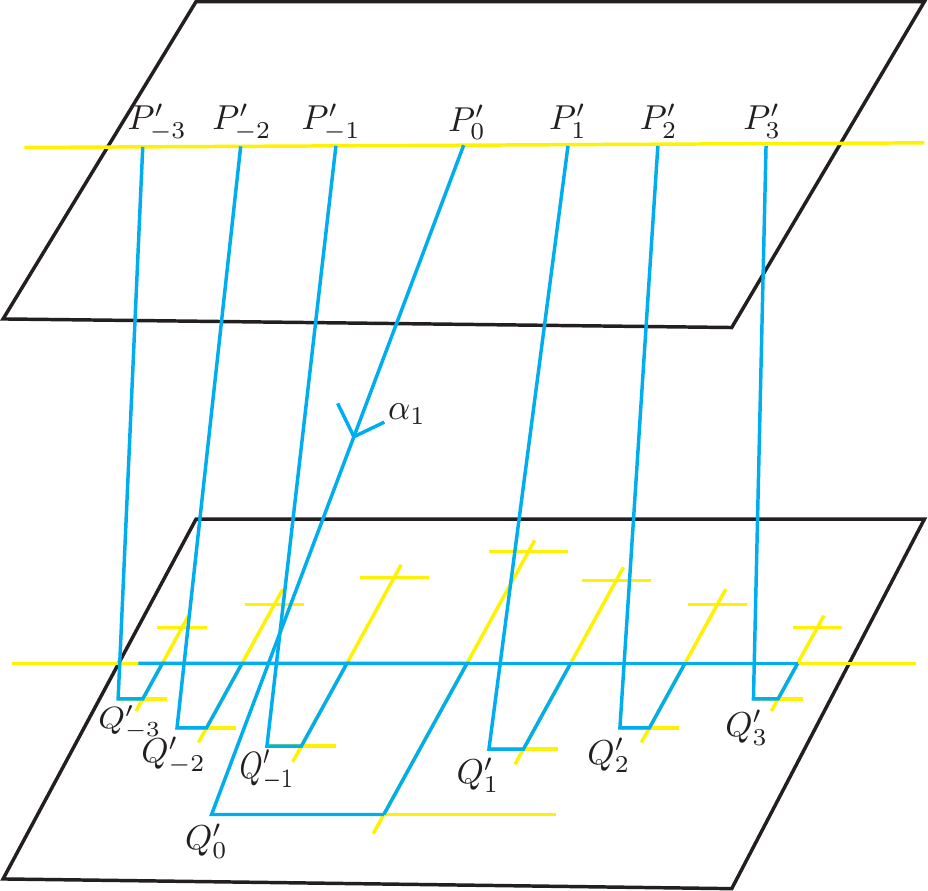}}
 \caption{\label{htp2} Homotopy for cells $d_i$}
\end{figure}
Then $d_i$ deforms continuously to a rectangle $d_i'$, uniformly with respect to $i$, with $\cup_{i\in\Z}d_i'\simeq\R\times[0,1].$

Similarly, one can define a homotopy for the $2$-cells associated to Relation $4.$ (In this case, $\alpha_2(P_i')$ should be moved to $\beta_3\beta_1(\alpha_2(P_i)))$.

 Relation $3$ gives rise to $2$-cells in the shape of a rectangle. Let $a$ be a rectangle containing $P$ and $\alpha_1 Q$. Deform $a$ by a homotopy carrying the edge $l$ with vertices $Q$ and $\beta_3 Q$ to $\beta_2\beta_1 l$ through the path $l\to\beta_1 l\to\beta_2\beta_1 l.$

For Relation $5$, let $Q=\alpha_2(P).$ In the cell $e$ containing $P$ and $Q$ and corresponding to Relation $5$, move the edge $l$ with vertices $\beta_3\beta_2\beta_3^{-1}(P)$ and $\beta_3\beta_2\beta_3^{-1}(Q)$ through the path $l\rightarrow\beta_3^{-1}l\to\beta_2^{-1}\beta_3^{-1}l$; see Figure~\ref{htp3}.
\begin{figure}
\center{\includegraphics[width=.35\textwidth]{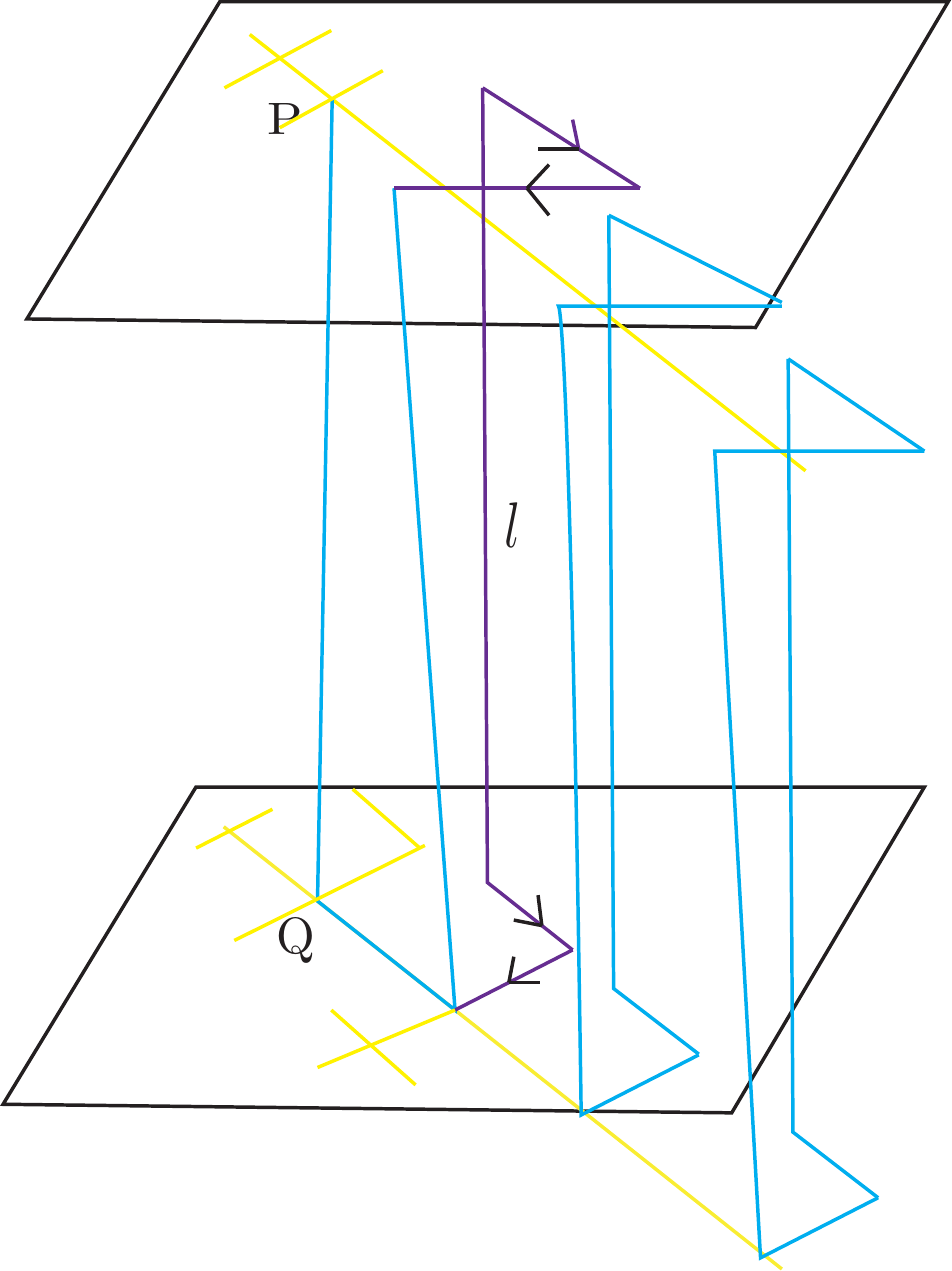}}
 \caption{\label{htp3} Homotopy for cells $e$}
\end{figure}
Then $e$ is deformed to a $2$-cell $e'$ in the shape of a rectangle.
Deform $e'$ again by a homotopy carrying the edge $l$ with vertices $Q$ and $\beta_3^{-1} Q$ to $\beta_3\beta_1 l$ through the path $l\to\beta_1 l\to\beta_3\beta_1 l.$

\smallskip

After this process, we obtain a $2$-CW complex $\widetilde X'$. 
All the $2$-cells in $\widetilde X$ are turned into rectangular-shaped $2$-cells in $\widetilde X'$. 
Algebraically, this process corresponds to the abelianization of all relations. Note that here we are using the special structure of the pure braid group: Indeed, the deformations can be done uniformly, because all relations in $P_4$ have the form 
\begin{equation}
    \label{eq:rel}
\alpha_i\beta_j\alpha_i^{-1}=C\beta_j C^{-1}
\end{equation}
where $C$ is a word depending on $i,j$, having finite letters chosen from $\beta_1, \beta_2, \beta_3.$
The deformation from $\widetilde X$ to $\widetilde X'$ corresponds to replacing (\ref{eq:rel}) by $\alpha_i\beta_j\alpha_i^{-1}=\beta_j$. 
Indeed, the relations determine the group 
\[
F_2\times F_3=\langle\alpha_1, \alpha_2, \beta_1, \beta_2, \beta_3|\alpha_i\beta_j=\beta_j\alpha_i, \forall i,j\rangle.
\]
Therefore we have shown that $\widetilde X$ is homotopic to  
\[
\widetilde X'=E(F_2\times F_3)=EF_2\times EF_3
\]
which is a contractible space.
The lemma is then proved.
\end{proof}

This completes the proof of the proposition.  
\end{proof}

\subsection{\texorpdfstring{$K$}{K}-homology of \texorpdfstring{$BP_4$}{}}
We are ready to compute the $K$-homology of $P_4$.
\begin{Thm}
\label{thm:K-homology}
We have that \[K_0(BP_4)\simeq K_1(BP_4)\simeq\Z^{12}\ .
\]
Hence, as $P_4$ is torsion-free, $
K_0^{P_4}(\underline EP_4)\simeq K_1^{P_4}(\underline EP_4)\simeq\Z^{12}$. 
\end{Thm}
To prove this theorem, we first observe the following two easy well-known facts: 
\begin{Lem}
\label{Lem:Xcircle}
For any finite CW complex $X$, we have
\[
K_i(X\times S^1)\simeq K_0(X)\oplus K_1(X) \qquad\text{for}\quad i=0,1.
\]
\end{Lem}

\begin{proof}
Let $i$ be $0$ or $1$. Note that
$
K_i(X\times S^1)\simeq K^i(C(X\times S^1))\simeq K^i(C(X)\otimes C(S^1)).
$
As $C(S^1)\simeq C_0(\R)\oplus\C,$
we have that
\[
K^i(C(X)\otimes C(S^1))\simeq K^i(C(X)\otimes C_0(\R))\oplus K^i(C(X)).
\]
Noting that $K^i(C(X)\otimes C_0(\R))=K^{i+1}(C(X)),$ the lemma is proved.
\end{proof}

\begin{Lem} We have \[
BP_4\simeq B(F_3\rtimes F_2)\times S^1.
\]
\end{Lem}

\begin{proof}
We use the isomorphism  $P_4\simeq F_3\rtimes F_2\rtimes F_1$.  By changing the representative of the generator of $F_1 \simeq \Z$ in $P_4\simeq F_3\rtimes F_2\rtimes F_1$, we can obtain the trivial action of $F_1$; see Remark \ref{obsv:P4}. Hence $P_4\simeq (F_3\rtimes F_2)\times\Z$.
Thus
\[
BP_4\simeq B(F_3\rtimes F_2)\times B\Z\simeq B(F_3\rtimes F_2)\times S^1;
\]
which proves the lemma.
\end{proof}

Recall that a model of $B(F_3\rtimes F_2)$ is given (Proposition~\ref{prop:BP4}) by the $2$-dimensional CW complex $X$  constructed in Lemma \ref{lem:uni}  (and associated to the group presentation of $F_3\rtimes F_2$). %that is, $X$ consists of the $5$ generators $\beta_1, \beta_2, \beta_3, \alpha_1, \alpha_2$ and $6$ relations.

\begin{Lem}
\label{Lem:X} We have $
K_0(X)\simeq\Z^7$ and $K_1(X)\simeq\Z^5.
$
\end{Lem}

\begin{proof}
%Because $X$ is a finite CW complex, then we have
By Lemma 4.1 in \cite{misval}, because $X$ is a $2$-dimensional CW complex,  we have
\[
K_0(X)\simeq H_0(X, \Z)\oplus H_2(X, \Z)\qquad K_1(X)\simeq H_1(X,\Z)\ .
\]
Note that $H_0(X, \Z)=\Z$, since $X$ is connected; and
\[
H_1(X, \Z)\simeq (F_3\rtimes F_2)/[F_3\rtimes F_2, F_3\rtimes F_2]\simeq\Z^5.
\]
To calculate $H_2(X, \Z)$, one notes that all $6$ relations are cycles (nontrivial and distinct), and there are at most $6$ $2$-cells, so $H_2(X, \Z)\simeq\Z^6.$
The lemma is thus proved.
\end{proof}

%\begin{Lem}
%$X$ is a model for $B(F_3\rtimes F_2).$
%\end{Lem}

%\begin{proof}
%Since $1\to F_3\to F_3\rtimes F_2 \to F_2\to 1$, we have a fibration $BF_3\to B(F_3\rtimes F_2)\to BF_2$ which implies that the minimal dimension for $B(F_3\rtimes F_2)$ is $2.$
%Note that from the fibration structure $B(F_3\rtimes F_2)$ has exactly $1$ and $2$ (to do) cells as in $X$.
%Then lemma is then proved.
%\end{proof}

\begin{proof}[Proof of Theorem~\ref{thm:K-homology}]
Making use of the Lemmas, we have for $i=0$ or $1$: 
\[
K_i^{P_4}(\underline E P_4)\simeq K_i(BP_4)\simeq K_0(B(F_3\rtimes F_2))\oplus K_1(B(F_3\rtimes F_2)),
\]
where the first isomorphism is due to $P_4$ being torsion-free, and the second isomorphism follows from Lemma~\ref{Lem:Xcircle}.
Then by Proposition~\ref{prop:BP4} and Lemma~\ref{Lem:X}, we have
\[
K_i(B(F_3\rtimes F_2))\simeq K_i(X)\simeq\begin{cases}\Z^7 & i=0 \\ \Z^5 & i=1 \end{cases}.
\]
The theorem then follows.
\end{proof}

\subsection{$K$-homology of $BP_n$}\label{K-homology}

In this section, we show that for $i=0$ or $1$,
\begin{equation}
\label{eq:Kiso}
K_i(BP_n)\simeq K_i(Y)\simeq\Z^{\frac{n!}{2}},
\end{equation}
where $Y=BF_{n-1}\times BF_{n-2}\times\cdots\times BF_1.$

Recall that rationally, the Chern character on $K$-homology for any finite CW complex $X$ is an isomorphism:
\begin{equation}
\label{eq:KhomH}
K_0(X)\otimes\Q\simeq \bigoplus_{i}H_{2i}(X, \Q)\qquad K_1(X)\otimes\Q\simeq\bigoplus_{i}H_{2i-1}(X, \Q).
\end{equation}
Thus the second isomorphism in (\ref{eq:Kiso}) holds up to torsion.

\begin{Lem}
\label{Lem:KofY}
Let $Y=BF_{n-1}\times BF_{n-2}\times\cdots\times BF_1.$ Then modulo torsion, we have
\[
K_0(Y)\simeq\Z^{\frac{n!}{2}}\qquad\text{and}\quad K_1(Y)\simeq\Z^{\frac{n!}{2}}.
\]
\end{Lem}

\begin{proof}
We need to show that $\sum_i\dim H_{2i}(Y)=\frac{n!}{2}=\sum_i\dim H_{2i-1}(Y).$
Observe that $H_{i}(Y, \Q)$ is generated by $i$-cells in $Y$.
Denote by $a_k$ the number of $k$-cells in the CW-complex $Y$.
It is enough to show that
\[
\sum_{i} a_{2i}=\sum_{i} a_{2i-1}=\frac{n!}{2}.
\]
Because $BF_k$ has one $0$-cell and $k$ $1$-cells, we have that for $Y=BF_{n-1}\times BF_{n-2}\times\cdots\times BF_1$,
\[
a_0=1, \qquad a_1=\sum_{i=1}^{n-1}i, \qquad a_2=\sum_{1\le i<j\le n-1}ij,\qquad \ldots,\qquad  a_{n-1}=1\cdot 2\cdots (n-1).
\]
That is, $a_m$ is equal to the coefficient of $t^m$ in the series $(1+t)(1+2t)\cdots(1+[n-1]t).$
Therefore we have
\[
\sum_{k=1}^{n-1}a_k=\prod_{l=1}^{n-1}(1+l)=\prod_{l=1}^n l=n!
\]
and
\[
\sum_{k=1}^{n-1}(-1)^ka_k=\prod_{l=1}^{n-1}(1-l)=0.
\]
The lemma is then proved.
\end{proof}

The first isomorphism in (\ref{eq:Kiso}) also holds up to torsion. In other words, the $K$-homology for $BP_n$ can be computed rationally.

\begin{Prop}
\label{prop:KhomPn}
Up to torsion,
\[
K_0(BP_n)\simeq K_1(BP_n)\simeq \Z^{\frac{n!}{2}}.
\]
\end{Prop}

\begin{proof}
 The integer cohomology of $BP_n$ of degree $m$ is torsion-free, with the power over $\Z$ being equal to the coefficient of $t^m$ in the series $(1+t)(1+2t)\cdots(1+[n-1]t).$ This was first computed by Arnol'd~\cite{Arnold};
 see also Example 2.3 of \cite{ACC03}.
Applying the proof of Lemma~\ref{Lem:KofY}, we find that up to torsion,
\[
K^0(BP_n)\simeq H^{even}(BP_n, \Z)\simeq \Z^{\frac{n!}{2}}\quad K^1(BP_n)\simeq H^{odd}(BP_n, \Z)\simeq \Z^{\frac{n!}{2}}.
\]
Together with (\ref{eq:KhomH}) and Poincar\'e duality, this gives the result. Note that $BP_n$ is a Poincar\'e complex in the sense of Wall \cite{Wall} and hence satisfies the Poincar\'e duality theorem.
To see that $BP_n$ is a Poincar\'e complex, notice that $BP_n$ is composed of iterated extensions of the $BF_i$s, and it is easy to verify that the $BF_i$s are Poincar\'e complexes.
\end{proof}

Finally, the torsion in  Proposition~\ref{prop:KhomPn} can be removed using Atiyah-Hirzebruch spectral sequences, analogous to Arnold's calculation of the cohomology groups $H^*(BP_n)$, using Serre spectral sequences, as there is a fibration for $BP_n$. Then using Poincar\'e duality as in Proposition \ref{prop:KhomPn}, we get the result at the level of the homology groups. Let us review some key steps and properties of pure braid groups in computing $H^*(BP_n)$. For more details, see \cite{Wilson}.

It is well known that the ordered configuration space of $n$ points in $\C$ is a model for the classifying space of $P_n$: 
\[
BP_n=\{(z_1, \ldots, z_n)\in \C^n| z_i\neq z_j \text{ if } i\neq j\}.
\]
Note that for $n=4$, this model is homotopic to the model of $BP_4$ that we constructed explicitly earlier.
Consider the map 
\[
\rho_n: BP_n\rightarrow BP_{n-1} \qquad (z_1, \ldots, z_n)\mapsto(z_1, \ldots, z_{n-1}).
\]
This is a fibration whose fiber is homeomorphic to the $(n-1)$-times punctured plane, which is homotopic to a wedge of $n-1$ circles, so that the filtration can be written as $$BF_n \rightarrow BP_n\rightarrow BP_{n-1}.$$ This can also be induced by the short exact sequence $1\rightarrow F_{n-1}\rightarrow P_n\rightarrow P_{n-1}\rightarrow 1.$
The map 
\[
i_n: BP_{n-1}\rightarrow BP_n \qquad (z_1, \ldots, z_{n-1})\mapsto(z_1, \ldots, z_{n-1}, \mathrm{max}|z_i|+1)
\]
gives rise to a splitting of the fibration $\rho_n.$
Associated to a fibration, there is a monodromy action of $\pi_1(B)$ on the (co)homology of the fiber $F.$
In our setting, it can be checked that $P_{n-1}=\pi_1(BP_{n-1})$ acts on the homology of $BF_{n-1}$ trivially.  
Then the $E_2$ page of the Atiyah-Hirzebruch spectral sequence  
\[
E_2^{p,q}(BP_n)=H^p(BP_{n-1}, H^q(BF_{n-1}))
\]
is a cohomology with untwisted coefficients in $H^q(BF_{n-1})$ and is calculated by 
\[
E_2^{p,q}(BP_n)=\begin{cases}H^p(BP_{n-1}) & q=0\\ H^p(BP_{n-1})\otimes\Z^{n-1} & q=1 \\ 0 & q>1. \end{cases}
\]
See Exercise 39 of~\cite{Wilson}.
Though not obvious from the definition, following the outline in Exercise 40 of~\cite{Wilson} one can then show that the spectral sequences collapse on the $E_3$ page. This is because on $E_2$ page all differentials vanishes:
{\tiny
\[\xymatrix {0\ar[rrd] & 0 \ar[rrd] & 0 \ar[rrd] & 0 &   \\
             0\ar[rrd] & 0 \ar[rrd] & 0 \ar[rrd] & 0 &   \\  H^0(BP_{n-1})\otimes\Z^{n-1}\ar[rrd]^{d_2=0} & H^1(BP_{n-1})\otimes\Z^{n-1} \ar[rrd]^{d_2=0} & H^2(BP_{n-1})\otimes\Z^{n-1} \ar[rrd]^{d_2=0} & H^3(BP_{n-1})\otimes\Z^{n-1} &   \\
             H^0(BP_{n-1}) & H^1(BP_{n-1}) & H^2(BP_{n-1}) & H^3(BP_{n-1}) & }.
\]
}
Therefore we obtain 
\[
H^k(BP_n)=\bigoplus_{p+q=k} E^{p,q}_2(BP_n)=E^{k, 0}_2(BP_n)\oplus E^{k-1, 1}_2(BP_n). 
\] 
The group can then be calculated using induction.

\begin{Thm} We have
\[
K_0(BP_n)\simeq K_1(BP_n)\simeq \Z^{\frac{n!}{2}}.
\]
\end{Thm}

\begin{proof}
We shall use Atiyah--Hirzebruch spectral sequence (see for example \cite{Hatcher_SSAT, misval}
\[
E_{p,q}^2(BP_n)=H_p(BP_n,  K_q(pt))\Rightarrow K_{p+q}(BP_n).
\]
Because %the monodromy action is trivial, the homology 
\[
E_{p,q}^2(BP_n)=H_p(BP_{n}, K_q(pt))=\begin{cases}H_p(BP_{n}) & q\text{ even}\\ 0 & q\text{ odd}. \end{cases}
\]
Identify the homology and cohomology via Poincar\'e duality and one has that $E_{p,q}^2(BP_n)$ is torsion free. 
Note that the differential $d_2: E^2_{p,q}\rightarrow E^2_{p-2, q+1}$ vanishes for all integers $p, q$, so that spectral sequence collapses on the $E^3$ page, 
{\tiny
\[\xymatrix { 0 & 0  & 0  & 0 &  \\
             H_0(BP_{n}) & H_1(BP_{n}) & H_2(BP_{n})\ar[ull]^{d_2=0} & H_3(BP_{n}) \ar[ull]^{d_2=0} & \ar[ull]^{d_2=0}  \\  
             0 & 0 & 0 \ar[ull]^{d_2=0} & 0 \ar[ull]^{d_2=0} &  \ar[ull]^{d_2=0} \\
             H_0(BP_{n-1}) & H_1(BP_{n-1}) & H_2(BP_{n-1})\ar[ull]^{d_2=0} & H_3(BP_{n-1}) \ar[ull]^{d_2=0} & \ar[ull]^{d_2=0}  }.
\]
}
and also
\[
 K_0(BP_n)\simeq \bigoplus_{p+q\equiv 0 (2)}E_{p,q}^2(BP_n)\qquad K_1(BP_n)\simeq \bigoplus_{p+q\equiv 1 (2)}E_{p,q}^2(BP_n).
\]
Therefore 
\begin{align}
\label{eq:KH1}
  K_0(BP_n)\simeq H_{even}(BP_{n})\\
  K_1(BP_n)\simeq H_{odd}(BP_{n}). \label{eq:KH2}
\end{align}
Because $H_*(BP_n)$ is torsion-free for all $n$, from the above we obtain that $K_*(BP_n)$ is torsion-free as well.
Together with Proposition~\ref{prop:KhomPn}, the theorem then follows.
\end{proof}

Note that the calculation of the $K$-homology of $Y=BF_{n-1}\times\cdots\times BF_{1}$ is not used in the proof of the theorem, but the comparison to the $K$-homology of $BP_n$ provides intuition for the cycles of $BP_n$ that form a set of generators for $K_*(BP_n)$ in light of (\ref{eq:KH1}) and (\ref{eq:KH2}). 
By the fibration structure $BF_{n-1}\rightarrow BP_n\rightarrow BP_{n-1}$, and by induction, $BP_{n}$ has a model of dimension $n-1$, and every $k$-simplex of $BP_n$ is labelled by $A_{j_1, i_1}, \ldots , A_{j_k, i_k},$ where $1\le i_1<\cdots <i_r\le n$, $1\le j_l<i_l$, and $F_l(A_{1, l+1}, \ldots, A_{l, l+1})$ is the free subgroup of $P_n$. 
They are cycles because of the pure braid group relations, and they correspond bijectively to cycles of $Y$. There is a canonical map $BP_n\rightarrow Y$ inducing an isomorphism on $K$-homology; see Section~\ref{Sec5.3} for more details.

\section{$K$-theory of the reduced group $C^*$-algebra of $P_n$} 
\label{Sec4}

 In this section, we compute the right-hand side of the Baum--Connes morphism for $P_n$. We shall use the Pimsner--Voiculescu six-term exact sequence, as it allows us to compute the $K$-theory of the reduced crossed product of a $C^*$-algebra with a free group \cite{PV82}. Let $A$ be a $C^*$-algebra endowed with an action of the free group on $n$ generators $F_n=F_n(x_1, \dots x_n)$ by automorphisms $\phi:F_n\to {\rm Aut}(A)$. Following Pimsner and Voiculescu  (see \cite[Theorem 3.1, Theorem 3.5]{PV82}), we have two six-term exact sequences
 \label{PV1}
         \begin{equation}
    \label{eq:PV1}
    \xymatrix{
K_0(A)\ar[r]^-{w_n} & 
K_0(A\rtimes_{\phi', r} F_{n-1}) \ar[r]^{k_*} &
K_0(A\rtimes_{\phi, r} F_{n}) \ar[d]  
\\
 K_1(A\rtimes_{\phi, r} F_n)\ar[u] &
 K_1(A\rtimes_{\phi', r} F_{n-1}) \ar[l]_{k_*} &
 K_1(A) \ar[l]_-{w_n}
 }
    \end{equation}
where $\phi'$ is the restriction of the action to $F_{n-1}$,   $k:A\rtimes_{\phi', r} F_{n-1}\to A\rtimes_{\phi, r} F_{n} $ is the natural inclusion, $i:A\to A\rtimes_{\phi', r} F_{n-1}$ is the canonical inclusion, 
$w_n=i_*\circ (id_*-\phi(x_n^{-1})_*),$
and the vertical arrows correspond to the connecting homomorphisms of a sequence induced by a Toeplitz extension;
and 
\begin{equation}
\label{eq:PVFree}
\xymatrix{
(K_0(A))^n\ar[r]^\theta & K_0(A) \ar[r]^-{\pi_*}  & K_0(A\rtimes_r F_n) \ar[d]  \\
 K_1(A\rtimes_r F_n)\ar[u] & K_1(A) \ar[l]^-{\pi_*} & (K_1(A))^n \ar[l]_\theta
 }
\end{equation}
where $\theta$ is the map 
$$
\theta(\gamma_1\oplus\dots\oplus\gamma_n)=\sum_{i=1}^n(\gamma_i-\phi(x_i^{-1})_*(\gamma_j))\ .
$$
and the vertical arrows come from  connecting homomorphisms of a Toeplitz extension.

Moreover, we recall that in \cite{PV82}, Pimsner and Voiculescu used these six-term exact sequences to give the following computation of the $K$-theory of the $C^*$-algebra of a free group: 
\begin{equation} 
\label{eq:K1CF} 
    K_0(C^*_r(F_n(\alpha_1, \dots, \alpha_n)))\simeq \mathbb Z[1]
\end{equation}
and
\begin{equation}
\label{eq:K0CF}
  K_1(C^*_r(F_n(\alpha_1, \dots, \alpha_n)))\simeq \mathbb Z^n [u_{\alpha_1}, \dots, u_{\alpha_n}]
\end{equation}
where for $g\in F_n(\alpha_1, \dots, \alpha_n)$, we denote by $u_g$ the element in $C^*_r(F_n(\alpha_1, \dots, \alpha_n))$ such that  $u_{g}(h)=1$ if $h=g$, and $u_{g}(h)=0$ if $h \neq g$.

As the group $F_n$ is $K$-amenable (see \cite{Cuntz83} and \cite{Cuntz82}),
 \begin{equation}
\label{eq:K1CmF}
    K_0(C^*(F_n(\alpha_1, \dots, \alpha_n)))\simeq \mathbb Z[1]
\end{equation}
and 
\begin{equation}
\label{eq:K0CmF}
  K_1(C^*(F_n(\alpha_1, \dots, \alpha_n)))\simeq \mathbb Z^n [u_{\alpha_1}, \dots, u_{\alpha_n}]\ .
\end{equation}

\subsection{\texorpdfstring{$K$}{K}-theory of \texorpdfstring{$C^*_r(P_4)$}{}} 

We will start with the case of $n=4$. To compute the $K$-theory of $C^*_r(P_4)$, we will use the decomposition  $P_4=(F_3\rtimes_{\varphi} F_2)\times\mathbb Z$, where the action $\varphi$ is given by the relations in Remark \ref{obsv:P4}. Its reduced $C^*$-algebra $C^*_r(P_4)$ is then isomorphic to $C^*_r(F_3\rtimes F_2)\otimes C^*_r(\mathbb Z)$, and by the K\"unneth formula we have the following decompositions: 
\[
K_0(C^*_r(P_4))\simeq K_0(C^*_r(F_3\rtimes F_2))\otimes_{\mathbb Z} K_0(C^*_r(\mathbb Z))\oplus K_1(C^*_r(F_3\rtimes F_2))\otimes_{\mathbb Z} K_1(C^*_r(\mathbb Z)),
\]
\[
K_1(C^*_r(P_4))\simeq K_0(C^*_r(F_3\rtimes F_2))\otimes_{\mathbb Z} K_1(C^*_r(\mathbb Z))\oplus K_1(C^*_r(F_3\rtimes F_2))\otimes_{\mathbb Z} K_0(C^*_r(\mathbb Z))\ . 
\]
We then compute the $K$-theory of $C^*_r(F_3\rtimes F_2)$. By the following well-known lemma, the former  $C^*$-algebra is the reduced crossed product $C_r^* (F_3)\rtimes_r F_2$ 
(see  \cite[Example 2.3.6]{MR3618901}, \cite[\S 3.3]{williams}).

\begin{Lem}
\label{lem:crossedproduct}
Let $H$ and $N$ be discrete groups, and let $\phi: H\to {\rm Aut} (N)$ be an action by group automorphisms. Then $\phi$ gives actions of $H$ on the full and reduced group $C^*$-algebras $C^*(N)$ and  $C^*_r(N)$ that are given by the formula $\phi_h(f)(n)=f(h^{-1}nh),$ for $f\in C_c(N)$, $h\in H$ and $n\in N$, and one has
\begin{align*}
&C^*_r(N\rtimes_\phi H)\simeq C_ r^*(N)\rtimes_{ \phi,r} H,\\
&C^*(N\rtimes_\phi H)\simeq C^*(N)\rtimes_{\phi} H. \end{align*}
\end{Lem}

Applying Pimsner--Voiculescu's first sequence to compute the $K$-theory of  $C^*_r (F_3)\rtimes_r F_2$, where 
$F_3= F_3(\beta_1, \beta_2, \beta_3)$ and $ F_2=F_2(\alpha_1, \alpha_2)$, we get the following result.
\begin{Prop}
The $K$-theory of $C^*_r(F_3\rtimes F_2)$ is as follows: 
  \[
 K_0(C^*_r(F_3\rtimes F_2))\simeq \mathbb Z^7,\qquad \qquad K_1(C^*_r(F_3\rtimes F_2))\simeq \mathbb Z^5.
 \]
\end{Prop}
\begin{proof}
Denote by $\varphi:F_2\to \mathrm{Aut}(F_2)$ the action of $F_3$ on $F_2$ given by the relations which determine the structure of $P_4$.
Set $B:=C_r^* (F_3)\rtimes_r F_2 =C^*_r(F_3(\beta_1, \beta_2, \beta_3))\rtimes_r F_2(\alpha_1, \alpha_2)$. From \eqref{eq:PV1}, we have the exact sequence:
\begin{equation}
\label{eq:PV32}
\xymatrix{
    K_0(C^*_r(F_3(\beta_1, \beta_2, \beta_3)))\ar[r]^-{w_2} & K_0(C^*_r(F_3(\beta_1, \beta_2, \beta_3))\rtimes_r F(\alpha_1)) \ar[r]^-{k_*} & K_0(B) \ar[d]  
\\
    K_1(B)\ar[u] & K_1(C^*_r( F_3(\beta_1, \beta_2, \beta_3))\rtimes_r F(\alpha_1)) \ar[l]^-{k_*} & K_1(C^*_r (F_3(\beta_1, \beta_2, \beta_3))) \ar[l]^-{w_2}
 }
\end{equation}

We claim that the maps $w_2=i_*\circ (id_*-\varphi(\alpha_2^{-1})_*)$ are equal to zero. On $K_0$, we have $\varphi(\alpha_2^{-1})_*([1])=[1]=id_*([1])$, so  $w_2([1])=0$; 
and on $K_1$, for every $j\in \{1,2,3\}$ we have $\varphi(\alpha_2^{-1})_*([u_{\beta_j}])=[u_{\varphi(\alpha_2^{-1})(\beta_j)}]=[u_{\alpha_2^{-1}\beta_j \alpha_2}]$. 
By the braid group relations 4., 5. and 6. in Remark \ref{obsv:P4}, we have $\alpha_2^{-1}\beta_j \alpha_2=f_\beta\beta_j f_\beta^{-1}$, where $f_\beta$ is a product of elements in $\{\beta_1,\beta_2, \beta_3\}$. Hence \[[u_{\alpha_2^{-1}\beta_j \alpha_2}]=[u_{f_\beta\beta_j f_\beta^{-1}}]=[u_{\beta_j}] \;\; \text{in} \;\;K_1(C^*(F(\beta_1, \beta_2, \beta_3))\rtimes_r F(\alpha_1)),\] so  $w_2([u_{\beta_j}])=0$ for all $j$, and the claim is proved.

From \eqref{eq:PV32}, we get two short exact sequences:
\begin{eqnarray}
\label{eq:B_1}
 \hspace{4mm}
0\to   K_0(C^*_r(F_3(\beta_1, \beta_2, \beta_3))\rtimes_r F(\alpha_1)) \to  K_0(B)\to  K_1(C^*_r(F_3(\beta_1, \beta_2, \beta_3))\to 0
    \\
    0\to   K_1(C^*_r(F_3(\beta_1, \beta_2, \beta_3))\rtimes_r F(\alpha_1)) \to  K_1(B)\to  K_0(C^*_r(F_3(\beta_1, \beta_2, \beta_3))\to 0
\end{eqnarray}
We now use another PV sequence to compute $K_i(C^*_r(F_3(\beta_1, \beta_2, \beta_3))\rtimes_r F(\alpha_1)) $ for $i=0,1$.
Set $B_1:=C^*_r(F_3(\beta_1, \beta_2, \beta_3))\rtimes_r F(\alpha_1)$. We have from \eqref{eq:PVFree}that
\begin{equation}
\label{eq:PV31}
\xymatrix{
K_0(C^*_r(F_3(\beta_1, \beta_2, \beta_3)))\ar[r]^\theta & K_0(C^*_r(F_3(\beta_1, \beta_2, \beta_3))) \ar[r] & K_0(B_1) \ar[d]  
\\
 K_1(B_1)\ar[u] &
 K_1(C^*_r (F_3(\beta_1, \beta_2, \beta_3)))\ar[l] &
  K_1(C^*_r (F_3(\beta_1, \beta_2, \beta_3)))\ar[l]^\theta
 }
\end{equation}
where $\theta=id_*-\varphi(\alpha^{-1}_1)_*$ is again the trivial map in $K$-theory, because by relations 1., 2. and 3. in Remark \ref{obsv:P4}, we have that $\alpha_1^{-1}\beta_j \alpha_=f_\beta\beta_j f_\beta^{-1}$, where $f_\beta$ is a product of elements in $\{\beta_1,\beta_2, \beta_3\}$, as above.

We get the short exact sequences
\[
\xymatrix{
0\ar[r] & 
\mathbb{Z}\ar[r] &
 K_0(B_1)\ar[r]&
 \mathbb{Z}^3\ar[r] &
 0,  
 }
\]
\[
\xymatrix{
0\ar[r] & 
\mathbb{Z}^3\ar[r] &
 K_1(B_1)\ar[r]&
 \mathbb{Z} \ar[r] &
 0,  
 }
\]
from which we deduce 
$$
K_0(B_1)\simeq K_1(B_1)\simeq \mathbb{Z}^4\ .
$$
By inserting this in \eqref{eq:B_1}, we get
\[
    \xymatrix{
    0\ar[r] & 
    \mathbb{Z}^4\ar[r] &
     K_0(B)\ar[r]&
     \mathbb{Z}^3\ar[r] &
     0,  
    }
\]
\[
\xymatrix{
0\ar[r] & 
\mathbb{Z}^4\ar[r] &
 K_1(B)\ar[r]&
 \mathbb{Z}\ar[r] &
 0,  
 }
\]
so we deduce
\begin{equation}
K_0(B)\simeq \mathbb{Z}^7, \qquad \qquad K_1(B)\simeq \mathbb{Z}^5,
\end{equation}
which proves the proposition.
\end{proof}

By applying the K\"unneth formula to $P_4=B\times \mathbb{Z}$, we obtain the following corollary.
\begin{Cor}
The $K$-theory of $C^*_r(P_4)$ is given by 
\[
    K_0(C^*_r(P_4))\simeq \mathbb Z^{12},\qquad \qquad K_1(C^*_r(P_4))\simeq \mathbb Z^{12}. 
\]
\end{Cor}

\subsection{Going from $P_{n-1}$ to $P_n$}
\label{sec:P5}

Before computing the $K$-theory of $C^*_r(P_n)$ for all $n$, let us explain how to compute the $K$-theory of $C^*_r (P_5)$. Recall that $P_n=F_{n-1}\rtimes P_{n-1}$, so that $P_5=F_4\rtimes P_4=F_4\rtimes F_3\rtimes F_2 \times \mathbb{Z}$.  We first compute the $K$-theory of $C^*_r (F_4\rtimes F_3\rtimes F_2)$. By Lemma \ref{lem:crossedproduct}, we have 
\[
C^*_r (F_4\rtimes F_3\rtimes F_2)\simeq (C^*_r(F_4\rtimes F_3))\rtimes_r F_2,\
\]
so the $K$-theory of $C^*_r (P_5)$ can be computed via a Pimsner--Voiculescu sequence once we know 
the $K$-theory of $C^*_r(F_4\rtimes F_3)$, which is isomorphic to  $C^*_r(F_4)\rtimes F_3$. Notice that, as $F_n$ is $K$-amenable for all $n$, we have $K_i(C^*_r(F_4)\rtimes F_3)\simeq K_i(C^*_r(F_4)\rtimes_r F_3)$, for $i=0,1$.\\

We compute the $K$-theory groups $K_i(C^*_r(F_4)\rtimes_r F_3)$ via  Pimsner--Voiculescu sequences. 
%\paragraph{K-theory of $C^*_r(F_4)\rtimes_r F_3$}

Letting $n=3$ in  \eqref{eq:PVFree}, and using \eqref{eq:K0CmF} and \eqref{eq:K1CmF}, we get the sequence below.
\[
\xymatrix{
\mathbb{Z}^3\ar[r]^\theta & \mathbb{Z} \ar[r] & K_0(C^*_r(F_4)\rtimes_r F_3) \ar[d]  \\
 K_1(C^*_r(F_4)\rtimes_r F_3)\ar[u] & \mathbb{Z}^4 \ar[l] & (\mathbb{Z}^4)^3 \ar[l]^\theta
 }.
\]
Since again $\theta=0$ on both $K_0$ and $K_1$,
we have
$$
K_0(C^*_r(F_4)\rtimes_r F_3)\simeq \mathbb{Z}^{13}, \qquad \qquad K_1(C^*_r(F_4)\rtimes_r F_3)\simeq \mathbb{Z}^{7}.
$$
To compute the $K$-theory of $(C^*_r(F_4)\rtimes_r F_3)\rtimes_r F_2$, we apply  \eqref{eq:PVFree}
\[
\xymatrix{
(\mathbb{Z}^{13})^2\ar[r]^\theta & \mathbb{Z}^{13} \ar[r] & K_0(C^*_r(F_4)\rtimes_r F_3)\rtimes_r F_2) \ar[d]  \\
 K_1(C^*_r(F_4)\rtimes_r F_3)\rtimes_r F_2\ar[u] & \mathbb{Z}^7 \ar[l] & (\mathbb{Z}^7)^2 \ar[l]^\theta
 }
\]
and we get 
$$
K_0(C^*_r(F_4)\rtimes_r F_3)\rtimes_r F_2)\simeq \mathbb{Z}^{27}, \qquad \qquad K_1(C^*_r(F_4)\rtimes_r F_3)\rtimes_r F_2)\simeq \mathbb{Z}^{33}\ .
$$
By the K\"unneth formula, we finally obtain
$$
K_0(C^*_r(F_4)\rtimes_r F_3)\rtimes_r F_2\times \mathbb{Z})\simeq K_1(C^*_r(F_4)\rtimes_r F_3)\rtimes_r F_2\times\mathbb{Z}) \simeq \mathbb{Z}^{33}\oplus \mathbb{Z}^{27}\simeq \mathbb{Z}^{60}\ .
$$
%\subsection{$K$-theory of $C^*(F_n)\rtimes_r F_{n-1}$}
Let us now generalise this procedure to compute the $K$-theory of $C^*_r(F_n)\rtimes_r F_{n-1}$ in the following lemma.
\begin{Lem}
The $K$-theory of $C^*_r(F_n)\rtimes_r F_{n-1}$ is given by
$$
K_0(C^*_r(F_n)\rtimes_r F_{n-1})\simeq \mathbb{Z}^{1+n(n-1)}, \qquad \qquad
K_1(C^*_r(F_n)\rtimes_r F_{n-1})\simeq \mathbb{Z}^{2n-1}.
$$
\end{Lem}
\begin{proof}
We apply \eqref{eq:PVFree} 
\begin{equation}
\xymatrix{
(K_0(C^*_r(F_n)))^{n-1}\ar[r]^\theta & K_0(C^*_r(F_n)) \ar[r] & K_0(C^*_r(F_n)\rtimes_r F_{n-1})) \ar[d]  \\
 K_1(C^*_r(F_n)\rtimes_r F_{n-1}))\ar[u] & K_1(C^*_r(F_n)) \ar[l] & (K_1(C^*_r(F_n)))^{n-1} \ar[l]_\theta
 }
\end{equation}
the maps $\theta$ are trivial in $K$-theory, so for $i=0,1$ this gives 
\begin{eqnarray}
&K_0(C^*_r(F_n)\rtimes_r F_{n-1})\simeq K_0(C^*_r(F_n))\oplus (K_1(C^*_r(F_n)))^{n-1},  \\
 &K_1(C^*_r(F_n)\rtimes_r F_{n-1})\simeq K_1(C^*_r(F_n))\oplus (K_0(C^*_r(F_n)))^{n-1}\ .
\end{eqnarray}
%shorter
% \begin{equation}
% K_i(C^*(F_n)\rtimes_r F_{n-1})\simeq K_0(C^*(F_n))\oplus (K_1(C^*(F_n)))^{n-1} 
% \end{equation}
By the relations \eqref{eq:K0CmF}, we obtain the result.
\end{proof}

\subsection{$K$-theory of $C^*_r(P_n)$ }
We are now ready to compute the $K$-theory of $C^*_r(P_n)$ using Pimsner--Voiculescu sequences. We use the following remark. 
\begin{Rem}
\label{rem:free}
Let $A$ be a $C^*$-algebra whose $K$-theory is torsion-free and finitely generated.  This means that there exist two integers $a_0$ and $a_1$ such that  $K_0(A)\simeq \mathbb{Z}^{a_0}$ and $K_1(A)\simeq \mathbb{Z}^{a_1}$. If $\phi:F_k\to {\rm Aut}(A)$ is an action of the free group $F_k$ by automorphisms on $A$ such that the induced map $\theta$ in the PV sequence is zero, then we obtain  
$$
K_0(A\rtimes_r F_{k})\simeq \mathbb{Z}^{a_0+ka_1},
\qquad \qquad
K_1(A\rtimes_r F_{k})\simeq \mathbb{Z}^{a_1+ka_0}.
$$
In particular, the $K$-theory groups of $A\rtimes_r F_{k}$ are also torsion-free. 
\end{Rem}

\begin{Prop}
The $K$-theory of $C^*_r(P_n)$ is given by
$$
K_0(C^*_r(P_n))\simeq K_1(C^*_r(P_n))\simeq \mathbb{Z}^{\frac{n!}{2}}.
$$
\end{Prop}
\begin{proof}

We set $Q_{j}:= F_{n-1}\rtimes F_{n-2}\rtimes\dots\rtimes F_{n-j}$, for $j=1,\dots , n-2$.
%\begin{eqnarray*}
%\begin{aligned}
%&Q_1:=F_{n-1}\\
%&Q_2:= F_{n-1}\rtimes F_{n-2}\\
%&\dots\\
%&Q_{n-2}:= F_{n-1}\rtimes %F_{n-2}\rtimes\dots\rtimes F_2\ .
%\end{aligned}
%\end{eqnarray*} 

We first show inductively that all the groups $K_*(C^*_r (Q_j))$ are torsion free. For $j=1$, this holds true;  at each subsequent step, the $K$-theory of $C^*_r (Q_j)\simeq C^*_r( F_{n-1}\rtimes F_{n-2}\rtimes\dots\rtimes F_{n-j+1})\rtimes_r F_{n-j}$ is computed  as in Remark \ref{rem:free} via a Pimsner-Voiculescu sequence involving the $K$-theory of the reduced $C^*$-algebra $C^*_r( F_{n-1}\rtimes F_{n-2}\rtimes\dots\rtimes F_{n-j+1})$. 
%By $K$-amenability, the latter is the same as the reduced $C^*$-algebra $K$-theory  
%$K_*(C_r^*( F_{n-1}\rtimes F_{n-2}\rtimes\dots\rtimes F_{n-j+1}))$. 
We repeatedly apply Pimsner--Voiculescu sequences \eqref{eq:PVFree} 
and we have
\begin{eqnarray}
\label{eq:Qi}
&K_0(C^*_r(Q_{i}))\simeq K_0(C^*_r(F_{n-i}))\oplus (K_1(C^*_r(Q_{i-1})))^{n-i}  \\
&K_1(C^*_r(Q_{i}))\simeq K_1(C^*_r(F_{n-i}))\oplus (K_0(C^*_r(Q_{i-1})))^{n-i}\ .
\end{eqnarray}
To determine the rank, set
\[
x_i:={\rm rank }\,K_0(C^*_r (Q_i)), \qquad \qquad
y_i:={\rm rank} \,K_1(C^*_r (Q_i)) .
\]
From \eqref{eq:Qi}, we have for $i=2, \dots, n-2$
\begin{eqnarray*}
x_i={\rm rank }\, K_0(C^*_r (Q_{i-1})\rtimes_r F_{n-i})=x_{i-1}+y_{i-1}(n-i),\\
y_i={\rm rank} \,K_1(C_r^* (Q_{i-1})\rtimes_r F_{n-i})=y_{i-1}+x_{i-1}\cdot (n-i).
\end{eqnarray*}
with $x_1=1$ and $y_1=n-1$.
This implies
$$
x_i+ y_i= x_{i-1}+y_{i-1} + (x_{i-1}+y_{i-1})(n-i)\ .
$$
Note that $P_n=Q_{n-2}\times \mathbb{Z}$; and by the K\"unneth formula, we have 
${\rm rank}\, K_*(C^*_r (P_n))=x_{n-2}+y_{n-2}$.
The sum $s_i:=x_i+ y_i $ satisfies
$$
s_i=s_{i-1}(n-i+1)
$$
hence we deduce 
\begin{align*}
s_{n-2}=&s_{n-3}(n-n+2+1)\\
&=s_2\cdot (n-2)(n-3)\cdot\dots 4\cdot 3\\
&=n(n-1)\cdot\dots 4\cdot 3\\
&=\frac{n!}{2}.
\end{align*}
\end{proof}
As the group $P_n$ is $K$-amenable (see section \ref{K-amenability}), the $K$-theory of the maximal $C^*$-algebra of $P_n$ coincides with the reduced one, see for instance, \cite[Corollary 3.6]{JV}:
If $G$ is $K$-amenable, then for every $C^*$-dynamical system $(A,\alpha, G)$ one has
\begin{equation}
\label{JV-K-am}
   K_i(A\rtimes_{\alpha, r}G)\simeq K_i(A\rtimes_{\alpha}G), \qquad i=0,1\ .
\end{equation}
Since free groups are $K$-amenable, by Lemma   \ref{lem:crossedproduct} one has immediately:
\label{lem:K-am}
\begin{enumerate}
    \item $K_i(C^*(P_n))\simeq K_i(C^*_r(P_n))$
    \item For the iterated semidirect products
\begin{eqnarray*}
\begin{aligned}
&Q_1:=F_{n-1}\\
&Q_2:= F_{n-1}\rtimes F_{n-2}\\
&\vdots\\
&Q_{n-2}:= F_{n-1}\rtimes F_{n-2}\rtimes\dots\rtimes F_2\ .
\end{aligned}
\end{eqnarray*}
the $K$-theory of the maximal $C^*$-algebra is the same as the reduced one.  
\end{enumerate}
Therefore
$$
K_0(C^*(P_n))\simeq K_1(C^*(P_n))\simeq \mathbb{Z}^{\frac{n!}{2}}.
$$

\section{Isomorphism of the Baum--Connes assembly map}
\label{Isomorphism}

Let us denote by $\Gamma=F(\beta_1, \beta_2, \beta_3)\rtimes F(\alpha_1, \alpha_2)$ the group generated by $\alpha_1, \alpha_2, \beta_1, \beta_2, \beta_3$ satisfying the $6$ relations described in section \ref{6relations}. In this section, we use the structure of $P_4$ given by $P_4\simeq \Gamma\times\Z$, where $\Gamma=F_3\rtimes F_2$, and the K\"unneth theorem in $K$-theory, to reduce the proof of the Baum--Connes isomorphism for $P_4$ to that of $\Gamma$. We will use the explicit computations given in \cite{valbc} of the Baum--Connes assembly map in small homological degree to write down the explicit image for $\Gamma$ under the Baum--Connes assembly map. We will prove the following theorem.

\begin{Thm}
The assembly map
\[
\mu_r: K_i(B\Gamma)\rightarrow K_i(C^*_r(\Gamma))
\]
is an isomorphism for $i=0$ (even-degree case) and $i=1$ (odd-degree case). 
\end{Thm}

\subsection{Odd-degree Baum--Connes isomorphism for $\Gamma$}
We will start with the case $i=1$. 
Recall that $H_1(\Gamma, \Z)=\Gamma/[\Gamma, \Gamma]$ so that $H_1(\Gamma, \Z)$ is the abelian group generated by  $\beta_1, \beta_2, \beta_3, \alpha_1, \alpha_2.$ There is a classical isomorphism $H_1(\Gamma, \Z)\simeq H_1(B\Gamma)$, where every generator of $H_1(\Gamma, \Z)$ corresponds to a unique $1$-cycle coming from the 1-skeleton of the space $B\Gamma$. The correspondence is determined by the fact that $\Gamma=\pi_1(B\Gamma)$: Every element $\gamma\in \Gamma$ can be viewed as a pointed continuous map $\gamma:S^1\to B\Gamma$, thus inducing a map in $K$-homology $\gamma_*:K_1(S^1)\to K_1(B\Gamma)$. Let $D$ be the Dirac operator on $S^1$ and $\pi$ the representation of $C(S^1)$ on $L^2(S^1)$ given by pointwise multiplication. Then the class of the cycle ($\pi,D)$ is the generator of $K_1(S^1)\simeq \mathbb{Z}$; we denote it by $[D]$. Every element $\gamma\in \Gamma$ can then be mapped to the class of the cycle $\gamma_*(\pi,D)=\gamma_*([D])$ in $K_1(B\Gamma)$ (see \cite[Chapter 7]{valbc}). Moreover, every element $\gamma\in \Gamma$ can be mapped to the invertible element $[\gamma]\in K_1(C^*_r(\Gamma)),$ which is determined by the class of the Dirac element $\delta_{\gamma}$ in $C_c(\Gamma)$.  

We are going to prove the following theorem.
\begin{Thm}
\label{thm1}
Let $\Gamma$ be the group $F(\beta_1, \beta_2, \beta_3)\rtimes F(\alpha_1, \alpha_2)$ generated by the elements  $\alpha_1$, $\alpha_2$, $\beta_1$, $\beta_2$, $\beta_3$ that satisfy the $6$ relations described in Section \ref{6relations}.
The Baum-Connes assembly map
\[
\mu_r: K_1(B\Gamma)\rightarrow K_1(C^*_r(\Gamma))
\]
is an isomorphism with
\begin{align*}
&\mu_r((\alpha_i)_*[D])=[\alpha_i], \qquad i=1,2;\\
&\mu_r((\beta_i)_*[D])=[\beta_i], \qquad i=1,2,3,
\end{align*}
where $[D]$ is the $K$-homology cycle given by the Dirac operator on the unit circle $S^1$.
\end{Thm}
Following \cite{valbc} (see Chapter 7), define a morphism $\tilde\beta_a: \Gamma\rightarrow K_1(C^*_r(\Gamma))$ by sending $\gamma\in\Gamma$ to the invertible element $[\gamma]$ in $K_1(C^*_r(\Gamma)).$
The map $\tilde\beta_a$ gives rise to a well-defined morphism
\[
\beta_a: H_1(\Gamma,\Z)\rightarrow K_1(C^*_r(\Gamma)), \qquad \gamma[\Gamma, \Gamma]\mapsto [\gamma],
\]
because $K_1(C^*_r(\Gamma))$ is an abelian group. 
%Let $D$ be the Dirac operator on $S^1$. Pointwise multiplication by elements of $C(S^1)$ on $L^2(S^1)$ gives rise to an element in $K_1(S^1).$
The Dirac operator $D$ on $S^1$ gives rise to a class which generates $K_1(S^1).$
For every group element $\gamma\in \Gamma,$ denote by $\gamma: S^1\rightarrow B\Gamma$ a $1$-cycle representative for $[\gamma]\in H_1(B\Gamma)$.

Define a morphism $\tilde\beta_t: \Gamma\rightarrow K_1(B\Gamma)$ by sending $\gamma$ to $\gamma_*[D]$.
By \cite[Proposition 7.1]{valbc}, $\tilde\beta_t$ descends to
\[
\beta_t: H_1(\Gamma, \Z)\rightarrow K_1(B\Gamma), \qquad [\gamma]\mapsto\gamma_*[D],
\]
and we have the following lemma due to Natsume, which we will use to prove Theorem~\ref{thm1}.
\begin{Lem}[\cite{Natsume}, \cite{valbc} Proposition 7.2]
\label{Lem:Natusme}
The following diagram commutes.
\[
\xymatrix{  & H_1(\Gamma, \Z) \ar[dl]_{\beta_t} \ar[dr]^{\beta_a} &     \\
            K_1(B\Gamma) \ar[rr]^{\mu_r}&   & K_1(C^*_r(\Gamma))
                      }
\]
That is, $\beta_a=\mu_r\circ\beta_t.$
\end{Lem}

\begin{proof}[Proof of Theorem \ref{thm1}]
We are going to use the fact that there is an assembly map $$\mu:K_i(B\Gamma)\to K(C^*(\Gamma))$$ such that $\mu_r=\lambda_{\Gamma,*}\circ\mu$ where the morphism $\lambda_{\Gamma,*}:K_i(C^*(\Gamma))\to K(C^*_{r}(\Gamma))$ is induced by the regular representation $\lambda_{\Gamma}$ of $\Gamma$. We will then use the $K$-amenability of $P_n$ which implies that $\lambda_{\Gamma,*}$ is an isomorphism. The advantage of $\mu$ with respect to $\mu_r$ is that $\mu$ is functorial in $\Gamma$ and we will make use of its functoriality. Consider the group homomorphism $\psi: \Gamma\rightarrow \Gamma_{ab},$ where $\Gamma_{ab}=\Gamma/[\Gamma, \Gamma]\simeq\Z^5$ is generated by the cosets of $\alpha_1, \alpha_2, \beta_1, \beta_2, \beta_3.$
This map induces a continuous map $\psi: B\Gamma\rightarrow B\Z^5$ and a morphism of group $C^*$-algebras $\psi: C^*(\Gamma)\rightarrow C^*(\Z^5).$
By the functoriality of the Baum--Connes assembly map at the level of the maximal $C^*$-algebra, we have the following commutative diagram
\[
\xymatrix{  K_1(B\Gamma)\ar[r]^{\mu}\ar[d]_{\psi_{*}} &  K_1(C^*(\Gamma))  \ar[d]_{\psi_{*}}   \\
            K_1(B\Z^5) \ar[r]^{\mu'} &  K_1(C^*(\Z^5))
                      }
\]
where $\mu$ and $\mu'$ are the Baum-Connes assembly map defined at the level of the full $C^*$-algebra for $\Gamma$ and $\mathbb{Z}^5$.
From our calculation of $K_1(B\Gamma)$ (see Lemma \ref{Lem:X}), we see that $\psi_*: K_1(B\Gamma)\rightarrow K_1(B\Z^5)$ is an isomorphism.
In fact, the following diagram is commutative by definition
\[
\xymatrix{  H_1(\Gamma, \Z)\ar[r]\ar[d]_{\simeq} &  K_1(B\Gamma)  \ar[d]_{\psi_{*}}   \\
            H_1(\Z^5, \Z) \ar[r]^{\simeq} &  K_1(B\Z^5)
                      }
\]
and the left and bottom arrows are isomorphisms. This shows that $\psi_*$ is a surjective morphism. But $\psi_*$ is a surjective morphism from $\Z^5$ to itself. So $\psi_*$ is an isomorphism on $K$-homology.

As the Baum--Connes conjecture is known to be true for abelian groups, the map $\mu'$ is an isomorphism; so using the commutativity of the diagram, we get that the map $\psi_*: K_1(C^*(\Gamma))\rightarrow K_1(C^*(\Z^5))$ is surjective. On the other hand, we have computed that $$K_1(C^*(\Gamma))\simeq K_1(C^*(\Z^5))\simeq\Z^5;$$ therefore, as a surjective homomorphism from $\Z^5$ to itself is an isomorphism, the map $\psi_*$ is an isomorphism on $K$-theory as well, and hence $\mu$ is an isomorphism.

By commutativity of the diagram in Lemma~\ref{Lem:Natusme}, the map $\beta_a$ is surjective; and, being a surjective morphism from $\Z^5$ to itself, it is an isomorphism. Thus, $\beta_a$ is an isomorphism, mapping $\alpha_i[\Gamma, \Gamma]$ to $[\alpha_i]$ and $\beta_j[\Gamma, \Gamma]$ to $[\beta_j]$.
Because $\mu$ is an isomorphism, we have that $\beta_t$ is an isomorphism, mapping $[\alpha_i: S^1\rightarrow B\Gamma]\in H_1(B\Gamma)\simeq H_1(\Gamma, \Z)$ to $(\alpha_i)_*[D]\in K_1(B\Gamma).$
Therefore we know that the generators of $K_1(B\Gamma)$ are of the form $(\alpha_i)_*[D],$ $i=1,2;$ or $(\beta_j)_*[D],$ $j=1,2,3.$
The commutativity of the diagram of Lemma~\ref{Lem:Natusme} then implies that
\begin{align*}
&\mu((\alpha_i)_*[D])=[\alpha_i],\qquad i=1,2;\\
&\mu((\beta_i)_*[D])=[\beta_i],\qquad i=1,2,3.
\end{align*}
The theorem is then proved by noting the $K$-amenability of $\Gamma$ and applying $\lambda_{\Gamma,r}$ to get the elements of $C^*_r(\Gamma)$.
\end{proof}

\subsection{Even-degree Baum--Connes isomorphism for $\Gamma$}

Recall that $\Gamma$ is the group $F(\beta_1, \beta_2, \beta_3)\rtimes F(\alpha_1, \alpha_2)$  whose generators $\alpha_1, \alpha_2, \beta_1, \beta_2, \beta_3$ satisfy the $6$ relations of Section \ref{6relations}.
Each relation $R_i$ corresponds to a surface $\Sigma_i$ whose fundamental group is canonically related to $R_i$ as follows. 
\begin{align*}
 \pi_1(\Sigma_1)&=\langle a_1, a_2, a_3, a_4 \; | \;  a_1a_2a_1^{-1}=(a_3a_4)^{-1}a_4(a_3a_2)\rangle,\\
\pi_1(\Sigma_2)&=\langle b_1, b_2, b_3, b_4 \; | \;  b_1b_2b_1^{-1}=b_3^{-1}b_4b_3b_4^{-1}b_2\rangle,\\
\pi_1(\Sigma_3)&=\langle c_1, c_2 \; | \; c_1c_2c_1^{-1}=c_2\rangle,\\
\pi_1(\Sigma_4)&=\langle d_1, d_2, d_3, d_4\; | \; d_1d_2d_1^{-1}=(d_3d_4)^{-1}d_4(d_3d_2)\rangle,\\
\pi_1(\Sigma_5)&=\langle e_1, e_2, e_3, e_4, e_5, e_6 \; | \; e_1e_2e_1^{-1}=(e_6e_5)^{-1}e_5e_6e_2(e_3e_4)^{-1}(e_4e_3)\rangle,\\
\pi_1(\Sigma_6)&=\langle f_1, f_2, f_3, f_4 \; | \;  f_1f_2f_1^{-1}=f_3^{-1}f_4f_3f_4^{-1}f_2\rangle.
\end{align*}
Let $\Gamma_i=\pi_1(\Sigma_i),$ and set $\tilde\Gamma:=\Gamma_1*\Gamma_2*\cdots*\Gamma_6$, the free product of the $\Gamma_i$.
By the van Kampen theorem, the group $\tilde\Gamma$ is the fundamental group of
\[
\Sigma:=\Sigma_1\vee\Sigma_2\vee\cdots\vee\Sigma_6
\]
obtained by joining together a base point from each of the $\Sigma_i$.
Then the mapping defined by
\begin{align*}
\begin{split}
a_1, b_1, c_1 & \mapsto \alpha_1\\
d_1, e_1, f_1 & \mapsto \alpha_2\\
a_2, a_4, b_3, d_2, d_4, e_3, e_5, f_3& \mapsto \beta_1
\end{split}
\begin{split}
a_3, b_2, b_4, e_2 & \mapsto \beta_2\\
c_2, d_3, e_4, e_6, f_2, f_4 & \mapsto \beta_3
\end{split}
\end{align*} 
sends relations of $\tilde\Gamma$ to relations of $\Gamma$, and it determines a surjective morphism
\begin{equation}
\label{eq:fgroup}
f: \tilde\Gamma\rightarrow \Gamma.
\end{equation}
The map $f$ also leads to the morphism of full group $C^*$-algebras below.
\[
f: C^*(\tilde\Gamma)\rightarrow C^*(\Gamma).
\]

Let $X=B\Gamma$, and denote by $X_1, \ldots, X_6$ the $2$-simplices of $X$ corresponding to each $R_i.$ The union of the $X_i$ is the $2$-skeleton $X^{(2)};$ and since $X$ is a space of dimension $2,$ we have $X^{(2)}=X$.
The map $f$ in (\ref{eq:fgroup}) induces a map at the level of the classifying spaces $f: B\tilde\Gamma\rightarrow B\Gamma=X.$ Taking the $2$-skeleton, we obtain a continuous map
\begin{equation}
\label{eq:fspace}
f: \Sigma\rightarrow X
\end{equation}
such that $f(\Sigma_i)=X_i.$
Applying the fundamental group functor for (\ref{eq:fspace}) recovers
\[
f: \pi_1(\Sigma)\rightarrow\pi_1(X)
\]
 in (\ref{eq:fgroup}).
 By the functoriality of the Baum--Connes assembly map, we have the following commutative diagram.
$$\xymatrix{  K_0(\Sigma) \ar[r]^{\mu} \ar[d]_{f_*} & K_0(C^*(\tilde\Gamma)) \ar[d]_{f_*}   \\
            K_0(B\Gamma) \ar[r]^{\mu}&    K_0(C^*(\Gamma)).
                      }
$$
Note that the map in (\ref{eq:fspace}) gives rise to the two isomorphisms
 $$H_0(\Sigma)\simeq H_0(B\Gamma)\simeq\Z$$ and 
\begin{align*}
f_*: H_2(\Sigma)& \to H_2(B\Gamma), \qquad 
[\Sigma_i]\mapsto[X_i].
\end{align*}
The existence of an inverse Chern character map
\[
H_{even}(\Sigma)=H_0(\Sigma)\oplus H_2(\Sigma)\rightarrow K_0(\Sigma)
\]
sending the generator $[\Sigma_i]\in H_2(\Sigma)$ to $[D_{\Sigma_i}]\in K_0(\Sigma)$, where $D_{\Sigma_i}$ is the Dirac operator on $\Sigma_i$, which induces an isomorphism at the level of  $K$-homology classes, allows one to construct a morphism
\[
\beta_t: H_{even}(B\Gamma)\simeq H_{even}(\Sigma)\rightarrow K_0(\Sigma)\xrightarrow{f_*} K_0(B\Gamma),
\]
taking the composition with $f_*$ on $K$-homology,

By construction,
\begin{equation}
\label{eq:betat}
\beta_t([X_i])=f_*[D_{\Sigma_i}].
\end{equation}
The map $\beta_t$ is part of the lower-left of diagram
\[
\xymatrix{ H_{even}(B\Gamma)\ar@{.>}@/_3pc/[rrrd]_{\beta_a}\ar@{.>}[rrd]_{\beta_t}& H_{even}(\Sigma)\ar[r]^{\simeq}\ar[l]^{\simeq} & K_0(\Sigma) \ar[r]^{\mu} \ar[d]_{f_*} & K_0(C^*(\tilde\Gamma)) \ar[d]_{f_*}   \\
           & & K_0(B\Gamma) \ar[r]^{\mu_r}&    K_0(C^*_r(\Gamma)).
                      }
\]
The upper-right of this diagram is defined to be $\beta_a$; by definition,
(see \cite{valbc}) \[
\beta_a([X_i])=f_*(\mu([D_{\Sigma_i}])).
\]
The commutativity of this diagram is implied by the following lemma.

\begin{Lem}[\cite{valbc} Prop.~7.3]
\label{Lem:BM}
The diagram commutes with $\beta_t$ rationally injective:
\[
\xymatrix{  & H_{even}(\Gamma, \Z) \ar[dl]_{\beta_t} \ar[dr]^{\beta_a} &     \\
            K_0(B\Gamma) \ar[rr]^{\mu_r}&   & K_0(C^*_r(\Gamma))
                      }
\]
That is, $\beta_a=\mu_r\circ\beta_t.$
\end{Lem}

Now, from our calculations (see Lemma \ref{Lem:X}), we have,
\[
H_{even}(B\Gamma)\simeq K_0(B\Gamma)\simeq\Z^7.
\]
Knowing that $\beta_t$ maps generators to generators from (\ref{eq:betat}), and in view of the fact that $\beta_t$ is rationally injective (see \cite[Proposition 7.3]{valbc}, $\beta_t$ is an isomorphism.
The commutativity $\beta_a=\mu_r\circ\beta_t$ hence implies that $\beta_a$ is an isomorphism if and only if $\mu_r$ is an isomorphism.

Since $\beta_a=\mu_r\circ\beta_t$, we can also describe the map $\mu_r$ explicitly:
\begin{align*}
\mu_r: K_0(B\Gamma)&\to K_0(C^*_r(\Gamma)) \\ f_*[D_{\Sigma_i}]&\mapsto f_*\mu([D_{\Sigma_i}])
\end{align*}

We are ready to prove the following theorem.

\begin{Thm}
\label{thm2}
The assembly map
\[
\mu_r: K_0(B\Gamma)\rightarrow K_0(C^*_r(\Gamma))
\]
is an isomorphism, with
\begin{align*}
\mu_r(f_*[D_{\Sigma_i}])&=f_*(\mu[D_{\Sigma_i}]), \qquad 1\le i\le 6 \\ 
\mu_r(1)&=[1].
\end{align*}
\end{Thm}

\begin{proof}
Consider the trivial homomorphism $1\colon \Gamma\rightarrow\{e\}$. It induces a map $B\Gamma\rightarrow\{\mathrm{pt}\}$ and a map $C^*(\Gamma)\rightarrow\C$ such that 
the $K$-homology and $K$-theory functor lead to two morphisms 
$1_*:K_0(B\Gamma)\rightarrow K_0(\mathrm{pt})$ and $1_*:K_0(C^*(\Gamma))\rightarrow K_0(\C).$
The first morphism in $K$-homology is a surjective map capturing the $0$-simplex of $B\Gamma$.
The functoriality of the Baum--Connes assembly map gives rise to the commutative diagram
\[
\xymatrix{  K_0(B\Gamma)\ar[r]^{\mu}\ar[d]_{1_{*}} &  K_0(C^*(\Gamma))  \ar[d]_{1_{*}}   \\
            K_0(\mathrm{pt}) \ar[r]^{\mu_0} &  K_0(\C)
                      }
\]
where $\mu_0$ is the identity map from $\Z$ to itself.
Let $i=1$ or $2$, and let $j=1,2,$ or $3$. Denote by $\phi_{ij}$ the surjective morphism given by
\[
\phi_{ij}: \Gamma\rightarrow\Z^2 \qquad \alpha_p\mapsto\delta_{pi}\alpha_i, \quad \beta_p\mapsto\delta_{pj}\beta_j,
\]
where $\delta_{ij}$ is the Kronecker delta.
As above, it induces two maps 
\[
\phi_{ij}: B\Gamma\rightarrow B\Z^2\simeq T^2, \qquad \phi_{ij}: C^*(\Gamma)\rightarrow C^*(\Z^2),
\]
where we denote by $T^2$ the torus.  
Note that the map on the classifying space is given by collapsing all $1$-cells which do not represent $\alpha_i$ or $\beta_j$ to a point, and collapsing all $2$-cells that do not represent the group relation involving $\alpha_i\beta_j\alpha_i^{-1}$ (denoted $R_{ij}$) to a point.
Thus, the induced map on $K$-homology $\phi_{ij, *}: K_0(B\Gamma)\rightarrow K_0(T^2)$ is a surjective map that maps the $2$-cell represented by $R_{ij}$ to the Bott generator of $K_0(T^2)$.
As above, we also have the induced map on $K$-theory $\phi_{ij, *}: K_0(C^*(\Gamma))\rightarrow K_0(C^*(\Z^2))$  the commutative diagram
\[
\xymatrix{  K_0(B\Gamma)\ar[r]^{\mu}\ar[d]_{\phi_{ij,*}} &  K_0(C^*(\Gamma))  \ar[d]_{\phi_{ij,*}}   \\
            K_0(T^2) \ar[r]^{\mu_{ij}} &  K_0(C^*(\Z^2)).
                      }
\]
Here $\mu_{ij}$ is the assembly map for $\Z^2.$
Putting 7 diagrams (involving $1_*$ and $\phi_{ij,*}$ where $i=1,2$ and $j=1,2,3$) together, we have a commutative diagram
\[
\xymatrix{  K_0(B\Gamma)\ar[r]^{\mu}\ar[d]_{\phi_{*}} &  K_0(C^*(\Gamma))  \ar[d]_{\phi_{*}}   \\
           K_0(\mathrm{pt})\bigoplus\left[\bigoplus_{i,j}\tilde K_0(T^2)\right] \ar[r]^-{\mu'}_-{\simeq} & K_0(\C)\bigoplus\left[\bigoplus_{i,j}\tilde K_0(C^*(\Z^2))\right].
                      }
\]
Here, $\tilde K_0(T^2)$ is the reduced $K$-homology, excluding elements generated by the trivial cycle from $K_0(T^2),$ and $\tilde K_0(C^*(\Z^2))$ is the reduced $K$-theory, eliminating elements generated by the trivial projection from $K_0(C^*(\Z^2)).$
By construction, $\phi_*$ on $K$-homology (the left arrow) is an isomorphism. It is well known that $\mu'$ is an isomorphism for abelian groups $\Z^2$ and for the trivial group $\{e\}$.
Together with the commutativity of the diagram, the map $\phi_*$ on $K$-theory (the right arrow) is surjective. Because $\phi_*$ is a surjective group homomorphism from $\Z^7$ to itself, we conclude that $\phi_{*}$ on $K$-theory is an isomorphism.
Therefore the commutativity of the diagram implies that $\mu$ for $\Gamma$ is an isomorphism.
As the group $\Gamma $ is  $K$-amenable, we have $K_*(C^*(\Gamma))\simeq K_*(C^*_r(\Gamma)),$ hence we find 
that the assembly map $\mu_r$ is an isomorphism.
\end{proof}
\subsection{Isomorphism for $P_4$}
Let us now recover Oyono-Oyono's theorem for $P_4$ using the K\"unneth formula. 
\begin{Thm}
\label{thm:main}
The assembly map 
\[
\mu_r: K_i(BP_4)\rightarrow K_i(C^*_r(P_4))
\]
is an isomorphism for $i=0$ or $1$. 
\end{Thm}

\begin{proof}
Note that the isomorphism $P_4\simeq\Gamma\times\Z$ implies that 
\begin{align*}
K_i(BP_4)&\simeq K_0(B\Gamma)\otimes K_i(B\Z)\oplus K_1(B\Gamma)\otimes K_{i+1}(B\Z),\\
K_i(C^*_r(P_4))&\simeq K_0(C^*_r(\Gamma))\otimes K_i(C^*_r(\Z))\oplus K_1(C^*_r(\Gamma))\otimes K_{i+1}(C^*_r(\Z)).
\end{align*}
Following the definition of the assembly map in \cite{valbc} by twisting Mishchenko line bundles, we have $\mu_r^{P_4}(x\otimes y)=\mu_r^{\Gamma}(x)\otimes\mu_r^{\Z}(y)$ for $x\in K_i(B\Gamma)$ and $y\in K_j(B\Z)$, which are represented by Dirac-type operators. 
Then the assembly map for $P_4$ is an isomorphism if $\mu_r:K_i(B\Gamma)\rightarrow K_i(C^*_r(\Gamma))$ for $i= 0,$1 is an isomorphism. The theorem thus follows from Theorems~\ref{thm1} and \ref{thm2}.
\end{proof}

\subsection{Baum--Connes isomorphism for $P_n$}
\label{Sec5.3}
Let us now prove the following theorem, which is originally due to Oyono-Oyono.

\begin{Thm}[\cite{OO2001} Proposition 7.2]
\label{thm:maim.thm.BC}
The Baum--Connes assembly map for the pure braid group $P_n$
\[ \mu_r: K_i(BP_n)\rightarrow K_i(C^*_r(P_n))\] 
is an isomorphism.
\end{Thm}
For $k\in\{1, \ldots, n-1\}$, let 
\[
F_k:=F_k(A_{1, k+1}, A_{2, k+1}, \ldots, A_{k, k+1})
\]
be the free subgroup in $P_n$ (see section \ref{Sec2}).
There is a canonical homomorphism
\[
\rho: P_n\rightarrow F_1\times F_2\times\cdots\times F_{n-1}, \qquad A_{s,t}\mapsto (e,\ldots, e, A_{s,t}, e, \ldots, e)\,
\]
where $e$ is the identity element; 
here, $A_{s,t}\in F_{t-1}$.
In particular, all relations in the presentation for $P_n$ reduce to the form 
\[
\rho(A_{r,s})\rho(A_{i,j})=\rho(A_{i,j})\rho(A_{r,s}), \qquad i<j, \;\; r<s, \;\; s<j
\]
in the image.
The map $\rho$ induces maps between the classifying spaces and the $C^*$-algebras:
\[
B\rho: BP_n\rightarrow BF_1\times\cdots\times BF_{n-1}, \qquad \qquad \rho: C^*(P_n)\rightarrow C^*(F_1)\otimes\cdots\otimes C^*(F_{n-1}). 
\]
Consider the induced maps on $K$-homology and $K$-theory. By the functoriality of the Baum--Connes assembly map at the level of the maximal $C^*$-algebra, one has the following commutative diagram.
\begin{equation}
\label{eq:com}
\xymatrix{   K_i(BP_n)\ar[r]^{\mu}\ar[d]_{{B\rho}_{*}} &   K_i(C^*(P_n))  \ar[d]_{\rho_{*}}   \\
           K_i(BF_1\times\cdots \times BF_{n-1}) \ar[r]^-{\mu'}_-{\simeq}  & K_i(C^*(F_1)\otimes\cdots\otimes C^*(F_{n-1}))
                      }
\end{equation}
where $\mu'$ is an isomorphism because the groups $F_k$ and their  direct products have Haagerup's property (see \cite{OO2001-2}).
% where $\mu'$ is an isomorphism because the Baum--Connes conjecture is stable under direct products (see \cite{OO2001-2}) and it is an isomorphism for the free groups $F_k$ because they have Haagerup's property. 
Let us  describe the map $B\rho$ as a morphism between CW-complexes. 
For $0\le r\le n$, choose pairs of numbers $(i_k, j_k),$ where $k\in\{1, 2, \ldots, r\}$, satisfying
\begin{equation}
\label{eq:rel2}
1\le i_1<i_2<\cdots<i_r\le n ,\qquad 1\le j_k<i_k.
\end{equation}
It can be checked that every $r$-simplex of $BP_n$ depends uniquely on the pairs $(i_k, j_k),$ where $1\le k\le r$. 
Denote the $r$-simplex by $[A_{j_1, i_1}, \ldots, A_{j_r, i_r}]$.
Note that for a fixed $r$, the number of distinct $r$-simplices in $BP_n$ is equal to
\[
\sum_{1\le i_1<\cdots< i_r\le n}(i_1-1)\cdots(i_r-1)=a_r,
\]
which is the rank of the free part of $H^r(BP_n)$. 
Recall that $a_0+\cdots a_n=n!$, so that $BP_n$ has $n!$ simplices in total.
\begin{Ex}
The CW complex $BP_4$ has $1$ $0$-simplex; $6$ $1$-simplices $\alpha_1, \alpha_2, \beta_1, \beta_2, \beta_3$;
 $11$ $2$-simplices $R1, R2, R3, R4, R5, R6, c\times\alpha_1, c\times\alpha_2, c\times\beta_1, c\times\beta_2, c\times\beta_3$; and $6$ $3$-simplices $R_i\times c$ for $i\in\{1,\cdots, 6\}.$
\end{Ex}

The map $B\rho: BP_n\rightarrow BF_1\times\cdots\times BF_{n-1}$ is defined by sending the $r$-simplex $[A_{j_1, i_1}, \ldots, A_{j_r, i_r}]$ in $BP_n$
to the $r$-simplex 
\[
([A_{j_1, i_1}], \ldots, [A_{j_r, i_r}])\in BF_{i_1-1}\times\cdots  \times BF_{i_r-1}\subset BF_1\times\cdots  \times BF_{n-1}.
\]
Observe that $B\rho$ gives rise to an isomorphism 
\[
B\rho_*: H_i(BP_n)\rightarrow H_i(BF_1\times\cdots\times BF_{n-1}). 
\]
Because the Chern character maps 
\begin{align*}
&\mathrm{Ch}: K_{0/1}(BP_n)\rightarrow H_{even/odd}(BP_n)\\
&\mathrm{Ch}: K_{0/1}(BF_{i_r-1}\subset BF_1\times\cdots BF_{n-1})\rightarrow H_{even/odd}(BF_1\times\cdots  \times BF_{n-1})
\end{align*}
are isomorphisms, and by the functoriality of the Chern character,
we obtain an isomorphism on $K$-homology
\[
B\rho_*: K_i(BP_n)\rightarrow K_i(BF_1\times\cdots \times BF_{n-1}), \qquad i=0,1.
\] 
Because the Baum--Connes conjecture holds for free groups, we obtain that $\mu'$ in (\ref{eq:com}) is an isomorphism.
By the commutativity of (\ref{eq:com}), the map on $K$-theory 
\[
\rho_*: K_i(C^*(P_n))\rightarrow K_i(C^*(F_1)\otimes\cdots\otimes C^*(F_{n-1})) 
\]
is surjective. 
It is an easy exercise to compute that $K_i(C^*(F_1)\otimes\cdots\otimes C^*(F_{n-1}))\simeq \Z ^{\frac{n!}{2}}.$
Thus $\rho_*$ is a surjective morphism from $\Z^{\frac{n!}{2}}$ to itself. So $\rho_*$ is in fact an isomorphism. 
Therefore $\mu$ is an isomorphism, by the commutativity of the diagram (\ref{eq:com}). As $P_n$ is $K$-amenable (see section \ref{K-amenability}), Theorem \ref{thm:maim.thm.BC} is then proved.

\section{The full braid group on three stands $B_3$}
\label{Sec6}

In this section we consider full braid groups. The Baum--Connes correspondence for $B_n$  is known to be an isomorphisms by the work of Schick (\cite{Schick2007}). 
We provide the explicit description in the case $n=3$, modulo torsion. Note that in the paper \cite{BBV}, the authors had computed the $K$-theory of $C^{*}_r(B_3)$.

\subsection{$K$-homology of $BB_n$}
Modulo torsion, the $K$-homology of $BB_n$ is easier to compute using the rational isomorphism of the Chern character:
\[
\mathrm{Ch}: K_0(BB_n)\rightarrow \bigoplus_{i}H_{2i}(B_n, \Z), \qquad \qquad \mathrm{Ch}: K_1(BB_n)\rightarrow \bigoplus_{i}H_{2i-1}(B_n, \Z).
\]
Arnold computed the integral cohomology ring of the braid groups:
\[
H^0(B_n, \Z)\simeq\Z,\qquad \qquad H^1(B_n, \Z)\simeq\Z,
\]
and $H^i(B_n, \Z)$ is finite for $i>1$;
see~\cite{Arnold} and \cite{MR1634470}.
By Poincar\'e duality, we obtain the following result.

\begin{Prop}
Up to torsion,
\[
K_0(BB_n)\simeq K_1(BB_n)\simeq\Z.
\]
\end{Prop}

\begin{Rem}
As the referee pointed out, $B_3$ is a one relator group, its presentation complex has dimension $2$ and can be taken
as a model for its classifying space. So $K$-homology can be computed. See~\cite{BBV}.
\end{Rem}

\begin{Rem}
Calculating $K_*(BB_n)$ is a challenging task, since $K_*(BB_n)$ may contain torsion. For example, in Example~5.10 in~\cite{MR4218683}, the $K$-theory of the reduced group $C^*$-algebra of $B_4$ is computed to be
\[
K_0(C^*_r(B_4))\simeq\Z\oplus(\Z/2\Z), \qquad \qquad K_1(C^*_r(B_4))\simeq\Z.
\]
By the Baum-Connes isomorphism for the braid group, one knows that $K_0(BB_4)$ has torsion. 
\end{Rem}

\subsection{K-theory of $C^*_r(B_3)$}

Let $B_3=\langle\sigma_1, \sigma_2 | \sigma_1\sigma_2\sigma_1=\sigma_2\sigma_1\sigma_2 \rangle$ be the braid group on three strands.
The center of this group is generated by $(\sigma_1\sigma_2)^3=(\sigma_1\sigma_2\sigma_1)^2$. Let $x=\sigma_1\sigma_2\sigma_1$ and $y=\sigma_1\sigma_2$.
Then $B_3$ can be presented alternatively as
\[
B_3\simeq\langle x, y \; | \;  x^2=y^3\rangle,
\]
where $\langle x^2\rangle=\langle y^3\rangle=Z(B_3).$
Setting $G=\langle x\rangle, H=\langle y\rangle$ and $K=\langle x^2\rangle=\langle y^3\rangle$, then
\[
B_3=\langle x\rangle *_{Z(B_3)}\langle y\rangle=G*_KH.
\]
For an amalgamated free product, one has the following six-term exact sequence (See \cite{MR768305} Theorem A1).
\[
\xymatrix{
K_0(C^*_r(K)) \ar[r]^(.4){a } & K_0(C^*_r(G))\oplus K_0(C^*_r(H)) \ar[r]^(.6){d  } & K_0(C^*(B_3)) \ar[d]  \\
 K_1(C^*_r(B_3))\ar[u] & K_1(C^*_r(G))\oplus K_1(C^*_r(H)) \ar[l]_(.6){c  } & K_1(C^*_r(K)) \ar[l]_(.3){b  }
 }
\]
Note that $K_i(C^*_r(K))=K_i(C^*_r(G))=K_i(C^*_r(H))\simeq\Z$.
By definition,
\begin{align*}
a: \; & \Z\rightarrow\Z\oplus\Z, \qquad  a(x)=(x, x);\\
b: \; & \Z\rightarrow \Z\oplus\Z, \qquad b(x)=(2x, 3x).
\end{align*}
Thus $a$ and $b$ are injective, and then $c$ and $d$ are surjective. Therefore we have
\[
K_1(C^*_r(B_3))\simeq\Z\oplus\Z/\mathrm{im}(b)\simeq\Z,
\]
where the last isomorphism is due to the linear transformation
\[
\Z\oplus\Z\rightarrow\Z\oplus\Z, \qquad (x, y)\mapsto (3x-2y,-x+y)
\]
in $SL(2, \Z)$.
Similarly,
\[
K_0(C^*_r(B_3))\simeq\Z\oplus\Z/\mathrm{im}(a)\simeq\Z.
\]
Thus the following proposition is proved.

\begin{Prop}
\label{prop_} We have that
\begin{enumerate}
\item $K_0(C^*_r(B_3))\simeq\Z$ is generated by the unit of $C^*_r(B_3)$, and
\item $K_1(C^*_r(B_3))\simeq\Z$  is generated by $[\sigma_1]=[\sigma_2]$.
\end{enumerate}
\end{Prop}
\begin{Rem} 
The $K$-theory of $C^*_r(B_4)$ is computed in~\cite{MR4218683}, Example 5.10. At present, we are not aware of any direct method of computing $K_*(C^*_r(B_n))$ when $n\ge 5$.
\end{Rem}

The proof of the Baum--Connes isomorphism (rationally) for $B_3$ can be carried out analogous to Theorem~\ref{thm1} by considering the trivial morphism
$B_3\to\{e\}$,
 the quotient morphism
$
B_3\rightarrow B_3/[B_3, B_3]\simeq\Z$,
and these commutative diagrams:

\[
\xymatrix{  K_0(BB_3)\ar[r]^{\mu}\ar[d] &  K_0(C^*_r(B_3))  \ar[d]   \\
            K_0(B\{e\}) \ar[r]^{\mu'} &  K_0(C^*_r(\{e\}))
                      }
\qquad \qquad
\xymatrix{  K_1(BB_3)\ar[r]^{\mu}\ar[d] &  K_1(C^*_r(\Gamma))  \ar[d]   \\
            K_1(B\Z) \ar[r]^{\mu'} &  K_1(C^*_r(\Z))
                      }
\]
\begin{Thm}
The Baum--Connes assembly map
\[
K_i(BB_3)\rightarrow K_i(C^*_r(B_3)), \qquad i=0,1
\]
is an isomorphism rationally.
\end{Thm}

\section{Appendix}
In this appendix, we give some of diagrams that illustrate the structure of pure braid groups. 

\subsection{Generators of $P_4$} 
\phantom{.} 

\begin{figure}[H]
\setlength\columnsep{-1cm}

\begin{multicols}{3}
\centering
\includegraphics[height=0.15\textheight]{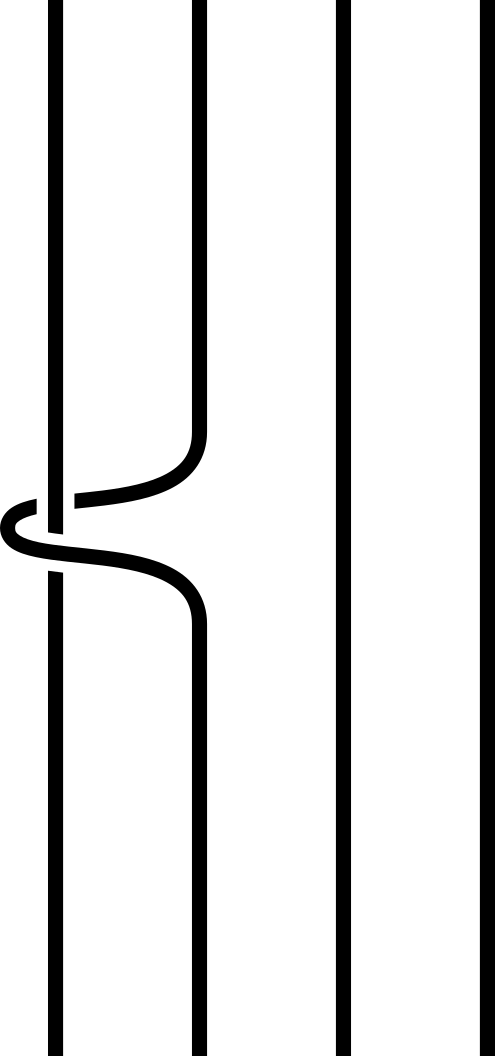}\\
\vspace{0.15cm}
\caption*{$\sigma_1^2$}
\includegraphics[height=0.15\textheight]{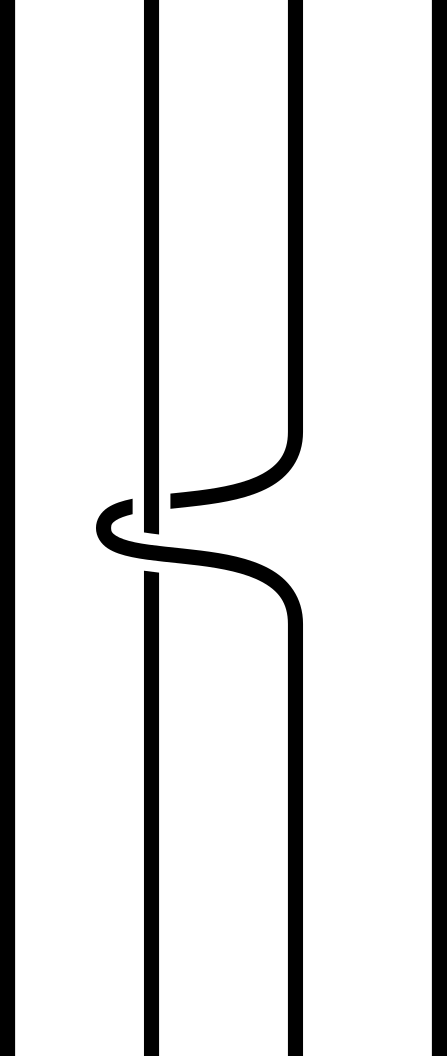}\\
\vspace{0.15cm}
\caption*{$\alpha_1=\sigma_2^2$}
\includegraphics[height=0.15\textheight]{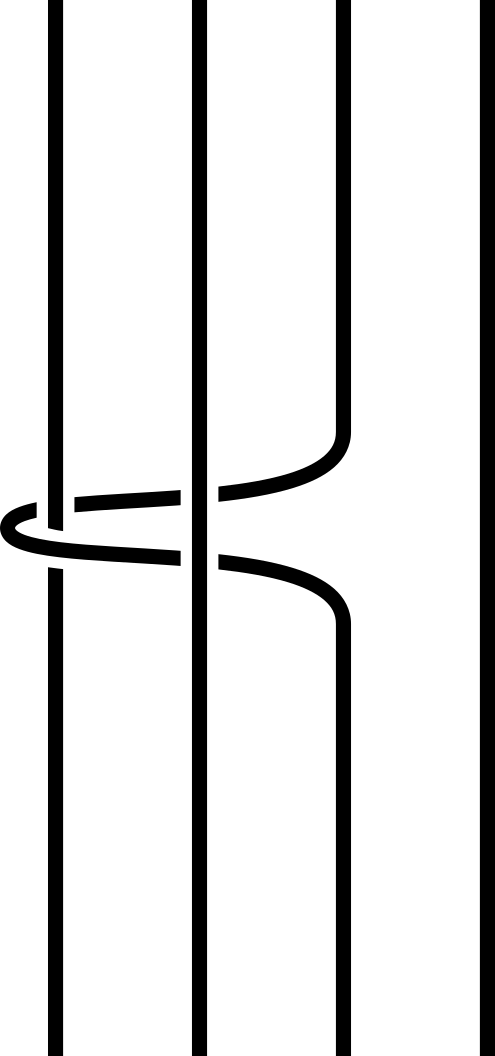}\\
\vspace{0.15cm}
\caption*{$\alpha_2=\sigma_2 \sigma_1^2 \sigma_2^{-1}$}
\end{multicols}
\end{figure}

%very annoying how latex thinks these .pngs have different widths.  should edit the .png metadata, but for now, this "height" instead of "width" specification works OK

%\phantom{.}

\begin{figure}[H]
\setlength\columnsep{-1cm}

\begin{multicols}{3}
\centering
\includegraphics[height=0.15\textheight]{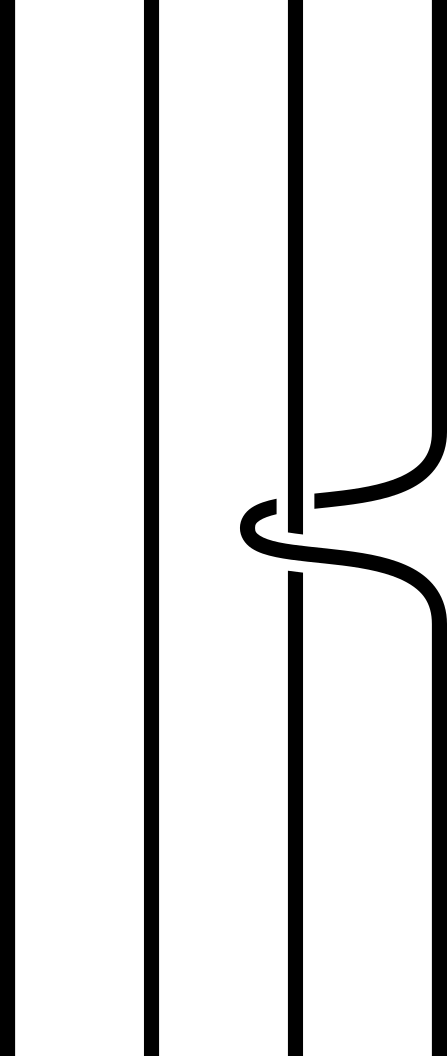}\\
\vspace{0.15cm}
\caption*{$\beta_1=\sigma_3^2$}
\includegraphics[height=0.15\textheight]{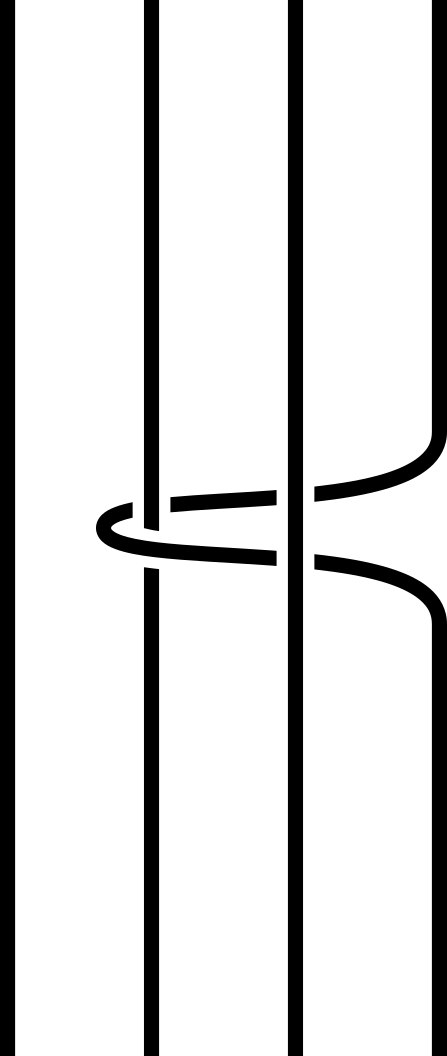}\\
\vspace{0.15cm}
\caption*{$\beta_2=\sigma_3 \sigma_2^2 \sigma_3^{-1}$}
\includegraphics[height=0.15\textheight]{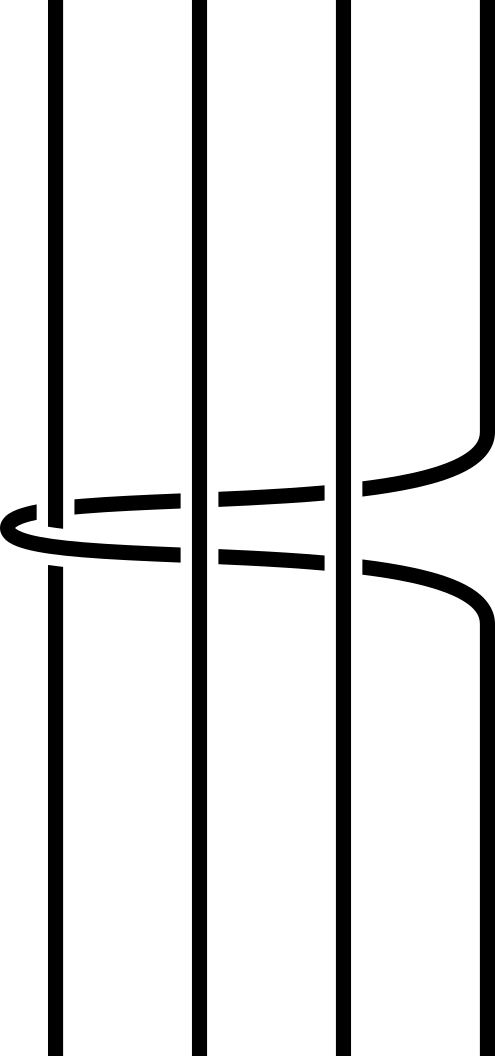}\\
\vspace{0.15cm}
\caption*{$\beta_3=\sigma_3 \sigma_2 \sigma_1^2 \sigma_2^{-1} \sigma_3^{-1}$}
\end{multicols}

\end{figure}

%\phantom{.}

\subsection{Relations for $F(\alpha_1, \alpha_2) \rtimes F(\sigma_1^2)$} \phantom{.} 

\begin{figure}[H]

\setlength\columnsep{-1cm}

\begin{multicols}{2}

\setlength\columnsep{0.1cm}

\begin{multicols}{2}
\centering
\includegraphics[height=0.15\textheight]{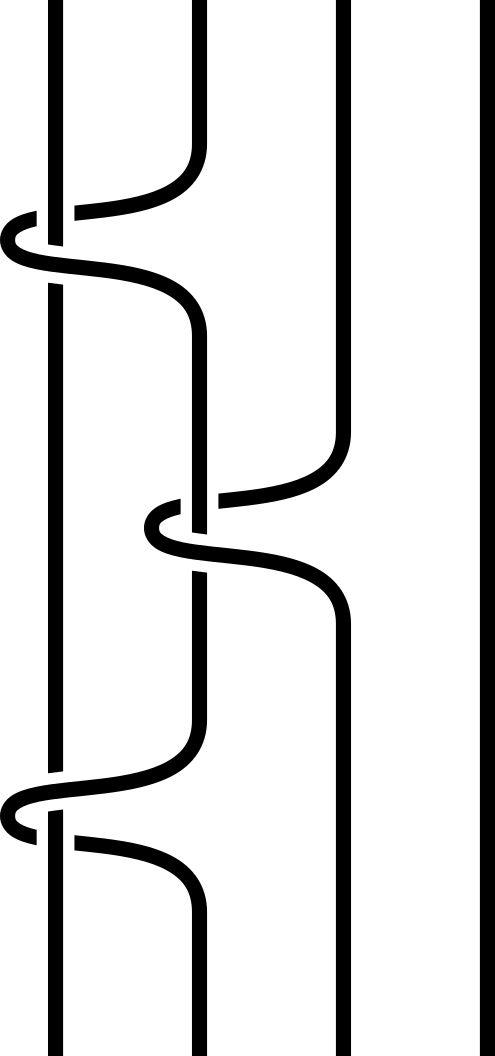}\\
%Test subfigure 1
\includegraphics[height=0.15\textheight]{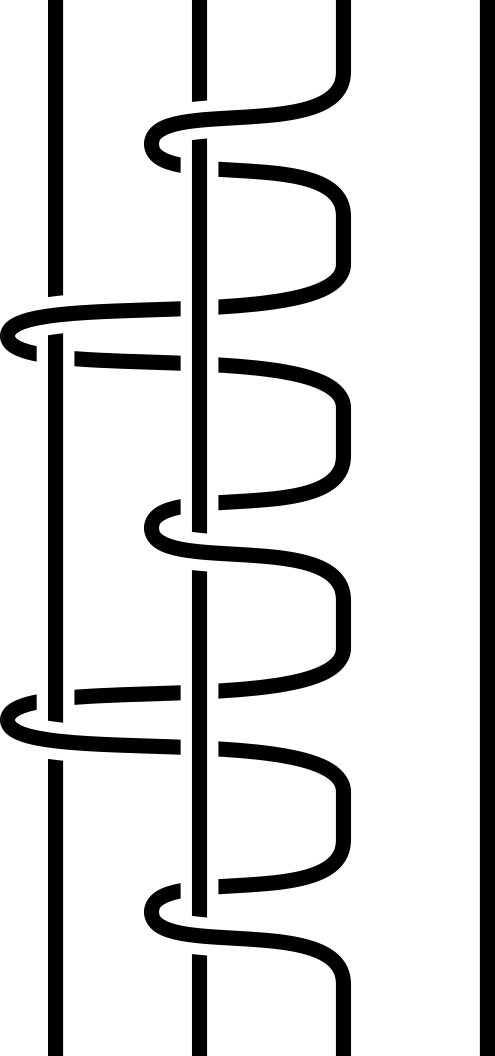}\\
%Test subfigure 2
\end{multicols}
\caption*{$\sigma_1^2 \alpha_1 \sigma_1^{-2}=(\alpha_2 \alpha_1)^{-1} \alpha_1 (\alpha_2 \alpha_1)$}

\setlength\columnsep{0.1cm}

\begin{multicols}{2}
\centering
\includegraphics[height=0.15\textheight]{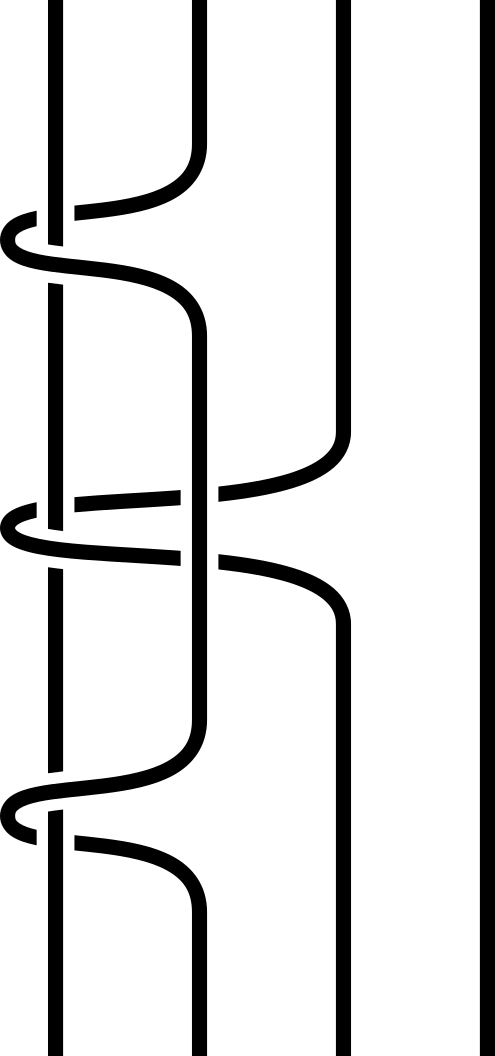}\\
%Test subfigure 1
\includegraphics[height=0.15\textheight]{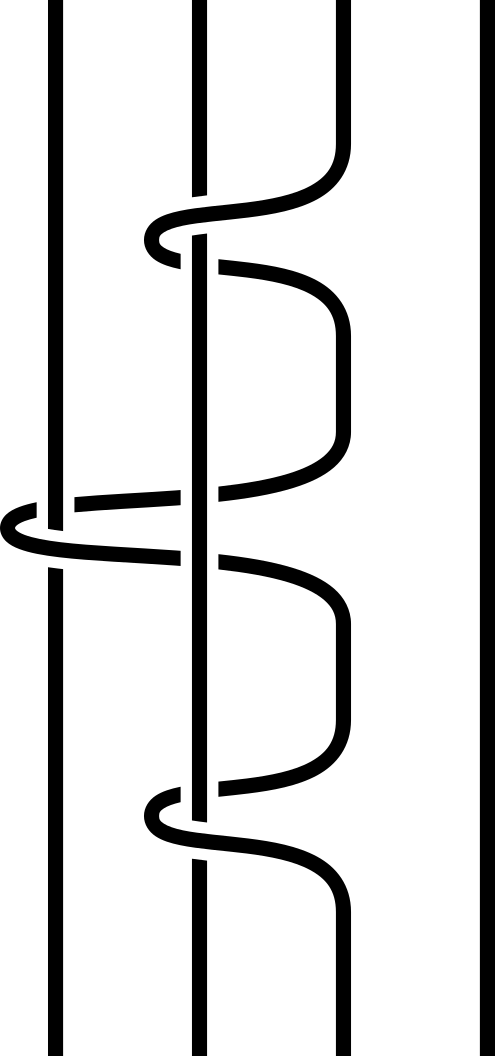}\\
%Test subfigure 2
\end{multicols}
\caption*{$\sigma_1^2 \alpha_2 \sigma_1^{-2}=\alpha_1^{-1} \alpha_2 \alpha_1$}

\end{multicols}

\end{figure}

\vspace{0.5cm}

\subsection{Relations for $F(\beta_1, \beta_2, \beta_3) \rtimes \left( F(\alpha_1, \alpha_2) \rtimes F(\sigma_1^2) \right)$} \phantom{.} 

\begin{figure}[H]

\begin{multicols}{2}
\setlength\columnsep{.1cm}

\begin{multicols}{2}
\centering
\includegraphics[height=0.15\textheight]{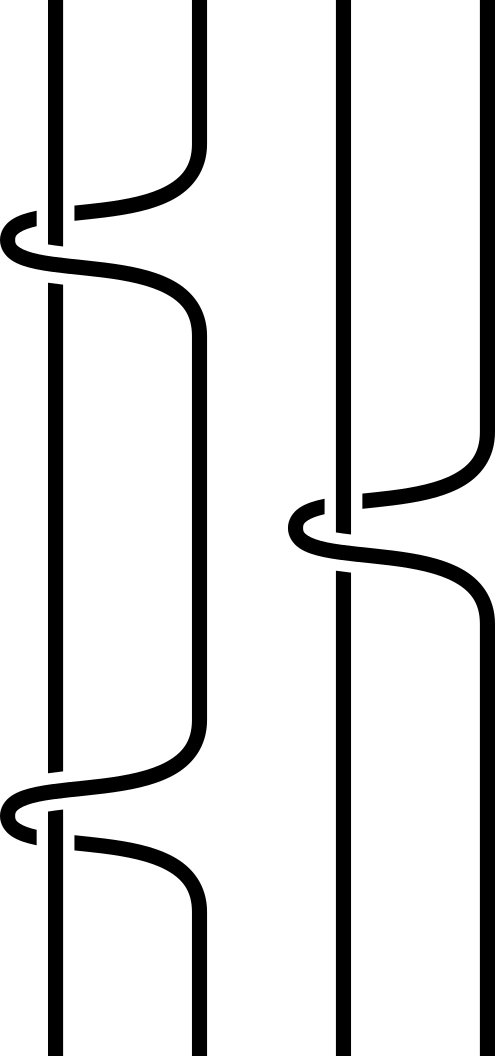}\\
%Test subfigure 1
\includegraphics[height=0.15\textheight]{b1.png}\\
%Test subfigure 2
\end{multicols}
\caption*{$\sigma_1^2 \beta_1 \sigma_1^{-2}=\beta_1$}

\begin{multicols}{2}
\centering
\includegraphics[height=0.15\textheight]{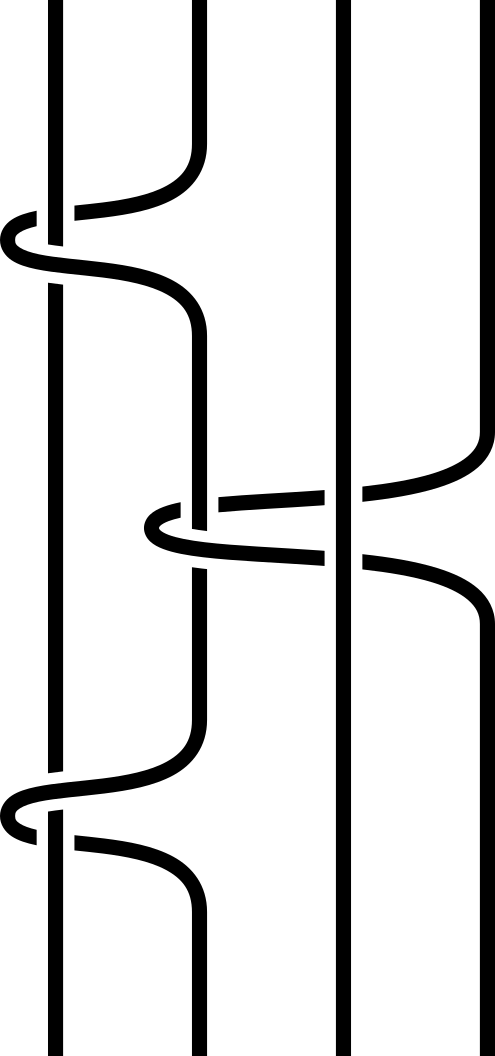}\\
%Test subfigure 1
\includegraphics[height=0.15\textheight]{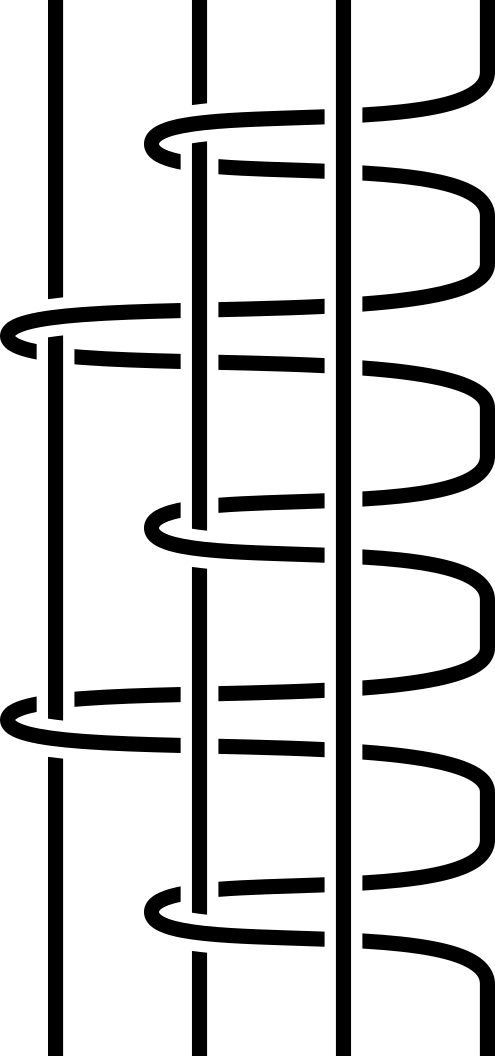}\\
%Test subfigure 2
\end{multicols}
\caption*{$\sigma_1^2 \beta_2 \sigma_1^{-2}=(\beta_3 \beta_2)^{-1} \beta_2 (\beta_3 \beta_2)$}

\end{multicols}

\end{figure}

\begin{figure}[H]

\begin{multicols}{1}
\setlength\columnsep{.01cm}

\begin{multicols}{2}
\centering
\includegraphics[height=0.15\textheight]{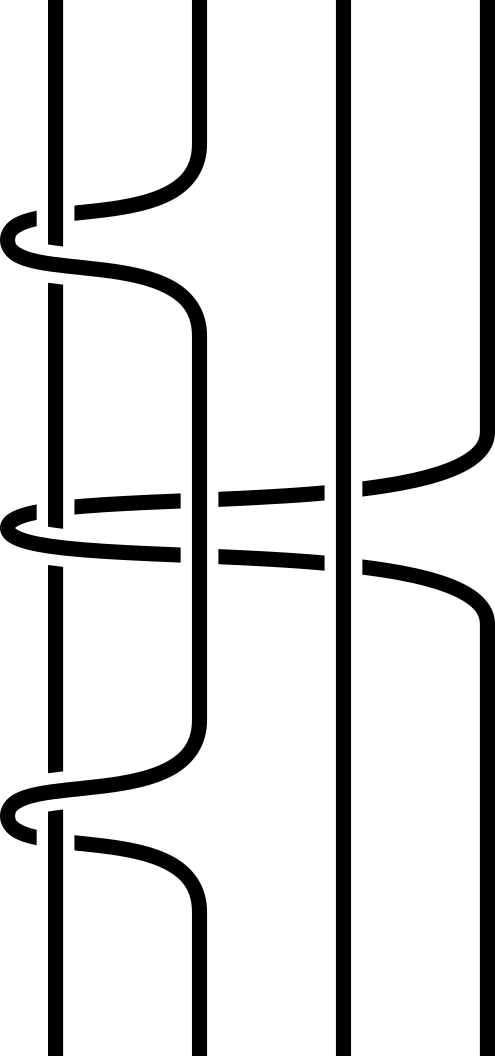}\\
%Test subfigure 1
\includegraphics[height=0.15\textheight]{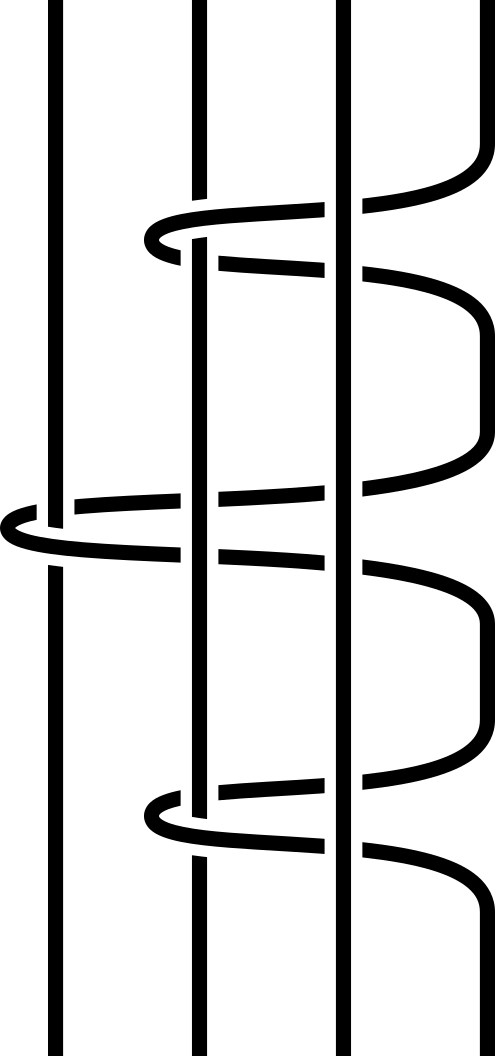}
\end{multicols}
\caption*{$\sigma_1^2 \beta_3 \sigma_1^{-2}=\beta_2^{-1} \beta_3 \beta_2$}

\begin{multicols}{2}
\centering
\includegraphics[height=0.15\textheight]{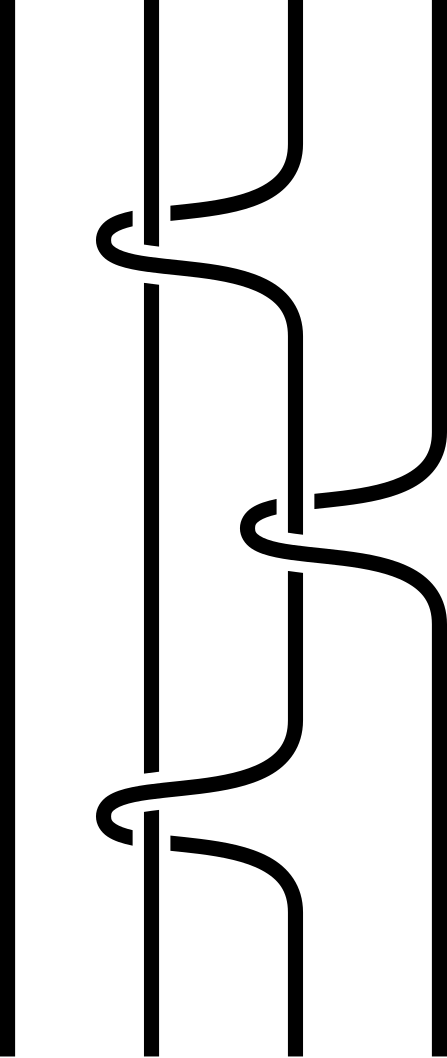}\\
%Test subfigure 1
\includegraphics[height=0.15\textheight]{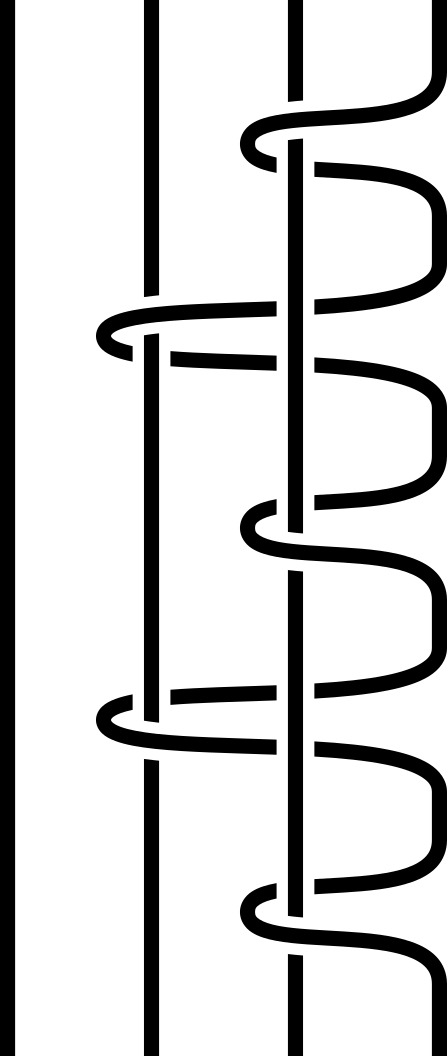}\\
%Test subfigure 2
\end{multicols}
\caption*{$\alpha_1 \beta_1 \alpha_1^{-1}=(\beta_2 \beta_1)^{-1} \beta_1 (\beta_2 \beta_1)$}

\end{multicols}

\end{figure}

\begin{figure}[H]

\begin{multicols}{2}

\setlength\columnsep{0.1cm}

\begin{multicols}{2}
\centering
\includegraphics[height=0.15\textheight]{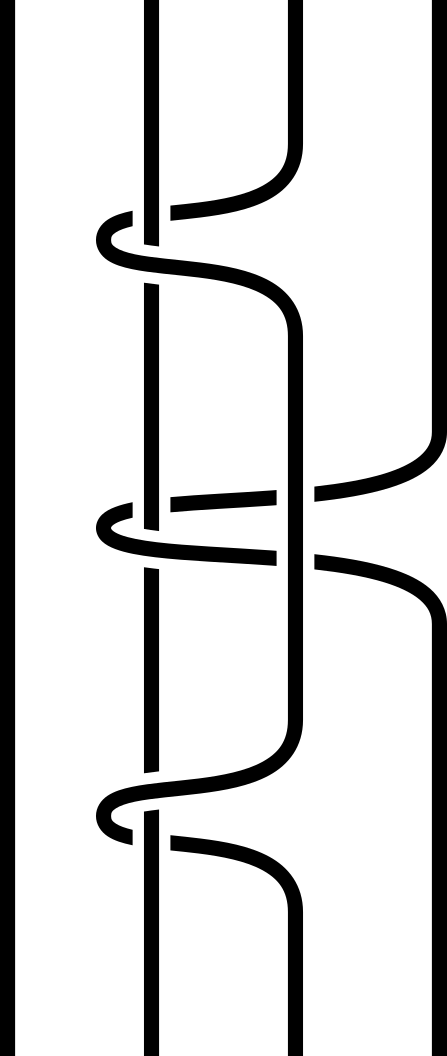}\\
%Test subfigure 1
\includegraphics[height=0.15\textheight]{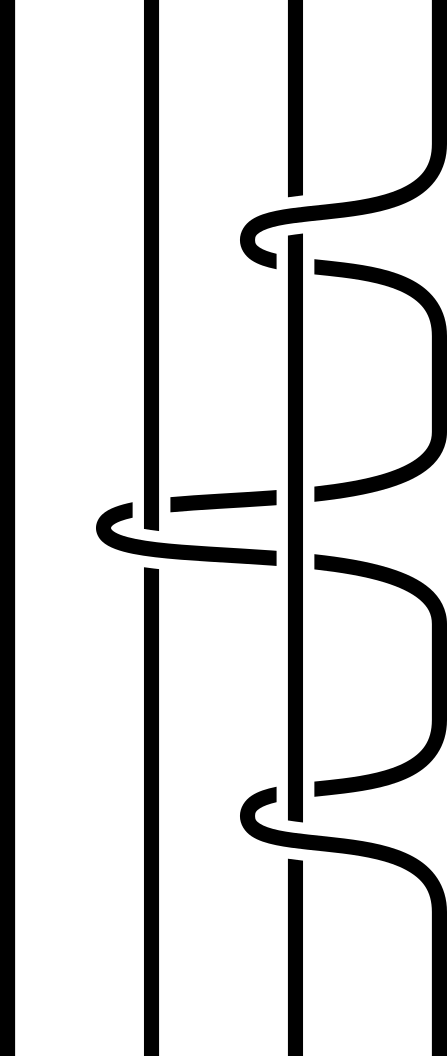}\\
%Test subfigure 2
\end{multicols}
\caption*{$\alpha_1 \beta_2 \alpha_1^{-1}=\beta_1^{-1} \beta_2 \beta_1$}

\begin{multicols}{2}
\centering
\includegraphics[height=0.15\textheight]{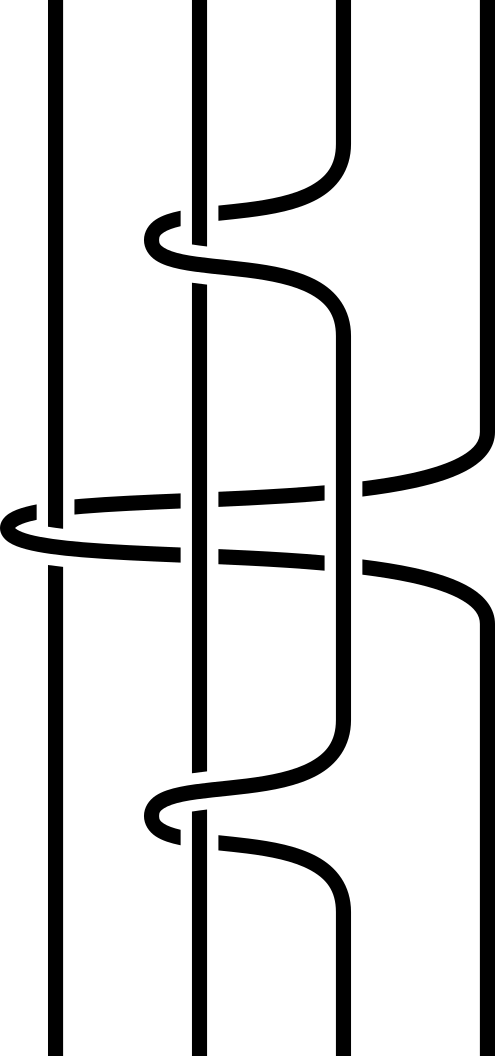}\\
%Test subfigure 1
\includegraphics[height=0.15\textheight]{b3.png}\\
%Test subfigure 2
\end{multicols}
\caption*{$\alpha_1 \beta_3 \alpha_1^{-1}=\beta_3$}

\end{multicols}

\end{figure}

\begin{figure}[H]

\begin{multicols}{2}
\setlength\columnsep{.1cm}

\begin{multicols}{2}
\centering
\includegraphics[height=0.15\textheight]{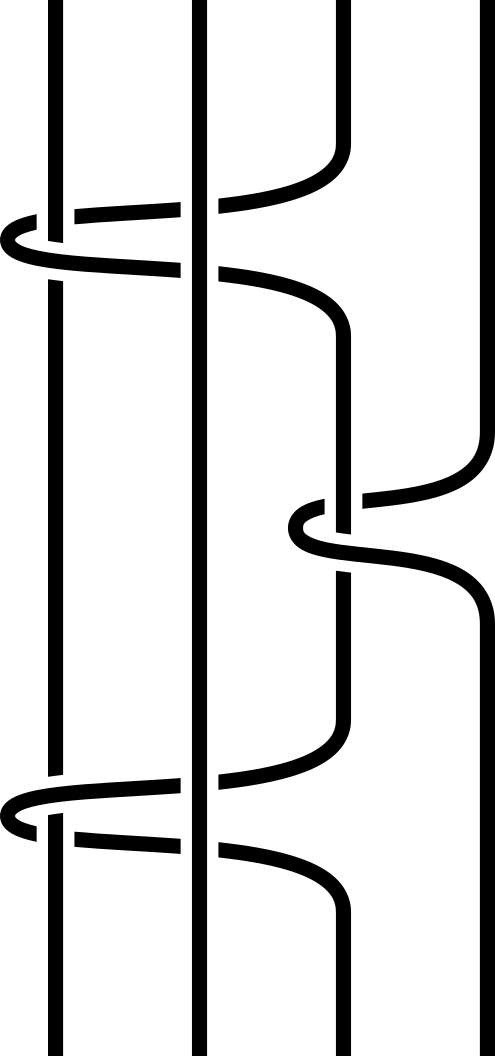}\\
%Test subfigure 1
\includegraphics[height=0.15\textheight]{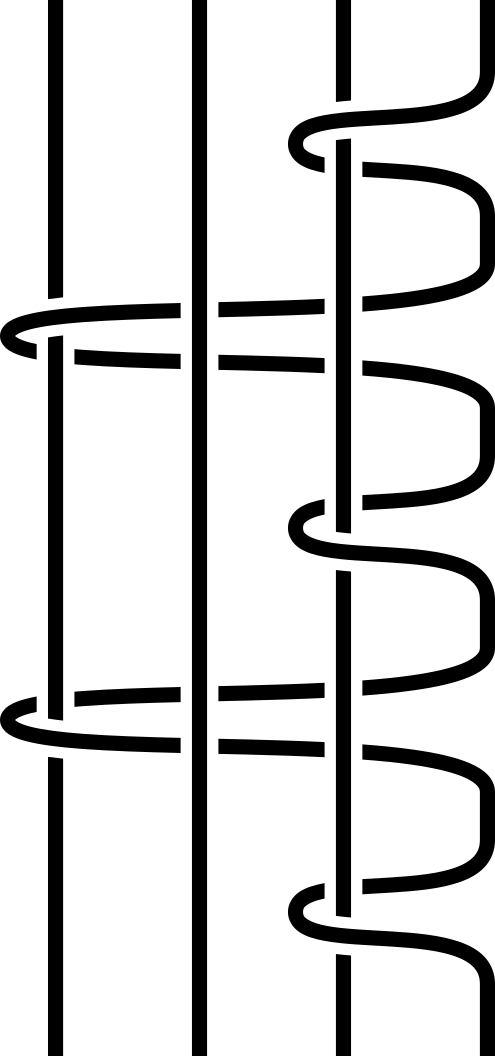}\\
%Test subfigure 2
\end{multicols}
\caption*{$\alpha_2 \beta_1 \alpha_2^{-1}=(\beta_3 \beta_1)^{-1} \beta_1 (\beta_3 \beta_1)$}

\begin{multicols}{2}
\centering
\includegraphics[height=0.15\textheight]{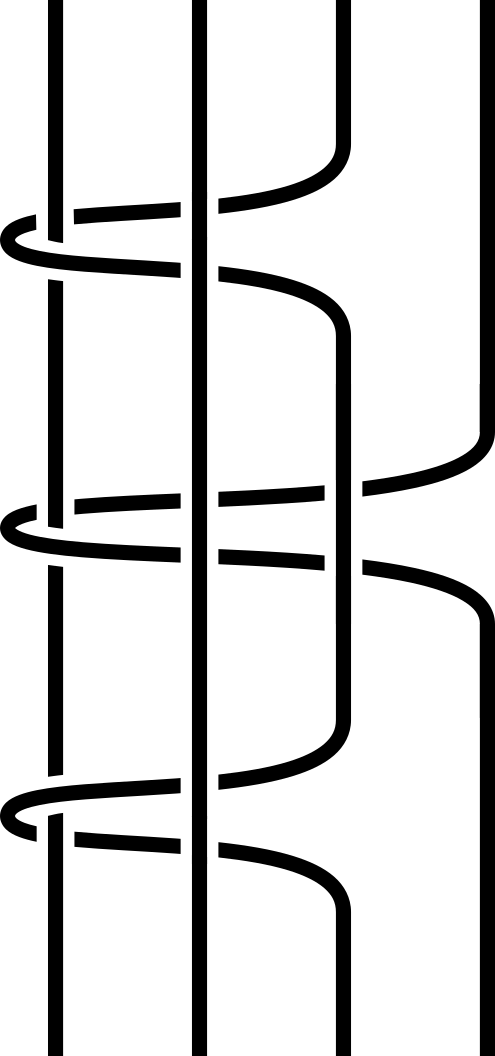}\\
%Test subfigure 1
\includegraphics[height=0.15\textheight]{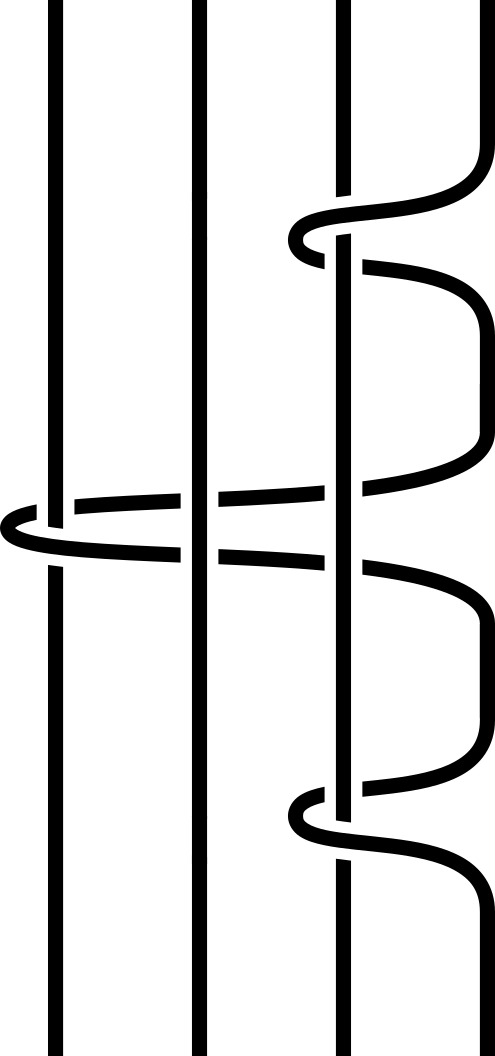}\\
%Test subfigure 2
\end{multicols}
\caption*{$\alpha_2 \beta_3 \alpha_2^{-1}=\beta_1^{-1} \beta_3 \beta_1$}

\end{multicols}

\end{figure}

\begin{figure}[H]

%\begin{multicols}{2}
%\setlength\columnsep{.1cm}

\setlength\columnsep{.01cm}

\begin{multicols}{2}
\setlength\columnsep{.01cm}
\centering
\includegraphics[height=0.15\textheight]{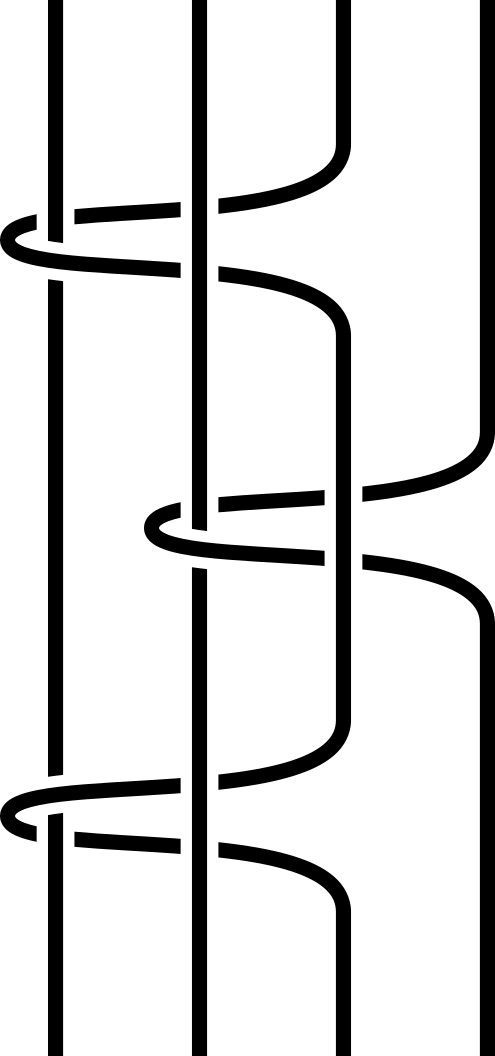}\\
%Test subfigure 1
\includegraphics[height=0.15\textheight]{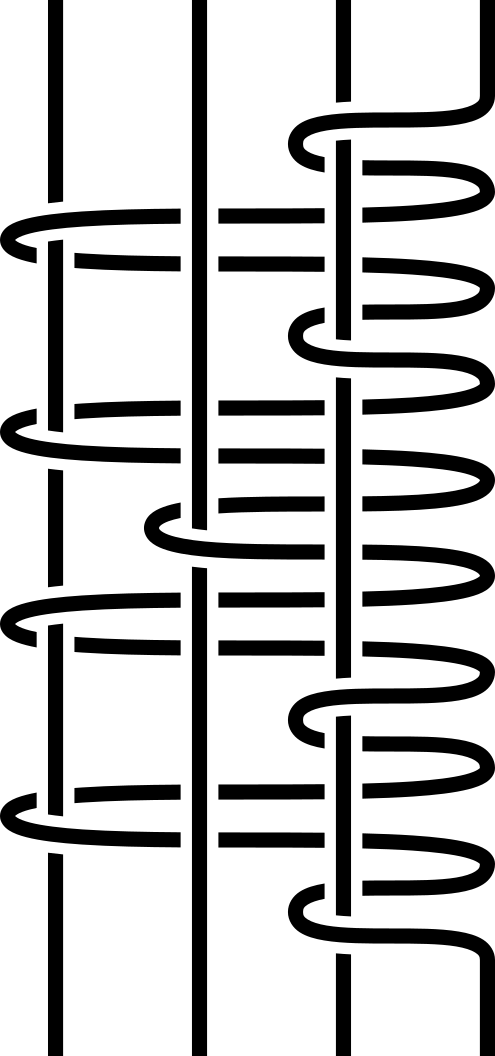}\\
%Test subfigure 2
\end{multicols}

\caption*{$\alpha_2 \beta_2 \alpha_2^{-1}=(\beta_3 \beta_1)^{-1} (\beta_1 \beta_3) \beta_2 (\beta_1 \beta_3)^{-1} (\beta_3 \beta_1)$}
%when you come back: switch this one with the long one above, then make the center

%\end{multicols}

\end{figure}

\vspace*{1cm}

\subsection{The center of $P_4$}\label{CenterP4} \phantom{.} 

\begin{figure}[H]

\begin{multicols}{2}
\centering
\includegraphics[height=0.15\textheight]{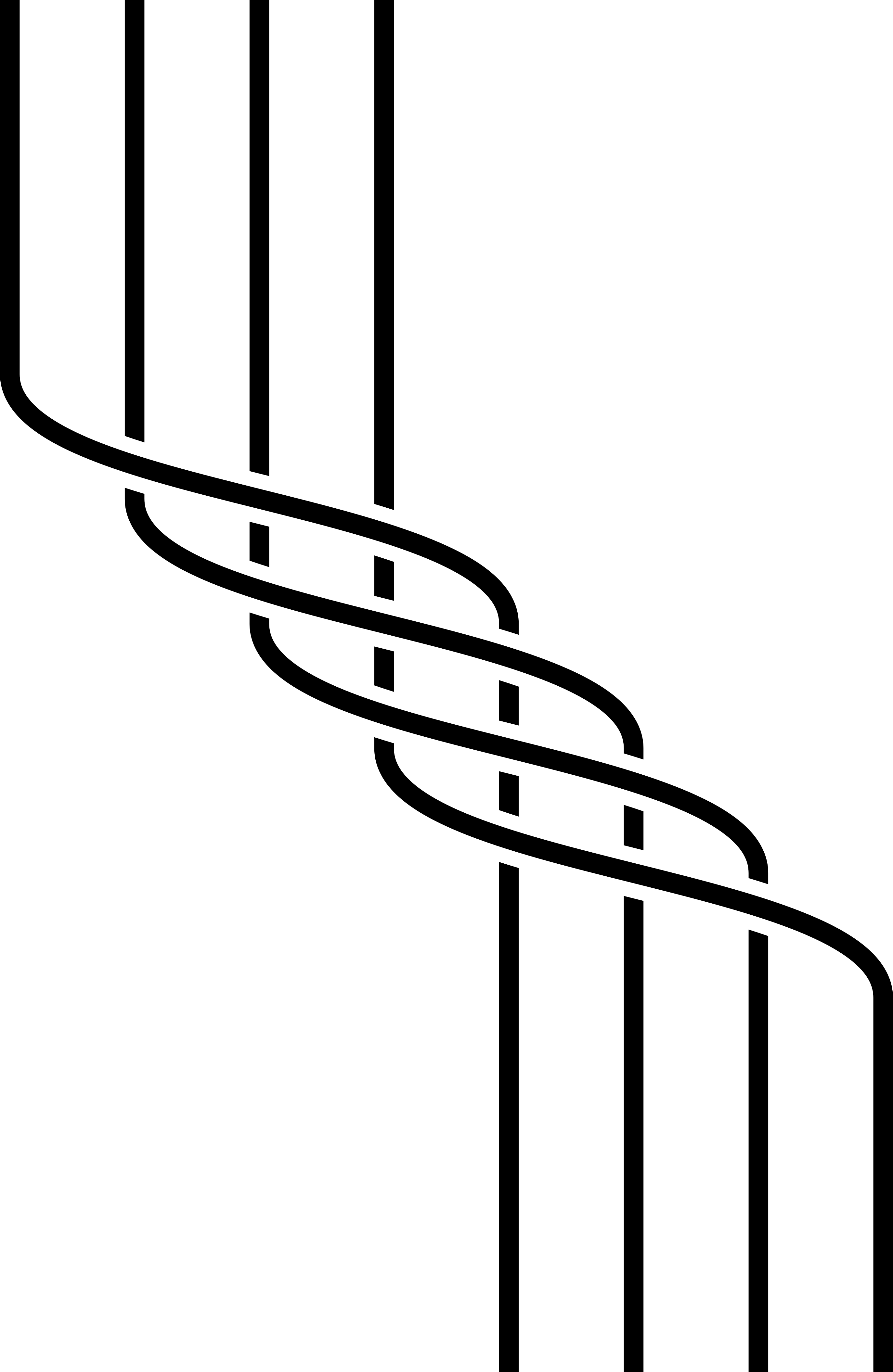}\\
%Test subfigure 1
\includegraphics[height=0.15\textheight]{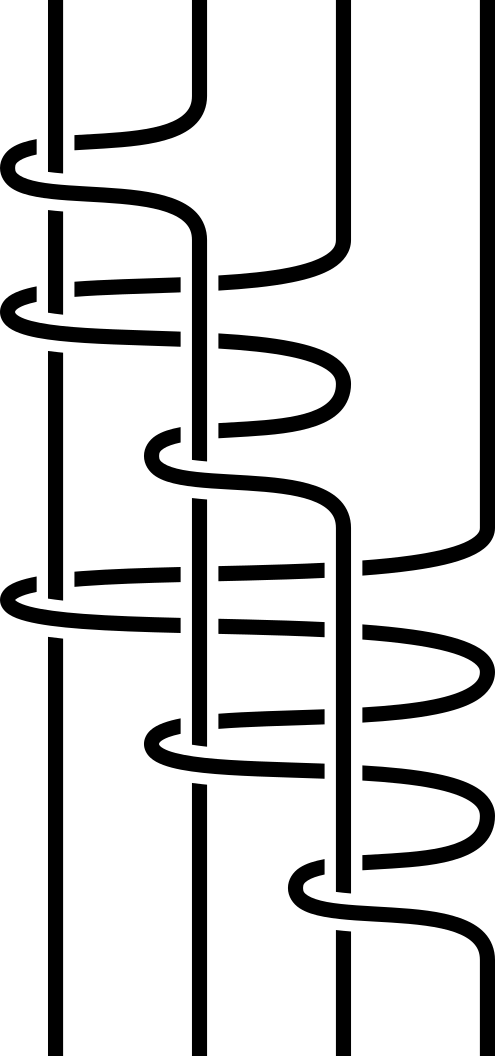}\\
%Test subfigure 2
\end{multicols}
\caption*{$\left( \sigma_1 \sigma_2 \sigma_3 \right)^4 =\sigma_1^2 \alpha_2 \alpha_1 \beta_3 \beta_2 \beta_1$}

\end{figure}

These diagrams show that the center splits off and gives us the direct product decomposition: 
$$
F(\beta_1, \beta_2, \beta_3) 
\rtimes \left( F(\alpha_1, \alpha_2) \rtimes \langle \sigma_1^2 \rangle \right)=
\left( F(\beta_1, \beta_2, \beta_3) \rtimes  F(\alpha_1, \alpha_2) \right) \times 
\langle \sigma_1^2 \alpha_2 \alpha_1 \beta_3 \beta_2 \beta_1 \rangle. $$

\bigskip

\bibliographystyle{alpha}
%\bibliography{ref}

\begin{thebibliography}{CS79}

\bibitem{ACC03}
{\scshape Adem, Alejandro and Cohen, Daniel and Cohen, Frederick R.} 
On representations and {$K$}-theory of the braid groups. 
{\em Math. Ann.} {\bf 326} (2003), 515--542. 
\mrev{1992276},
\zbl{1066.20042},
\arx{0110138},
\doi{10.1007/s00208-003-0435-8}.

\bibitem{Arnold}
{\scshape Arnol'd, Vladimir I.} 
The cohomology ring of the colored braid group, in ``Vladimir I. Arnold - Collected Works: Hydrodynamics, Bifurcation Theory, and Algebraic Geometry 1965-1972", (2014), Springer Berlin Heidelberg",183--186,
%\zbl{},
\doi{10.1007/978-3-642-31031-7\_18}
%could not find  mrev zbl, arx 







\bibitem{BBV}
{\scshape B\'{e}guin, C\'edric and Bettaieb, Hela and Valette, Alain.} 
{$K$}-theory for {$C^\ast$}-algebras of one-relator groups
{\em $K$-Theory} {\bf 16} (1999), 277--298. 
\mrev{1681180},
\zbl{0932.46063},
%\arx{},
\doi{10.1023/A:1007755408585}
%arxiv not found

\bibitem{BCH}
{\scshape Baum, Paul and  Connes, Alain and Higson, Nigel.}
Classifying space for proper actions and {$K$}-theory of group            {$C^*$}-algebras, in 
{\em $C^*$-algebras: 1943-1993} (San Antonio, TX, 1993), volume 167 of Contemp.
Math., Amer. Math. Soc., Providence, RI (1994), 240--291.
\mrev{1292018},
\zbl{0830.46061},
%\arx{},
\doi{10.1090/conm/167/1292018}
%arxiv not found

\bibitem{birm}
{\scshape Birman, Joan Sylvia.}
Braids, links, and mapping class groups
{\em Annals of Mathematics Studies, No. 82, Princeton University Press, Princeton, N.J.; University of
              Tokyo Press, Tokyo}, (1974) 
\mrev{0375281},
\zbl{0305.5701},
%\arx{},
\doi{10.1515/9781400881420}
%arxiv not found

\bibitem{CE01}
{\scshape Chabert, J\'er\^ome and Echterhoff, Siegfried.}
Permanence properties of the {B}aum-{C}onnes conjecture
{\em Doc. Math.} \textbf{6} (2001), 127--183
\mrev{1836047},
\zbl{0984.46047}
%\arx{},
%\doi{}
%arxiv and doi not found

\bibitem{RD-Chatterji}
{\scshape Chatterji, Indira.}
Introduction to the rapid decay property
{\em Contemp. Math.} \textbf{691} (2017), 53--72. Amer. Math. Soc., Providence, RI.
\mrev{3666050},
\zbl{1404.20032},
\arx{1604.06387},
\doi{10.1090/conm/691/13893}


\bibitem{MR3618901}
{\scshape Cuntz, Joachim and Echterhoff, Siegfried and Li, Xin and Yu, Guoliang.}
{$K$}-theory for group {$C^*$}-algebras and semigroup
              {$C^*$}-algebras,
{\em Oberwolfach Seminars, Birkh\"{a}user/Springer, Cham} \textbf{47} (2017), ix+319.
\mrev{3618901},
\zbl{1390.46001},
%\arx{},
\doi{10.1007/978-3-319-59915-1}
%arxiv not found

\bibitem{Cuntz82}
{\scshape Cuntz, Joachim.}
The {$K$}-groups for free products of {$C^{\ast}
              $}-algebras
{\em Operator algebras and applications, {P}art {I} ({K}ingston,
              {O}nt., 1980), Proc. Sympos. Pure Math.} \textbf{38} (1982), 81--84. Amer. Math. Soc., Providence, RI.
\mrev{0679696},
\zbl{0502.46050}
%\arx{},
%\doi{}
%arxiv and doi not found

\bibitem{Cuntz83}
{\scshape Cuntz, Joachim.}
{\,$K$}-theoretic amenability for discrete groups,
{\em J. Reine Angew. Math.} \textbf{344} (1983), 180--195. \mrev{0716254},
\zbl{0511.46066},
%\arx{},
\doi{10.1515/crll.1983.344.180}
%arxiv not found



\bibitem{Flores-Pooya-Valette}
   {\scshape Flores, Ramon and Pooya, Sanaz and Valette, Alain.}
   {$K$}-homology and {$K$}-theory for the lamplighter groups of
              finite groups,
{\em Proc. Lond. Math. Soc.},
  \textbf{115} (2017), 6, 1207--1226.
  \mrev{3741850}, 
  \zbl{1409.46043}
  \arx{1610.02798}
  \doi{10.1112/plms.12061}
  
  
  
\bibitem{GJV}
  {\scshape Gomez-Aparicio, Maria Paula and  Julg, Pierre and  Valette, Alain.}
   The Baum-Connes conjecture: an extended survey, in 
{\em Advances in Noncommutative Geometry: On the Occasion of Alain Connes' 70th Birthday},
Springer, 127--244,
  \mrev{4300553}, 
  \zbl{1447.58006}
  \arx{1905.10081}
  \doi{10.1007/978-3-030-29597-4\_3}

\bibitem{Hatcher_SSAT}
  {\scshape Hatcher, Allen.} {\em Spectral Sequences in Algebraic Topology}, available at
 \url{https://pi.math.cornell.edu/~hatcher/SSAT/SSch1.pdf}

\bibitem{Higson-Kasparov}
   {\scshape Higson, Nigel and Kasparov, Gennadi.}
    {$E$}-theory and {$KK$}-theory for groups which act properly
              and isometrically on {H}ilbert space,
   {\em Invent. Math.} {\bf 144} (2001) 1, 23--74,
\mrev{1821144}
\zbl{0988.19003}
\doi{10.1007/s002220000118}
%arxiv not found

\bibitem{Isley}
  {\scshape Isley, Oliver.} {$K$}-theory and {$K$}-homology for semi-direct products of $\mathbb{Z}_2$ by $\mathbb{Z}$, Th\`ese de doctorat, Universit\'e de Neuch\^atel (2011) 
\url{https://core.ac.uk/download/pdf/20657352.pdf}


\bibitem{JV} 
  {\scshape Julg, Pierre and Valette, Alain.} {$K$}-theoretic amenability for {${\rm SL}_{2}({\bf  Q}_{p})$}, and the action on the associated tree, {\em J. Funct. Anal.} 
{\bf 58} (1984){2}, 194--215.
 \mrev{757995}
 \zbl{0559.46030}
 \doi{10.1016/0022-1236(84)90039-9}
%arxiv not found


\bibitem{Lafforgue}
 {\scshape Lafforgue, Vincent.}
   {$K$}-th\'{e}orie bivariante pour les alg\`ebres de {B}anach et
              conjecture de {B}aum-{C}onnes, 
 {\em Invent. Math.} {\bf 149}, (2002) 1, 1--95,
  \mrev{1914617}
  \zbl{1084.19003}
  \doi{10.1007/s002220200213}
  %arxiv not found

\bibitem{MR4218683}
{\scshape  Li, Xin and Omland Tron and Spielberg, Jack.} $C^*$-algebras of right LCM one-relator monoids and Artin-Tits monoids of finite type. {\em Comm. Math. Phys.}, {\bf 381} (3) (2021) 1263--1308.
\mrev{MR4218683}
\zbl{1472.46056}
\doi{10.1007/s00220-020-03758-5}
\arx{1807.08288}

\bibitem{Matthey02} 
{\scshape Matthey, Michel.}
   Mapping the homology of a group to the {$K$}-theory of its
              {$C^*$}-algebra,
{\em llinois J. Math.}
 {\bf 46}, (2002) 3
953--977
    \mrev{1951251}
    \zbl{1021.19003}
 %   \doi{10.1215/ijm/1258130995}
  %arxiv not found


\bibitem{misval}
{\scshape Mislin, Guido and Valette, Alain.}
   {\em Proper group actions and the {B}aum-{C}onnes conjecture}, Advanced Courses in Mathematics. CRM Barcelona, Birkh\"{a}user Verlag, Basel
    (2003) viii+131
  \mrev{2027168}
  \zbl{1028.46001}
  \doi{10.1007/978-3-0348-8089-3}
   %arxiv not found


\bibitem{MR768305}
{\scshape Natsume, Toshikazu.}
 On $K_\ast(C^\ast({\rm SL}_2({\bf Z})))$,  {A}ppendix to  ``{$K$}-theory for certain group {$C^\ast$}-algebras'' by E. C. Lance, 
{\em J. Operator Theory}
{\bf 13},
     (1985) 1, 103--118,
    \mrev{768305}
    \zbl{0581.46056}
   % doi not found yet 
   %arxiv not found

\bibitem{Natsume}
{\scshape  Natsume, Toshikazu.}
  The {B}aum-{C}onnes conjecture, the commutator theorem, and
              {R}ieffel projections, 
{\em C. R. Math. Rep. Acad. Sci. Canada},
 {\bf 10} (1988) 1, 13--18
   \mrev{925294}
\zbl{0664.46069}
 % doi not found yet 
   %arxiv not found


\bibitem{OO2001-2}
{\scshape Oyono-Oyono, Herv\'e.} Baum-{C}onnes conjecture and extensions, {\em J. Reine Angew. Math.} {\bf 532},
  (2001) 133--149
   \mrev{1817505}
   \zbl{0973.46064}
   \doi{10.1515/crll.2001.020}
   

\bibitem{OO2001}
{\scshape Oyono-Oyono, Herv\'e.}  Baum-{C}onnes conjecture and group actions on trees, 
{\em $K$-Theory}, {\bf 24},
   (2001), {2} 115--134
  \mrev{1869625},
  \zbl{1008.19001},
  \doi{10.1023/A:1012786413219}
   
   
     
\bibitem{Pimsner86}
{\scshape Pimsner, Michael V.}
  {$KK$}-groups of crossed products by groups acting on trees, {\em Invent. Math.}
 {\bf 86}, (1986) 3, 603--634
   \mrev{860685},
   \zbl{0757.46059},
   \doi{10.1007/BF01389271}
     


\bibitem{Pooya19}
{\scshape Pooya, Sanaz.} {K-theory and {K}-homology of finite wreath products with free
              groups}, {\em Illinois J. Math.} {\bf 63}, (2019) 2.
  \mrev{3987499},
\zbl{1431.46053},
\doi{10.1215/00192082-7768735}


\bibitem{MR724030}
{\scshape Penington, M.G. and Plymen, Roger J.}.
The {D}irac operator and the principal series for complex
              semisimple {L}ie groups
{\em J. Funct. Anal.} \textbf{53} (1983), 269--286.
 \mrev{0724030},
\zbl{0542.22013},
%\arx{},
\doi{10.1016/0022-1236(83)90035-6}
%arxiv not found


\bibitem{PV80}
{\scshape Pimsner, Mihai V. and  Voiculescu, Dimitri.}
   Exact sequences for {$K$}-groups and {E}xt-groups of certain  cross-product {$C^{\ast} $}-algebras,
{\em J. Operator Theory} {\bf 4}, (1980) {1}, 93--118
    \mrev{0587369},
   % doi not found yet 
    \zbl{0474.46059}

\bibitem{PV82}
 {\scshape Pimsner, Mihai V. and  Voiculescu, Dimitri.} {$K$}-groups of reduced crossed products by free groups, {\em J. Operator Theory}
{\bf 8}, (1982) 1, {131--156}
\mrev{0587369}
\zbl{0533.46045}

\bibitem{Pooya-Valette}
{\scshape Pooya, Sanaz; Valette, Alain.}
K-theory for the {$C^*$}-algebras of the solvable
              {B}aumslag-{S}olitar groups
{\em Glasg. Math. J.} \textbf{60} (2018), 481--486.
 \mrev{3784059},
\zbl{1395.46054},
\arx{1604.05607},
\doi{10.1017/S0017089517000210}


\bibitem{Schick2007}
{\scshape Schick, Thomas.}
Finite group extensions and the {B}aum-{C}onnes conjecture
{\em Geom. Topol.} \textbf{11} (2007), 1767--1775.
 \mrev{2350467},
\zbl{1201.58019},
\arx{0209165},
\doi{10.2140/gt.2007.11.1767}


\bibitem{MR755672}
{\scshape Valette, Alain.}
{$K$}-theory for the reduced {$C^{\ast} $}-algebra of a
              semisimple {L}ie group with real rank {$1$} and finite centre {\em Quart. J. Math. Oxford Ser. (2)} \textbf{35} (1984), 341--359.
 \mrev{0755672},
\zbl{0545.22006},
%\arx{},
\doi{10.1093/qmath/35.3.341}

\bibitem{MR799592}
{\scshape Valette, Alain.}
Dirac induction for semisimple {L}ie groups having one
              conjugacy class of {C}artan subgroups, Operator algebras and their connections with topology and
              ergodic theory ({B}u\c{s}teni, 1983)
               {\em Lecture Notes in Math.} \textbf{1132} (1985), 526--555.
 \mrev{0799592},
\zbl{0569.22009},
%\arx{},
\doi{10.1007/BFb0074908}
%arxiv not found

\bibitem{valbc}
{\scshape Valette, Alain.}
Introduction to the {B}aum-{C}onnes conjecture, From notes taken by Indira Chatterji,
              With an appendix by Guido Mislin, Lectures in Mathematics ETH Z\"{u}rich
               {\em Birkh\"{a}user Verlag, Basel} \textbf{} (2002).
 \mrev{1907596},
%\zbl{1136.58013},
%\arx{},
\doi{10.1007/978-3-0348-8187-6}
%arxiv not found

\bibitem{MR1634470}
{\scshape Vershinin, Vladimir V.}
Homology of braid groups and their generalizations, Knot theory ({W}arsaw, 1995), Banach Center Publ.
               {\em Banach Center Publ. Polish Acad. Sci. Inst. Math., Warsaw} \textbf{42} (1998), 421--446.
 \mrev{1634470},
\zbl{0905.20032},
%\arx{},
%\doi{}
%arxiv and doi not found

\bibitem{Wall}
{\scshape Wall, Charles T. C.}
%Charles Terence Clegg
Poincar\'{e} complexes. {I}
               {\em Ann. of Math. (2)} \textbf{86} (1967), 213--245.
 \mrev{0217791},
%\zbl{},
%\arx{},
\doi{10.2307/197068}
%arxiv, zbl not found

\bibitem{Wassermann}
{\scshape Wassermann, Antony.}
Une d\'{e}monstration de la conjecture de {C}onnes-{K}asparov pour
              les groupes de {L}ie lin\'{e}aires connexes r\'{e}ductifs
               {\em C. R. Acad. Sci. Paris S\'{e}r. I Math} \textbf{304} (1987), 559--562.
 \mrev{0894996},
\zbl{0615.22011}
%\arx{},
%\doi{}
%arxiv, doi not found

\bibitem{williams}
{\scshape Williams, Dana P.}
Crossed Products of C$^*$-Algebras, Mathematical Surveys and Monographs
               {\em AMS} \textbf{} (2007).
 \mrev{2288954},
%\zbl{},
%\arx{},
\doi{10.1007/978-3-0348-8089-3}
%arxiv, zbl not found

\bibitem{Wilson}
  {\scshape Wilson, Jenny.} {\em The geometry and topology of braid groups}, available at
 \url{www.math.lsa.umich.edu/\~jchw/RTG-Braids.pdf}, RTG Geometry--Topology Summer School
University of Chicago
 % \mrev{},
%\zbl{},
%\arx{},
%\doi{}





\end{thebibliography}

%\newpage

%
% \noindent \textsc{Sara Azzali, 
% Universit\`a degli Studi di Bari, 
% Dipartimento di Matematica,
% Via E. Orabona 4, 70125 Bari, Italy},
% \texttt{sara.azzali@uniba.it}

% %
% \bigskip
% \smallskip
% \noindent \textsc{Sarah L. Browne, 
% The University of Kansas, Department of Mathematics, 1460 Jayhawk Blvd, Lawrence, KS, 66045.} \texttt{slbrowne@ku.edu}
% \bigskip

% \noindent \textsc{Maria Paula Gomez Aparicio, Universit\'e Paris-Saclay, CNRS, Laboratoire de math\'ematiques d'Orsay, 91405, Orsay, France},\\
% \texttt{maria-paula.gomez-aparicio@universite-paris-saclay.fr}
% \bigskip
% \smallskip

% \noindent \textsc{Lauren C. Ruth, Mercy College, 555 Broadway, Dobbs Ferry, NY 10522, USA}
% \texttt{LRuth@mercy.edu}
% \bigskip
% \smallskip
% %

% \noindent \textsc{Hang Wang, School of Mathematical Sciences, East China Normal University, Shanghai, China 200241}
% \texttt{wanghang@math.ecnu.edu.cn}
% \bigskip
% \smallskip

\end{document}